\documentclass[12pt,reqno]{amsart}
\usepackage[margin=1in]{geometry}
\usepackage{amsmath,amssymb,amsthm,graphicx,amsxtra, setspace}
\usepackage[utf8]{inputenc}
\usepackage{mathrsfs}
\usepackage{hyperref}
\usepackage{upgreek}
\usepackage{mathtools}
\allowdisplaybreaks

\usepackage[pagewise]{lineno}

\newtheorem{theorem}{Theorem}[section]
\newtheorem{lemma}[theorem]{Lemma}
\newtheorem{proposition}[theorem]{Proposition}

\newtheorem{definition}[theorem]{Definition}
\newtheorem{example}[theorem]{Example}
\newtheorem{remark}[theorem]{Remark}
\newtheorem{hypothesis}[theorem]{Hypothesis}

\let\emptyset\varnothing

\let\originalleft\left
\let\originalright\right
\renewcommand{\left}{\mathopen{}\mathclose\bgroup\originalleft}
\renewcommand{\right}{\aftergroup\egroup\originalright}

\newcommand{\Tr}{\mathop{\mathrm{Tr}}}
\renewcommand{\d}{\/\mathrm{d}\/}

\def\w{\textbf{W}^{\varepsilon}_{{\theta}^{\varepsilon}}}

\def\e{\varepsilon}

\def\t{t\wedge\tau_N}
\def\tt{t\wedge\tau_N^{n,\e}}
\def\S{\mathrm{S}}
\def\T{T\wedge\tau_N}
\def\TT{T\wedge\tau_N^{n,\e}}
\def\L{\mathbb{L}}
\def\A{\mathrm{A}}
\def\I{\mathrm{I}}
\def\F{\mathrm{F}}
\def\C{\mathrm{C}}
\def\f{\mathbf{f}}
\def\J{\mathrm{J}}
\def\B{\mathrm{B}}
\def\D{\mathrm{D}}
\def\y{\mathbf{y}}

\def\Z{\mathrm{Z}}
\def\E{\mathbb{E}}
\def\X{\mathbb{X}}
\def\x{\mathbf{x}}

\def\z{\mathbf{z}}
\def\v{\mathbf{v}}
\def\V{\mathbb{v}}
\def\w{\mathbf{w}}
\def\W{\mathrm{W}}
\def\G{\mathrm{G}}
\def\Q{\mathrm{Q}}
\def\M{\mathrm{M}}

\def\no{\nonumber}
\def\V{\mathbb{V}}
\def\wi{\widetilde}
\def\Q{\mathrm{Q}}
\def\u{\mathrm{U}}
\def\P{\mathrm{P}}
\def\u{\mathbf{u}}
\def\H{\mathbb{H}}

\newcommand{\R}{\mathbb{R}}

\renewcommand{\d}{\/\mathrm{d}\/}

\newcommand{\Addresses}{{
		\footnote{
			
			\noindent \textsuperscript{1}Department of Mathematics, Indian Institute of Technology Roorkee-IIT Roorkee,
			Haridwar Highway, Roorkee, Uttarakhand 247667, INDIA.\par\nopagebreak
			\noindent  \textit{e-mail:} \texttt{manilfma@iitr.ac.in, maniltmohan@gmail.com.}
			
			\noindent \textsuperscript{*}Corresponding author.

			\textit{Key words:} convective Brinkman-Forchheimer equations, strong solution, large deviation principle, weak convergence. 
			
			Mathematics Subject Classification (2010): Primary 60H15; Secondary 35R60, 35Q30, 76D05, 37L55.

}}}

\begin{document}
	
	
	\title[LDP  for stochastic convective Brinkman-Forchheimer equations]{Wentzell-Freidlin Large Deviation Principle for the stochastic convective Brinkman-Forchheimer equations
			\Addresses}
	\author[M. T. Mohan ]{Manil T. Mohan\textsuperscript{1*}}

	\maketitle
	
	\begin{abstract}
This work addresses some asymptotic behavior of solutions to the stochastic convective Brinkman-Forchheimer (SCBF)  equations perturbed by multiplicative Gaussian noise  in  bounded domains.   Using a weak convergence approach of Budhiraja and Dupuis, we establish the Laplace principle for the strong solution to the SCBF equations in a suitable Polish space. Then, the Wentzell-Freidlin  large deviation principle is derived using the well known results of Varadhan and Bryc. The large deviations for short time are also considered in this work. Furthermore, we study the exponential estimates on certain exit times associated with the solution  trajectory of the  SCBF equations. Using contraction principle, we study these exponential estimates of  exit times from the frame of reference of Freidlin-Wentzell type large deviations principle. This work also improves several LDP results available in the literature for the tamed Navier-Stokes equations as well as Navier-Stokes equations with damping in bounded domains. 
	\end{abstract}

	\section{Introduction}\label{sec1}\setcounter{equation}{0}
	In this paper, we consider the convective Brinkman-Forchheimer (CBF)  equations subject to external random forcing and study some asymptotic behavior of its solution. The CBF  equations in a bounded domain $\mathcal{O}\subset\R^n$ ($n=2,3$) with a smooth boundary $\partial\mathcal{O}$ are given by
\begin{equation}\label{1}
\left\{
	\begin{aligned}
	\frac{\partial \u}{\partial t}-\mu \Delta\u+(\u\cdot\nabla)\u+\alpha\u+\beta|\u|^{r-1}\u+\nabla p&=\mathbf{f}, \ \text{ in } \ \mathcal{O}\times(0,T), \\ \nabla\cdot\u&=0, \ \text{ in } \ \mathcal{O}\times(0,T), \\
	\u&=\mathbf{0}\ \text{ on } \ \partial\mathcal{O}\times(0,T), \\
	\u(0)&=\u_0 \ \text{ in } \ \mathcal{O},\\
	\int_{\mathcal{O}}p(x,t)\d x&=0, \ \text{ in } \ (0,T).
	\end{aligned}
	\right.
	\end{equation}
Physically, the convective Brinkman-Forchheimer equations \eqref{1} describe the motion of incompressible fluid flows in a saturated porous medium. One can also consider these equations as a modification (by an absorption term $\alpha\u+\beta|\u|^{r-1}\u$) of the classical  Navier-Stokes equations. Here $\u(t , x) \in \R^n$ represents the velocity field at time $t$ and position $x$, $p(t,x)\in\R$ denotes the pressure field, $\f(t,x)\in\R^n$ is an external forcing. The final condition in \eqref{1} is imposed for the uniqueness of the pressure $p$. The constant $\mu$ represents the positive Brinkman coefficient (effective viscosity), the positive constants $\alpha$ and $\beta$ represent the Darcy (permeability of porous medium) and Forchheimer (proportional to the porosity of the material) coefficients, respectively. The absorption exponent $r\in[1,\infty)$ and  $r=3$ is known as the critical exponent. In the deterministic setting, the global solvability results of the system \eqref{1} are known in the literature and interested readers are referred to see \cite{SNA,CLF,KT2,MTM7}, etc.

The works  \cite{ZBGD,WLMR,WL,MRXZ1}, etc discuss  the global solvability and asymptotic behavior  of solutions to the stochastic counterpart of the \eqref{1} and related models in the whole space or on a torus. These papers showed the existence of a unique strong solution 
\begin{align}\label{2}
\u\in\mathrm{L}^2(\Omega;\mathrm{L}^{\infty}(0,T;\H^1(\mathcal{O})))\cap\mathrm{L}^2(0,T;\H^2(\mathcal{O})),
\end{align} with $\mathbb{P}$-a.s., paths in $\C([0,T];\H^1(\mathcal{O})),$  for $\u_0\in\mathrm{L}^2(\Omega;\H^1(\mathcal{O}))$. In bounded domains, there is a technical difficulty in getting the strong solutions to the SCBF equations \eqref{1} with the regularity given in \eqref{2}.   The major difficulty in working with bounded domains is that $\mathrm{P}_{\H}(|\u|^{r-1}\u)$ ($\mathrm{P}_{\H}:\L^2(\mathcal{O})\to\H$ is the Helmholtz-Hodge orthogonal projection) need not be zero on the boundary, and $\mathrm{P}_{\H}$ and $-\Delta$ are not necessarily commuting (see Example 2.19, \cite{JCR4}). This implies that  the equality 
\begin{align}\label{3}
&\int_{\mathcal{O}}(-\Delta\u(x))\cdot|\u(x)|^{r-1}\u(x)\d x\nonumber\\&=\int_{\mathcal{O}}|\nabla\u(x)|^2|\u(x)|^{r-1}\d x+\frac{r-1}{4}\int_{\mathcal{O}}|\u(x)|^{r-3}|\nabla|\u(x)|^2|^2\d x,
\end{align}
may not be useful in the context of bounded domains. The authors in \cite{ZDRZ,LHGH,LHGH1,LHGH2,MRTZ,BYo}, etc considered stochastic 3D tamed Navier-Stokes equations and related models on bounded domains with Dirichlet boundary conditions and addressed various problems like global solvability, existence of random attractors, existence and uniqueness of invariant measures, stability, etc. As far as the strong solutions are concerned, some of these works proved regularity results in the space provided in \eqref{2}, by using the estimate given in \eqref{3}, which may not hold true always. Recently, the author in the work \cite{MTM8} considered the SCBF equations perturbed by multiplicative Gaussian noise and showed the existence and uniqueness of strong solutions  in a larger space than \eqref{2} and discussed about some asymptotic behavior of strong solutions. The existence and uniqueness of strong solutions to SCBF equations ($r> 3,$ for any $\mu$ and $\beta$, $r=3$ for $2\beta\mu\geq 1$) in bounded domains with $\u_0\in\mathrm{L}^2(\Omega;\L^2(\mathcal{O})))$ is obtained in the space
 \begin{align}\label{1.4}\mathrm{L}^2(\Omega;\mathrm{L}^{\infty}(0,T;\L^2(\mathcal{O}))\cap\mathrm{L}^2(0,T;\H_0^1(\mathcal{O})))\cap\mathrm{L}^{r+1}(\Omega;\mathrm{L}^{r+1}(0,T;\L^{r+1}(\mathcal{O}))),\end{align}with $\mathbb{P}$-a.s. paths in $\C([0,T];\L^2(\mathcal{O}))$.  The monotonicity as well as the demicontinuity properties of the linear and nonlinear operators and a stochastic generalization of the Minty-Browder technique are exploited the proofs. The energy equality (It\^o's formula) satisfied by the SCBF equations is established by approximating the strong solution using the finite-dimensional space spanned by the first $n$ eigenfunctions of the Stokes operator, which can approximate functions defined on smooth bounded domains  in such a way that the approximations are bounded and converge in both Sobolev and Lebesgue spaces simultaneously. The stability results as well as existence and uniqueness of invariant measures are also discussed in the work \cite{MTM8}. The well-posedness and asymptotic behavior of strong solutions to  the SCBF equations perturbed by pure jump noise is considered in the work \cite{MTM10}.

The large deviations theory is one of the classical areas in probability theory with many deep developments and variety of applications. The theory of large deviations deals with the probabilities of rare events that are exponentially small as a function of some parameter. In the case of stochastic differential equations, this parameter can be considered as  the amplitude of the noise perturbing a dynamical system.  The works \cite{BD1,chow,DaZ,KX}, etc developed the Wentzell-Freidlin type large deviation estimates for a class of infinite dimensional stochastic differential equations. Large deviation principles for the 2D stochastic Navier-Stokes equations driven by Gaussian noise have been established in \cite{ICAM,SSSP}, etc. A large deviation principle of Freidlin-Wentzell type for the stochastic tamed 3D Navier-Stokes equations driven by multiplicative noise  in the whole space or on a torus  is established in \cite{MRTZ1}. Small time large deviations principles  for the stochastic 3D tamed Navier-Stokes equations in bounded domains is established in the work \cite{MRTZ}.  Large deviation principle for  3D tamed Navier-Stokes equations driven by multiplicative L\'evy noise in periodic domains is established in \cite{HGHL}.  The authors in the work \cite{LHGH1} obtained a small time large deviation principle for the stochastic 3D Navier-Stokes equation with damping in bounded domains. But the authors used the estimate \eqref{3} to obtain the existence of a global strong solution and they exploited the regularity given in \eqref{2} to obtain a small time LDP. Due to the technical difficulty discussed above in the case of bounded domains, it appears to us that the LDP results obtained in the work \cite{LHGH1} are true only in periodic domains.    In  this work, we establish the Wentzell-Freidlin (see \cite{FW})  large deviation principle for the SCBF equations using the well known results of Varadhan and Bryc (see \cite{DZ,Va}). We also show the LDP of the strong solutions to the SCBF equations for short time (with the regularity given in \eqref{1.4} for the velocity field), which in the finite dimensional case is the celebrated Varadhan’s large deviation estimate. Furthermore, for the SCBF equations perturbed by additive Gaussian noise, we derive an exponential inequality for the energy of the solution trajectory and examine exit times of solutions from the $R$-ball by using small noise asymptotic granted  by large deviations theory.

The organization of the paper is as follows. In the next section, we define the linear and nonlinear operators, and provide the necessary function spaces needed to obtain the global solvability results of the system \eqref{1}.  In section \ref{sec3}, we formulate the SCBF equations perturbed by Gaussian noise and  discuss about the existence and uniqueness of global strong solutions. The Wentzell-Freidlin type large deviation principle for the SCBF equations using the well known results of Varadhan-Bryc and Budhiraja-Dupuis  is established in section \ref{sec4} (Theorems \ref{thm4.14}, \ref{compact} and \ref{weak}). The  LDP for the strong solutions to the SCBF equations for short time is studied in the section \ref{sec5} (Theorem \ref{thm5.2}). In section \ref{sec6}, we consider  the SCBF equations perturbed by additive Gaussian noise and  derive an exponential inequality for the energy of the solution trajectory (Proposition \ref{prop3.1}). In the final section, we examine exit times of solutions of the SCBF equations from the $R$-ball by using small noise asymptotic granted  by large deviations theory (Theorem \ref{thm4.8}). 

\section{Mathematical Formulation}\label{sec2}\setcounter{equation}{0}
This section provides the necessary function spaces needed to obtain the global solvability results of the system \eqref{1}.   In our analysis, the parameter $\alpha$ does not play a major role and we set $\alpha$ to be zero in \eqref{1} in the rest of the paper.

\subsection{Function spaces} Let $\C_0^{\infty}(\mathcal{O};\R^n)$ denotes the space of all infinitely differentiable functions  ($\R^n$-valued) with compact support in $\mathcal{O}\subset\R^n$.  We define 
\begin{align*} 
\mathcal{V}&:=\{\u\in\C_0^{\infty}(\mathcal{O},\R^n):\nabla\cdot\u=0\},\\
\mathbb{H}&:=\text{the closure of }\ \mathcal{V} \ \text{ in the Lebesgue space } \L^2(\mathcal{O})=\mathrm{L}^2(\mathcal{O};\R^n),\\
\mathbb{V}&:=\text{the closure of }\ \mathcal{V} \ \text{ in the Sobolev space } \H_0^1(\mathcal{O})=\mathrm{H}_0^1(\mathcal{O};\R^n),\\
\widetilde{\L}^{p}&:=\text{the closure of }\ \mathcal{V} \ \text{ in the Lebesgue space } \L^p(\mathcal{O})=\mathrm{L}^p(\mathcal{O};\R^n),
\end{align*}
for $p\in(2,\infty)$. Then under some smoothness assumptions on the boundary, we characterize the spaces $\H$, $\V$ and $\widetilde{\L}^p$ as 
$
\H=\{\u\in\L^2(\mathcal{O}):\nabla\cdot\u=0,\u\cdot\mathbf{n}\big|_{\partial\mathcal{O}}=0\}$,  with norm  $\|\u\|_{\H}^2:=\int_{\mathcal{O}}|\u(x)|^2\d x,
$
where $\mathbf{n}$ is the outward normal to $\partial\mathcal{O}$,
$
\V=\{\u\in\H_0^1(\mathcal{O}):\nabla\cdot\u=0\},$  with norm $ \|\u\|_{\V}^2:=\int_{\mathcal{O}}|\nabla\u(x)|^2\d x,
$ and $\widetilde{\L}^p=\{\u\in\L^p(\mathcal{O}):\nabla\cdot\u=0, \u\cdot\mathbf{n}\big|_{\partial\mathcal{O}}\},$ with norm $\|\u\|_{\widetilde{\L}^p}^p=\int_{\mathcal{O}}|\u(x)|^p\d x$, respectively.
Let $(\cdot,\cdot)$ denotes the inner product in the Hilbert space $\H$ and $\langle \cdot,\cdot\rangle $ denotes the induced duality between the spaces $\V$  and its dual $\V'$ as well as $\widetilde{\L}^p$ and its dual $\widetilde{\L}^{p'}$, where $\frac{1}{p}+\frac{1}{p'}=1$. Note that $\H$ can be identified with its dual $\H'$. We endow the space $\V\cap\widetilde{\L}^{p}$ with the norm $\|\u\|_{\V}+\|\u\|_{\widetilde{\L}^{p}},$ for $\u\in\V\cap\widetilde{\L}^p$ and its dual $\V'+\widetilde{\L}^{p'}$ with the norm $$\inf\left\{\max\left(\|\v_1\|_{\V'},\|\v_1\|_{\widetilde{\L}^{p'}}\right):\v=\v_1+\v_2, \ \v_1\in\V', \ \v_2\in\widetilde{\L}^{p'}\right\}.$$   We first note that $\mathcal{V}\subset\V\cap\widetilde{\L}^{p}\subset\H$ and $\mathcal{V}$ is dense in $\H,\V$ and $\widetilde{\L}^{p},$ and hence $\V\cap\widetilde{\L}^{p}$ is dense in $\H$. We have the following continuous  embedding also:
$$\V\cap\widetilde{\L}^{p}\hookrightarrow\H\equiv\H'\hookrightarrow\V'+\widetilde\L^{\frac{p}{p-1}}.$$ One can define equivalent norms on $\V\cap\widetilde\L^{p}$ and $\V'+\widetilde\L^{\frac{p}{p-1}}$ as  (see \cite{NAEG})
\begin{align*}
\|\u\|_{\V\cap\widetilde\L^{p}}=\left(\|\u\|_{\V}^2+\|\u\|_{\widetilde\L^{p}}^2\right)^{1/2}\ \text{ and } \ \|\u\|_{\V'+\widetilde\L^{\frac{p}{p-1}}}=\inf_{\u=\v+\w}\left(\|\v\|_{\V'}^2+\|\w\|_{\widetilde\L^{\frac{p}{p-1}}}^2\right)^{1/2}.
\end{align*}
The  following  interpolation inequality is frequently used in the paper. Assume $1\leq s\leq r\leq t\leq \infty$, $\theta\in(0,1)$ such that $\frac{1}{r}=\frac{\theta}{s}+\frac{1-\theta}{t}$ and $\u\in\L^s(\mathcal{O})\cap\L^t(\mathcal{O})$, then we have 
\begin{align}\label{211}
\|\u\|_{\L^r}\leq\|\u\|_{\L^s}^{\theta}\|\u\|_{\L^t}^{1-\theta}. 
\end{align}
\subsection{Linear operator}
	Let $\mathrm{P}_{\H} : \L^p(\mathcal{O}) \to\H$ denotes the \emph{Helmholtz-Hodge  projection} (\cite{DFHM}). For $p=2$, $\mathrm{P}_{\H}$ becomes an orthogonal projection and for $2<p<\infty$, it is a bounded linear operator. We define
\begin{equation*}
\left\{
\begin{aligned}
\A\u:&=-\mathrm{P}_{\H}\Delta\u,\;\u\in\D(\A),\\ \D(\A):&=\V\cap\H^{2}(\mathcal{O}).
\end{aligned}
\right.
\end{equation*}
It can be easily seen that the operator $\A$ is a non-negative self-adjoint operator in $\H$ with $\V=\D(\A^{1/2})$ and \begin{align}\label{2.7a}\langle \A\u,\u\rangle =\|\u\|_{\V}^2,\ \textrm{ for all }\ \u\in\V, \ \text{ so that }\ \|\A\u\|_{\V'}\leq \|\u\|_{\V}.\end{align}
For a bounded domain $\mathcal{O}$, the operator $\A$ is invertible and its inverse $\A^{-1}$ is bounded, self-adjoint and compact in $\H$. Thus, using spectral theorem, the spectrum of $\A$ consists of an infinite sequence $0< \uplambda_1\leq \uplambda_2\leq\ldots\leq \uplambda_k\leq \ldots,$ with $\uplambda_k\to\infty$ as $k\to\infty$ of eigenvalues. 
Moreover, there exists an orthogonal basis $\{e_k\}_{k=1}^{\infty} $ of $\H$ consisting of eigenvectors of $\A$ such that $\A e_k =\uplambda_ke_k$,  for all $ k\in\mathbb{N}$.  We know that $\u$ can be expressed as $\u=\sum_{k=1}^{\infty}\langle\u,e_k\rangle e_k$ and $\A\u=\sum_{k=1}^{\infty}\uplambda_k\langle\u,e_k\rangle e_k$. Thus, it is immediate that 
\begin{align}\label{poin}
\|\nabla\u\|_{\mathbb{H}}^2=\langle \A\u,\u\rangle =\sum_{k=1}^{\infty}\uplambda_k|\langle \u,e_k\rangle|^2\geq \uplambda_1\sum_{k=1}^{\infty}|\langle\u,e_k\rangle|^2=\uplambda_1\|\u\|_{\mathbb{H}}^2,
\end{align}
which is the Poincar\'e inequality.  In this work, we also need the fractional powers of $\A$.  For $\u\in \H$ and  $\alpha>0,$ we define
$\A^\alpha \u=\sum_{k=1}^\infty \uplambda_k^\alpha (\u,e_k) e_k,  \ \u\in\D(\A^\alpha), $ where $\D(\A^\alpha)=\left\{\u\in \H:\sum_{k=1}^\infty \uplambda_k^{2\alpha}|(\u,e_k)|^2<+\infty\right\}.$ 
Here  $\D(\A^\alpha)$ is equipped with the norm 
\begin{equation} \label{fn}
\|\A^\alpha \u\|_{\H}=\left(\sum_{k=1}^\infty \uplambda_k^{2\alpha}|(\u,e_k)|^2\right)^{1/2}.
\end{equation}
It can be easily seen that $\D(\A^0)=\H,$  $\D(\A^{1/2})=\V$. We set $\V_\alpha= \D(\A^{\alpha/2})$ with $\|\u\|_{\V_{\alpha}} =\|\A^{\alpha/2} \u\|_{\H}.$   Using Rellich-Kondrachov compactness embedding theorem, we know that for any $0\leq s_1<s_2,$ the embedding $\D(\A^{s_2})\subset \D(\A^{s_1})$ is also compact.

\subsection{Bilinear operator}
Let us define the \emph{trilinear form} $b(\cdot,\cdot,\cdot):\V\times\V\times\V\to\R$ by $$b(\u,\v,\w)=\int_{\mathcal{O}}(\u(x)\cdot\nabla)\v(x)\cdot\w(x)\d x=\sum_{i,j=1}^n\int_{\mathcal{O}}\u_i(x)\frac{\partial \v_j(x)}{\partial x_i}\w_j(x)\d x.$$ If $\u, \v$ are such that the linear map $b(\u, \v, \cdot) $ is continuous on $\V$, the corresponding element of $\V'$ is denoted by $\B(\u, \v)$. We also denote (with an abuse of notation) $\B(\u) = \B(\u, \u)=\mathrm{P}_{\H}(\u\cdot\nabla)\u$.
An integration by parts gives 
\begin{equation*}
\left\{
\begin{aligned}
b(\u,\v,\v) &= 0,\text{ for all }\u,\v \in\V,\\
b(\u,\v,\w) &=  -b(\u,\w,\v),\text{ for all }\u,\v,\w\in \V.
\end{aligned}
\right.\end{equation*}
In the trilinear form, an application of H\"older's inequality yields
\begin{align*}
|b(\u,\v,\w)|=|b(\u,\w,\v)|\leq \|\u\|_{\widetilde{\L}^{r+1}}\|\v\|_{\widetilde{\L}^{\frac{2(r+1)}{r-1}}}\|\w\|_{\V},
\end{align*}
for all $\u\in\V\cap\widetilde{\L}^{r+1}$, $\v\in\V\cap\widetilde{\L}^{\frac{2(r+1)}{r-1}}$ and $\w\in\V$, so that we get 
\begin{align}\label{2p9}
\|\B(\u,\v)\|_{\V'}\leq \|\u\|_{\widetilde{\L}^{r+1}}\|\v\|_{\widetilde{\L}^{\frac{2(r+1)}{r-1}}}.
\end{align}
Hence, the trilinear map $b : \V\times\V\times\V\to \R$ has a unique extension to a bounded trilinear map from $(\V\cap\widetilde{\L}^{r+1})\times(\V\cap\widetilde{\L}^{\frac{2(r+1)}{r-1}})\times\V$ to $\R$. It can also be seen that $\B$ maps $ \V\cap\widetilde{\L}^{r+1}$  into $\V'+\widetilde{\L}^{\frac{r+1}{r}}$. Using interpolation inequality (see \eqref{211}), we get 
\begin{align}\label{212}
\left|\langle \B(\u,\u),\v\rangle \right|=\left|b(\u,\v,\u)\right|\leq \|\u\|_{\widetilde{\L}^{r+1}}\|\u\|_{\widetilde{\L}^{\frac{2(r+1)}{r-1}}}\|\v\|_{\V}\leq\|\u\|_{\widetilde{\L}^{r+1}}^{\frac{r+1}{r-1}}\|\u\|_{\H}^{\frac{r-3}{r-1}},
\end{align}
for all $\v\in\V\cap\widetilde{\L}^{r+1}$. Thus, for $r>3$, we have 
\begin{align}\label{2.9a}
\|\B(\u)\|_{\V'+\widetilde{\L}^{\frac{r+1}{r}}}\leq\|\u\|_{\widetilde{\L}^{r+1}}^{\frac{r+1}{r-1}}\|\u\|_{\H}^{\frac{r-3}{r-1}}.
\end{align}
Using \eqref{2p9}, for $\u,\v\in\V\cap\widetilde{\L}^{r+1}$, we also have 
\begin{align}\label{lip}
\|\B(\u)-\B(\v)\|_{\V'+\widetilde{\L}^{\frac{r+1}{r}}}&\leq \|\B(\u-\v,\u)\|_{\V'}+\|\B(\v,\u-\v)\|_{\V'}\nonumber\\&\leq \left(\|\u\|_{\widetilde{\L}^{\frac{2(r+1)}{r-1}}}+\|\v\|_{\widetilde{\L}^{\frac{2(r+1)}{r-1}}}\right)\|\u-\v\|_{\widetilde{\L}^{r+1}}\nonumber\\&\leq \left(\|\u\|_{\H}^{\frac{r-3}{r-1}}\|\u\|_{\widetilde{\L}^{r+1}}^{\frac{2}{r-1}}+\|\v\|_{\H}^{\frac{r-3}{r-1}}\|\v\|_{\widetilde{\L}^{r+1}}^{\frac{2}{r-1}}\right)\|\u-\v\|_{\widetilde{\L}^{r+1}},
\end{align}
for $r>3$, by using the interpolation inequality. For $r=3$, a calculation similar to \eqref{lip} yields 
\begin{align*}
\|\B(\u)-\B(\v)\|_{\V'+\widetilde{\L}^{\frac{4}{3}}}&\leq \left(\|\u\|_{\widetilde{\L}^{4}}+\|\v\|_{\widetilde{\L}^{4}}\right)\|\u-\v\|_{\widetilde{\L}^{4}},
\end{align*}
  hence $\B(\cdot):\V\cap\widetilde{\L}^{4}\to\V'+\widetilde{\L}^{\frac{4}{3}}$ is a locally Lipschitz operator. For more details, see \cite{Te1}. 
\subsection{Nonlinear operator}
Let us now consider the operator $\mathcal{C}(\u):=\P_{\H}(|\u|^{r-1}\u)$. It is immediate that $\langle\mathcal{C}(\u),\u\rangle =\|\u\|_{\widetilde{\L}^{r+1}}^{r+1}$ and the map $\mathcal{C}(\cdot):\widetilde{\L}^{r+1}\to\widetilde{\L}^{\frac{r+1}{r}}$ is Gateaux differentiable with Gateaux derivative $\mathcal{C}'(\u)\v=r\mathrm{P}_{\H}(|\u|^{r-1}\v),$ for $\v\in\widetilde{\L}^{r+1}$. For $0<\theta<1$ and $\u,\v\in\widetilde{\L}^{r+1}$, using Taylor's formula (Theorem 6.5, \cite{RFC}), we have 
\begin{align}\label{213}
\langle \P_{\H}(|\u|^{r-1}\u)-\P_{\H}(|\v|^{r-1}\v),\w\rangle&\leq \|(|\u|^{r-1}\u)-(|\v|^{r-1}\v)\|_{\widetilde{\L}^{\frac{r+1}{r}}}\|\w\|_{\widetilde{\L}^{r+1}}\nonumber\\&\leq \sup_{0<\theta<1}r\|(\u-\v)|\theta\u+(1-\theta)\v|^{r-1}\|_{\widetilde{\L}^{{\frac{r+1}{r}}}}\|\w\|_{\widetilde{\L}^{r+1}}\nonumber\\&\leq \sup_{0<\theta<1}r\|\theta\u+(1-\theta)\v\|_{\widetilde{\L}^{r+1}}^{r-1}\|\u-\v\|_{\widetilde{\L}^{r+1}}\|\w\|_{\widetilde{\L}^{r+1}}\nonumber\\&\leq r\left(\|\u\|_{\widetilde{\L}^{r+1}}+\|\v\|_{\widetilde{\L}^{r+1}}\right)^{r-1}\|\u-\v\|_{\widetilde{\L}^{r+1}}\|\w\|_{\widetilde{\L}^{r+1}},
\end{align}
for all $\u,\v,\w\in\widetilde{\L}^{r+1}$. 
Thus the operator $\mathcal{C}(\cdot):\widetilde{\L}^{r+1}\to\widetilde{\L}^{\frac{r+1}{r}}$ is locally Lipschitz. Moreover, 	for any $r\in[1,\infty)$, we have 
\begin{align}\label{224}
&\langle\mathrm{P}_{\H}(\u|\u|^{r-1})-\mathrm{P}_{\H}(\v|\v|^{r-1}),\u-\v\rangle\nonumber\\&=\langle|\u|^{r-1},|\u-\v|^2\rangle+\langle|\v|^{r-1},|\u-\v|^2\rangle+\langle\v|\u|^{r-1}-\u|\v|^{r-1},\u-\v\rangle\nonumber\\&=\||\u|^{\frac{r-1}{2}}(\u-\v)\|_{\H}^2+\||\v|^{\frac{r-1}{2}}(\u-\v)\|_{\H}^2\nonumber\\&\quad+\langle\u\cdot\v,|\u|^{r-1}+|\v|^{r-1}\rangle-\langle|\u|^2,|\v|^{r-1}\rangle-\langle|\v|^2,|\u|^{r-1}\rangle.
\end{align}
But, we know that 
\begin{align*}
&\langle\u\cdot\v,|\u|^{r-1}+|\v|^{r-1}\rangle-\langle|\u|^2,|\v|^{r-1}\rangle-\langle|\v|^2,|\u|^{r-1}\rangle\nonumber\\&=-\frac{1}{2}\||\u|^{\frac{r-1}{2}}(\u-\v)\|_{\H}^2-\frac{1}{2}\||\v|^{\frac{r-1}{2}}(\u-\v)\|_{\H}^2+\frac{1}{2}\langle\left(|\u|^{r-1}-|\v|^{r-1}\right),\left(|\u|^2-|\v|^2\right)\rangle \nonumber\\&\geq -\frac{1}{2}\||\u|^{\frac{r-1}{2}}(\u-\v)\|_{\H}^2-\frac{1}{2}\||\v|^{\frac{r-1}{2}}(\u-\v)\|_{\H}^2.
\end{align*}
From \eqref{224}, we finally have 
\begin{align}\label{2.23}
&\langle\mathrm{P}_{\H}(\u|\u|^{r-1})-\mathrm{P}_{\H}(\v|\v|^{r-1}),\u-\v\rangle\geq \frac{1}{2}\||\u|^{\frac{r-1}{2}}(\u-\v)\|_{\H}^2+\frac{1}{2}\||\v|^{\frac{r-1}{2}}(\u-\v)\|_{\H}^2\geq 0,
\end{align}
for $r\geq 1$.  	It is important to note that 
\begin{align}\label{a215}
\|\u-\v\|_{\wi\L^{r+1}}^{r+1}&=\int_{\mathcal{O}}|\u(x)-\v(x)|^{r-1}|\u(x)-\v(x)|^{2}\d x\nonumber\\&\leq 2^{r-2}\int_{\mathcal{O}}(|\u(x)|^{r-1}+|\v(x)|^{r-1})|\u(x)-\v(x)|^{2}\d x\nonumber\\&\leq 2^{r-2}\||\u|^{\frac{r-1}{2}}(\u-\v)\|_{\L^2}^2+2^{r-2}\||\v|^{\frac{r-1}{2}}(\u-\v)\|_{\L^2}^2,
\end{align}
for $r\geq 1$ (replace $2^{r-2}$ with $1,$ for $1\leq r\leq 2$).
\subsection{Monotonicity}
Let us now discuss about the monotonicity as well as the hemicontinuity properties of the linear and nonlinear operators, which plays a crucial role in the global solvability of the system \eqref{1}. 
\begin{definition}[\cite{VB}]
	Let $\X$ be a Banach space and let $\X^{'}$ be its topological dual.
	An operator $\G:\mathrm{D}\rightarrow
	\X^{'},$ $\mathrm{D}=\mathrm{D}(\G)\subset \X$ is said to be
	\emph{monotone} if
	$$\langle\G(x)-\G(y),x-y\rangle\geq
	0,\ \text{ for all } \ x,y\in \mathrm{D}.$$ 
	The operator $\G(\cdot)$ is said to be \emph{hemicontinuous}, if for all $x, y\in\X$ and $w\in\X',$ $$\lim_{\lambda\to 0}\langle\G(x+\lambda y),w\rangle=\langle\G(x),w\rangle.$$
	The operator $\G(\cdot)$ is called \emph{demicontinuous}, if for all $x\in\mathrm{D}$ and $y\in\X$, the functional $x \mapsto\langle \G(x), y\rangle$  is continuous, or in other words, $x_k\to x$ in $\X$ implies $\G(x_k)\xrightarrow{w}\G(x)$ in $\X'$. Clearly demicontinuity implies hemicontinuity. 
\end{definition}
\begin{theorem}[Theorem 2.2., \cite{MTM7}]\label{thm2.2}
	Let $\u,\v\in\V\cap\widetilde{\L}^{r+1}$, for $r>3$. Then,	for the operator $\G(\u)=\mu \A\u+\B(\u)+\beta\mathcal{C}(\u)$, we  have 
	\begin{align}\label{fe}
	\langle(\G(\u)-\G(\v),\u-\v\rangle+\eta\|\u-\v\|_{\H}^2\geq 0,
	\end{align}
	where \begin{align}\label{215}\eta=\frac{r-3}{2\mu(r-1)}\left(\frac{2}{\beta\mu (r-1)}\right)^{\frac{2}{r-3}}.\end{align} That is, the operator $\G+\eta\mathrm{I}$ is a monotone operator from $\V\cap\widetilde{\L}^{r+1}$ to $\V'+\widetilde{\L}^{\frac{r+1}{r}}$. 
\end{theorem}

\begin{theorem}[Theorem 2.3, \cite{MTM7}]\label{thm2.3}
	For the critical case $r=3$ with $2\beta\mu \geq 1$, the operator $\G(\cdot):\V\cap\widetilde{\L}^{r+1}\to \V'+\widetilde{\L}^{\frac{r+1}{r}}$ is globally monotone, that is, for all $\u,\v\in\V$, we have 
\begin{align}\label{218}\langle\G(\u)-\G(\v),\u-\v\rangle\geq 0.\end{align}
	\end{theorem}

\begin{lemma}[Lemma 2.5, \cite{MTM7}]\label{lem2.8}
	The operator $\G:\V\cap\widetilde{\L}^{r+1}\to \V'+\widetilde{\L}^{\frac{r+1}{r}}$ is demicontinuous. 
\end{lemma}

\subsection{Abstract formulation and weak solution} We take the Helmholtz-Hodge orthogonal projection $\mathrm{P}_{\H}$ in \eqref{1} to obtain  the abstract formulation for  $t\in(0,T)$ as: 
\begin{equation}\label{kvf}
\left\{
\begin{aligned}
\frac{\d\u(t)}{\d t}+\mu \A\u(t)+\B(\u(t))+\beta\mathcal{C}(\u(t))&=\f(t),\ \mathrm{ in }\  \V'+\wi\L^{\frac{r+1}{r}},\\
\u(0)&=\u_0\in\H\cap\widetilde{\L}^{r+1},
\end{aligned}
\right.
\end{equation}
for $r\geq 3$, where $\f\in\mathrm{L}^2(0,T;\V')$. Strictly speaking, one should use $\mathrm{P}_{\H}\f$ instead of $\f$, for simplicity, we use $\f$. Let us now provide the definition of \emph{weak solution} of the system \eqref{kvf} for $r\geq 3$.
\begin{definition}\label{weakd}
For $r\geq 3$,	a function  $$\u\in\mathrm{C}([0,T];\H)\cap\mathrm{L}^2(0,T;\V)\cap\mathrm{L}^{r+1}(0,T;\widetilde{\L}^{r+1}),$$  with $\partial_t\u\in\mathrm{L}^{2}(0,T;\mathbb{V}')+\mathrm{L}^{\frac{r+1}{r}}(0,T;\L^{\frac{r+1}{r}}),$  is called a \emph{weak solution} to the system (\ref{kvf}), if for $\f\in\mathrm{L}^2(0,T;\V')$, $\u_0\in\H$ and $\v\in\V\cap\widetilde{\L}^{r+1}$, $\u(\cdot)$ satisfies:
	\begin{equation}\label{3.13}
	\left\{
	\begin{aligned}
	\langle\partial_t\u(t)+\mu \mathrm{A}\u(t)+\mathrm{B}(\u(t))+\beta\mathcal{C}(\u(t)),\v\rangle&=\langle\f(t),\v\rangle,\\
	\lim_{t\downarrow 0}\int_{\mathcal{O}}\u(t)\v\d x&=\int_{\mathcal{O}}\u_0\v\d x,
	\end{aligned}
	\right.
	\end{equation}
	and the energy equality:
	\begin{align}\label{237}
	\frac{1}{2}\frac{\d}{\d t}\|\u(t)\|^2_{\H}+\mu \|\u(t)\|^2_{\V}+\beta\|\u(t)\|_{\widetilde{\L}^{r+1}}^{r+1}=\langle \f(t),\u(t)\rangle.
	\end{align}
\end{definition}
The mononicity property of the linear and nonlinear operators established in Theorem \ref{thm2.2}, demicontinuity property obtained in Lemma \ref{lem2.8} and  the Minty-Browder technique are used to obtain the following global solvability results. Note that the following Theorem is  true for $2\leq n\leq 4$ ($r\in[1,\infty),$ for $n=2$).
\begin{theorem}[Theorem 3.4, \cite{MTM7}]\label{main2}
	Let $\u_0\in \H$ and $\f\in\mathrm{L}^{2}(0,T;\V')$  be given.  Then  there exists a unique weak solution to the system (\ref{kvf}) satisfying $$\u\in\mathrm{C}([0,T];\H)\cap\mathrm{L}^2(0,T;\V)\cap\mathrm{L}^{r+1}(0,T;\widetilde{\L}^{r+1}),$$ for $r\geq 3$ ($2\beta\mu\geq 1$ for $r=3$).
\end{theorem}

\section{Stochastic Navier-Stokes-Brinkman-Forchheimer  equations}\label{sec3}\setcounter{equation}{0} In this section, we consider the following stochastic  convective Brinkman-Forchheimer  equations perturbed by multiplicative Gaussian noise:
\begin{equation}\label{31}
\left\{
\begin{aligned}
\d\u(t)-\mu \Delta\u(t)&+(\u(t)\cdot\nabla)\u(t)+\beta|\u(t)|^{r-1}\u(t)+\nabla p(t)\\&=\Phi(t,\u(t))\d\W(t), \ \text{ in } \ \mathcal{O}\times(0,T), \\ \nabla\cdot\u(t)&=0, \ \text{ in } \ \mathcal{O}\times(0,T), \\
\u(t)&=\mathbf{0}\ \text{ on } \ \partial\mathcal{O}\times(0,T), \\
\u(0)&=\u_0 \ \text{ in } \ \mathcal{O},
\end{aligned}
\right.
\end{equation} 
where $\W(\cdot)$ is an $\H$-valued Wiener process. We also discuss the global solvability of the system \eqref{31} under suitable assumptions on the noise coefficient.

\subsection{Noise coefficient} Let $(\Omega,\mathscr{F},\mathbb{P})$ be a complete probability space equipped with an increasing family of sub-sigma fields $\{\mathscr{F}_t\}_{0\leq t\leq T}$ of $\mathscr{F}$ satisfying:
\begin{enumerate}
	\item [(i)] $\mathscr{F}_0$ contains all elements $F\in\mathscr{F}$ with $\mathbb{P}(F)=0$,
	\item [(ii)] $\mathscr{F}_t=\mathscr{F}_{t+}=\bigcap\limits_{s>t}\mathscr{F}_s,$ for $0\leq t\leq T$.
\end{enumerate} 

\begin{definition}
	A stochastic process $\{\W(t)\}_{0\leq
		t\leq T}$ is said to be an \emph{$\H$-valued $\mathscr{F}_t$-adapted
		Wiener process} with covariance operator $\Q$ if
	\begin{enumerate}
		\item [$(i)$] for each non-zero $h\in \H$, $\|\Q^{\frac{1}{2}}h\|_{\H}^{-1} (\W(t), h)$ is a standard one dimensional Wiener process,
		\item [$(ii)$] for any $h\in \H, (\W(t), h)$ is a martingale adapted to $\mathscr{F}_t$.
	\end{enumerate}
\end{definition}
The stochastic process $\{\W(t) : 0\leq t\leq T\}$ is a $\H$-valued Wiener process with covariance $\Q$ if and only if for arbitrary $t$, the process $\W(t)$ can be expressed as $\W(t,x) =\sum_{k=1}^{\infty}\sqrt{\mu_k}e_k(x)\beta_k(t)$, where  $\beta_{k}(t),k\in\mathbb{N}$ are independent one dimensional Brownian motions on $(\Omega,\mathscr{F},\mathbb{P})$ and $\{e_k \}_{k=1}^{\infty}$ are the orthonormal basis functions of $\H$ such that $\Q e_k=\mu_k e_k$.  If $\W(\cdot)$ is an $\H$-valued Wiener process with covariance operator $\Q$ with $\Tr \Q=\sum_{k=1}^{\infty} \mu_k< +\infty$, then $\W(\cdot)$ is a Gaussian process on $\H$ and $ \E[\W(t)] = 0,$ $\textrm{Cov} [\W(t)] = t\Q,$ $t\geq 0.$ The space $\H_0=\Q^{\frac{1}{2}}\H$ is a Hilbert space equipped with the inner product $(\cdot, \cdot)_0$, $$(\u, \v)_0 =\sum_{k=1}^{\infty}\frac{1}{\uplambda_k}(\u,e_k)(\v,e_k)= (\Q^{-\frac{1}{2}}\u, \Q^{-\frac{1}{2}}\v),\ \text{ for all } \ \u, \v\in \H_0,$$ where $\Q^{-\frac{1}{2}}$ is the pseudo-inverse of $\Q^{\frac{1}{2}}$.

Let $\mathcal{L}(\H)$ denotes the space of all bounded linear operators on $\H$ and $\mathcal{L}_{\Q}:=\mathcal{L}_{\Q}(\H)$ denotes the space of all Hilbert-Schmidt operators from $\H_0:=\Q^{\frac{1}{2}}\H$ to $\H$.  Since $\Q$ is a trace class operator, the embedding of $\H_0$ in $\H$ is Hilbert-Schmidt and the space $\mathcal{L}_{\Q}$ is a Hilbert space equipped with the norm $ \left\|\Phi\right\|^2_{\mathcal{L}_{\Q}}=\Tr(\Phi {\Q}\Phi^*)=\sum_{k=1}^{\infty}\| {\Q}^{\frac{1}{2}}\Phi^*e_k\|_{\H}^2 $ and inner product $ \left(\Phi,\Psi\right)_{\mathcal{L}_{\Q}}=\Tr\left(\Phi {\Q}\Psi^*\right)=\sum_{k=1}^{\infty}({\Q}^{\frac{1}{2}}\Psi^*e_k,{\Q}^{\frac{1}{2}}\Phi^*e_k) $. For more details, the interested readers are referred to see \cite{DaZ}.

\begin{hypothesis}\label{hyp}
	The noise coefficient $\Phi(\cdot,\cdot)$ satisfies: 
	\begin{itemize}
		\item [(H.1)] The function $\Phi\in\C([0,T]\times\V;\mathcal{L}_{\Q}(\H))$.
		\item[(H.2)]  (Growth condition)	There exists a positive	constant $K$ such that for all $t\in[0,T]$ and $\u\in\H$,
		\begin{equation*}
		\|\Phi(t, \u)\|^{2}_{\mathcal{L}_{\Q}} 	\leq K\left(1 +\|\u\|_{\H}^{2}\right),
		\end{equation*}
		
		\item[(H.3)]  (Lipschitz condition)
		There exists a positive constant $L$ such that for any $t\in[0,T]$ and all $\u_1,\u_2\in\H$,
		\begin{align*}
		\|\Phi(t,\u_1) - \Phi(t,\u_2)\|^2_{\mathcal{L}_{\Q}}\leq L\|\u_1 -	\u_2\|_{\H}^2.
		\end{align*}
	\end{itemize}
\end{hypothesis}

\subsection{Abstract formulation of the stochastic system}\label{sec2.4}
On  taking orthogonal projection $\mathrm{P}_{\H}$ onto the first equation in \eqref{31}, we get 
\begin{equation}\label{32}
\left\{
\begin{aligned}
\d\u(t)+[\mu \A\u(t)+\B(\u(t))+\beta\mathcal{C}(\u(t))]\d t&=\Phi(t,\u(t))\d\W(t), \ t\in(0,T),\\
\u(0)&=\u_0,
\end{aligned}
\right.
\end{equation}
where $\u_0\in\mathrm{L}^2(\Omega;\H)$. Strictly speaking, one should write $\mathrm{P}_{\H}\Phi$ instead of $\Phi$.

Let us now provide the definition of a unique global strong solution in the probabilistic sense to the system (\ref{32}).
\begin{definition}[Global strong solution]
	Let $\u_0\in\mathrm{L}^2(\Omega;\H)$ be given. An $\H$-valued $(\mathscr{F}_t)_{t\geq 0}$-adapted stochastic process $\u(\cdot)$ is called a \emph{strong solution} to the system (\ref{32}) if the following conditions are satisfied: 
	\begin{enumerate}
		\item [(i)] the process $\u\in\mathrm{L}^2(\Omega;\mathrm{L}^{\infty}(0,T;\H)\cap\mathrm{L}^2(0,T;\V))\cap\mathrm{L}^{r+1}(\Omega;\mathrm{L}^{r+1}(0,T;\widetilde{\L}^{r+1}))$ and $\u(\cdot)$ has a $\V\cap\widetilde{\L}^{r+1}$-valued  modification, which is progressively measurable with continuous paths in $\H$ and $\u\in\C([0,T];\H)\cap\mathrm{L}^2(0,T;\V)\cap\mathrm{L}^{r+1}(0,T;\widetilde{\L}^{r+1})$, $\mathbb{P}$-a.s.,
		\item [(ii)] the following equality holds for every $t\in [0, T ]$, as an element of $\V'+\wi\L^{\frac{r+1}{r}},$ $\mathbb{P}$-a.s.
		\begin{align}\label{4.4}
		\u(t)&=\u_0-\int_0^t\left[\mu \A\u(s)+\B(\u(s))+\beta\mathcal{C}(\u(s))\right]\d s+\int_0^t\Phi(s,\u(s))\d \W(s),
		\end{align}
		\item [(iii)] the following It\^o formula holds true: 
	\begin{align}\label{a34}
	&	\|\u(t)\|_{\H}^2+2\mu \int_0^t\|\u(s)\|_{\V}^2\d s+2\beta\int_0^t\|\u(s)\|_{\widetilde{\L}^{r+1}}^{r+1}\d s\nonumber\\&=\|{\u_0}\|_{\H}^2+\int_0^t\|\Phi(s,\u(s))\|_{\mathcal{L}_{\Q}}^2\d s+2\int_0^t(\Phi(s,\u(s))\d\W(s),\u(s)),
	\end{align}
	for all $t\in(0,T)$, $\mathbb{P}$-a.s.
	\end{enumerate}
\end{definition}
An alternative version of condition (\ref{4.4}) is to require that for any  $\v\in\V\cap\widetilde{\L}^{r+1}$:
\begin{align}\label{4.5}
(\u(t),\v)&=(\u_0,\v)-\int_0^t\langle\mu \A\u(s)+\B(\u(s))+\beta\mathcal{C}(\u(s)),\v\rangle\d s\no\\&\quad+\int_0^t\left(\Phi(s,\u(s))\d \W(s),\v\right),\ \mathbb{P}\text{-a.s.}
\end{align}	
\begin{definition}
	A strong solution $\u(\cdot)$ to (\ref{32}) is called a
	\emph{pathwise  unique strong solution} if
	$\widetilde{\u}(\cdot)$ is an another strong
	solution, then $$\mathbb{P}\big\{\omega\in\Omega:\u(t)=\widetilde{\u}(t),\ \text{ for all }\ t\in[0,T]\big\}=1.$$ 
\end{definition}

\begin{theorem}[Theorem 3.7, \cite{MTM8}]\label{exis2}
	Let $\u_0\in \mathrm{L}^2(\Omega;\H)$, for $r\geq 3$ be given ($2\beta\mu\geq 1,$ for $r=3$).  Then there exists a \emph{pathwise unique strong solution}
	$\u(\cdot)$ to the system (\ref{32}) such that \begin{align*}\u&\in\mathrm{L}^2(\Omega;\mathrm{L}^{\infty}(0,T;\H)\cap\mathrm{L}^2(0,T;\V))\cap\mathrm{L}^{r+1}(\Omega;\mathrm{L}^{r+1}(0,T;\widetilde{\L}^{r+1})),\end{align*} with $\mathbb{P}$-a.s., continuous trajectories in $\H$.
\end{theorem}

\section{Large Deviation Principle}\label{sec4}\setcounter{equation}{0}

In this section, we establish the Wentzell-Freidlin (see \cite{FW}) type large deviation principle for the SCBF equations using the well known results of Varadhan and Bryc (see \cite{DZ,Va}) and Budhiraja-Dupuis (see \cite{BD1}). Interested readers are referred to see   \cite{SSSP} (LDP for the 2D stochastic Navier-Stokes equations), \cite{ICAM} (LDP for some 2D hydrodynamic systems),  \cite{MTM6} (LDP for the 2D Oldroyd fluids) for application of such methods to various hydrodynamic models. 

Let $(\Omega,\mathscr{F},\mathbb{P})$ be a probability space with an increasing family $\{\mathscr{F}_t\}_{t\geq 0}$ of the sub $\sigma$-fields of $\mathscr{F}$ satisfying the usual conditions.  We consider the following stochastic CBF system:
\begin{equation}\label{2.18a}
\left\{
\begin{aligned}
\d\u^{\e}(t)&=-\left[\mu\A\u^{\e}(t)+\B(\u^{\e}(t))+\beta|\u^{\e}(t)|^{r-1}\u^{\e}(t)\right]\d t+\sqrt{\e}\Phi(t,\u^{\e}(t))\d\W(t),\\
\u^{\e}(0)&=\u_0,
\end{aligned}
\right.
\end{equation}
for some fixed point $\u_0$ in $\H$. From Theorem 3.7, \cite{MTM8} (see Theorem \ref{exis2}), it is known that the system \eqref{2.18a} has a pathwise unique strong solution $\u^{\e}(\cdot)$ with $\mathscr{F}_t$-adapted paths (that is, for any $t\in[0,T]$ and $x\in\mathcal{O}$,  $\u^{\e}(t,x)$ is $\mathscr{F}_t$-measurable) in  $$\C([0,T];\H)\cap \mathrm{L}^2(0,T;\V)\cap\mathrm{L}^{r+1}(0,T;\widetilde\L^{r+1}),\ \mathbb{P}\text{-a.s.},$$ for $r\geq 3$ ($2\beta\mu\geq 1,$ for $r=3$). Moreover, such a strong solution satisfies the energy equality (It\^o's formula):
	\begin{align}\label{6.2}
&	\|\u^{\e}(t)\|_{\H}^2+2\mu \int_0^t\|\u^{\e}(s)\|_{\V}^2\d s+2\beta\int_0^t\|\u^{\e}(s)\|_{\widetilde{\L}^{r+1}}^{r+1}\d s\nonumber\\&=\|{\u_0}\|_{\H}^2+\e\int_0^t\|\Phi(s,\u^{\e}(s))\|_{\mathcal{L}_{\Q}}^2\d s+2\sqrt{\e}\int_0^t(\Phi(s,\u^{\e}(s))\d\W(s),\u^{\e}(s)),
\end{align}
for all $t\in[0,T]$, $\mathbb{P}$-a.s. As the parameter $\e\downarrow 0$, the solution $\u^{\e}(\cdot)$ of (\ref{2.18a}) tends to the solution of the following deterministic system:
\begin{equation}\label{2.19}
\left\{
\begin{aligned}
\d\u^{0}(t)&=-\left[\mu\A\u^{0}(t)+\B(\u^{0}(t))+\beta\mathcal{C}(\u^0(t))\right]\d t,\\
\u^{0}(0)&=\u_0\in\H.
\end{aligned}
\right.
\end{equation}
In this section, we  investigate the large deviations of $\u^{\e}(\cdot)$ from the deterministic solution $\u^0(\cdot)$, as $\e\downarrow 0$. From Theorem 3.4, \cite{MTM7} (see  \cite{CLF} also), it is known that the system (\ref{2.19}) has a unique weak solution in the Leray-Hopf sense, satisfying the energy equality
\begin{align*}
\|\u^0(t)\|_{\H}^2+2\mu\int_0^t\|\u^0(s)\|_{\V}^2\d s+2\beta\int_0^t\|\u^0(s)\|_{\wi\L^{r+1}}^{r+1}\d s=\|\u_0\|_{\H}^2,
\end{align*}
for all $t\in[0,T]$ in the Polish space $\C([0,T];\H)\cap \mathrm{L}^2(0,T;\V)\cap\mathrm{L}^{r+1}(0,T;\widetilde\L^{r+1})$.
\subsection{Preliminaries}
In this subsection, we provide some preliminaries regarding the Large deviation principle (LDP). Let us denote by $\mathscr{E}$, a complete separable metric space  (Polish space)  with the Borel $\sigma$-field $\mathscr{B}(\mathscr{E})$.

\begin{definition}
	A function $\mathrm{I} : \mathscr{E}\rightarrow [0, \infty]$ is called a \emph{rate	function} if $\I$ is lower semicontinuous. A rate function $\I$ is	called a \emph{good rate function,} if for arbitrary $M \in [0,	\infty)$, the level set $\mathcal{K}_M = \big\{x\in\mathscr{E}: \I(x)\leq M\big\}$ is compact in $\mathscr{E}$.
\end{definition}
\begin{definition}[Large deviation principle]\label{LDP} Let $\I$ be a rate function on $\mathscr{E}$. A family $\big\{\mathrm{X}^{\varepsilon}: \varepsilon
	> 0\big\}$ of $\mathscr{E}$-valued random elements is said to satisfy \emph{the large deviation principle} on $\mathscr{E}$ with rate function $\I$, if the following two conditions hold:
	\begin{enumerate}
		\item[(i)] (Large deviation upper bound) For each closed set $\F\subset \mathscr{E}$:
		$$ \limsup_{\varepsilon\rightarrow 0} \varepsilon\log \mathbb{P}\left(\mathrm{X}^{\e}\in\F\right) \leq -\inf_{x\in \F} \I(x).$$
		\item[(ii)] (Large deviation lower bound) For each open set $\G\subset \mathscr{E}$:
		$$ \liminf_{\varepsilon\rightarrow 0}\varepsilon \log \mathbb{P}(\mathrm{X}^{\e}\in\G) \geq -\inf_{x\in \G} \I(x).$$
	\end{enumerate}
\end{definition}

\begin{definition}
	Let $\I$ be a rate function on $\mathscr{E}$. A family $\big\{\mathrm{X}^{\e} :\e > 0\big\}$ of $\mathscr{E}$-valued random elements is said to satisfy the \emph{Laplace principle} on $\mathscr{E}$ with rate function $\I,$ if for each real-valued, bounded and continuous function $h$ defined on $\mathscr{E}$, that is, for $h\in\C_b(\mathscr{E})$,
	\begin{equation}\label{LP}
	\lim_{\varepsilon \rightarrow 0} {\varepsilon }\log
	\mathbb{E}\left\{\exp\left[-
	\frac{1}{\varepsilon}h(\mathrm{X}^{\varepsilon})\right]\right\} = -\inf_{x
		\in \mathscr{E}} \big\{h(x) + \I(x)\big\}. 
	\end{equation}
\end{definition}

\begin{lemma}[Varadhan's Lemma, \cite{Va}]\label{VL}
	Let $\mathscr{E}$ be a Polish space and $\{\mathrm{X}^{\varepsilon}: \varepsilon > 0\}$
	be a family of $\mathscr{E}$-valued random elements satisfying LDP with rate
	function $\I$. Then $\{\mathrm{X}^{\varepsilon}: \varepsilon > 0\}$ satisfies
	the Laplace principle on $\mathscr{E}$ with the same rate function $\I$.
\end{lemma}
\begin{lemma}[Bryc's Lemma, \cite{DZ}]\label{BL}
	The Laplace principle implies the LDP with the same rate function.
\end{lemma}
It should be noted that Varadhan's Lemma together with Bryc's converse of Varadhan's Lemma state that for Polish space valued random elements, the Laplace principle and the large deviation principle are equivalent. 

\subsection{Functional setting and Budhiraja-Dupuis LDP}\label{sec4.2}
In this subsection, the notation and terminology are built in order to state the large deviations result of Budhiraja and Dupuis \cite{BD1} for Polish space valued random elements. Let us define $$\mathcal{A}:=\left\{\H_0\text{-valued }\{\mathscr{F}_t\}\text{-predictable processes }h\text{ such that }\int_0^T\|h(s)\|_{0}^2\d s<+\infty,\ \mathbb{P}\text{-a.s.}\right\}, $$ where $\H_0=\Q^{\frac{1}{2}}\H$ and $$\mathcal{S}_M: = \left\{h\in\mathrm{L}^2(0,T;\H_0): \int_0^T \|h(s)\|_0^2\d s \leq M\right\}.$$ It is known from \cite{BD2} that the space $\mathcal{S}_M$ is a compact metric space under the metric $\widetilde{d}(\u,\v)=\sum\limits_{j=1}^{\infty}\frac{1}{2^j}\left|\int_0^T(\u(s)-\v(s),\widetilde{e}_j(s))_0\d s\right|$, where $\{\widetilde{e}_j\}_{j=1}^{\infty}$ are orthonormal basis of $\mathrm{L}^2(0,T;\H_0)$. Since every compact metric space is complete, the set $\mathcal{S}_M$ endowed with the weak topology obtained from the metric $\widetilde{d}$ is a Polish space. Let us now define $$\mathcal{A}_M=\big\{h\in\mathcal{A}: h(\omega)\in \mathcal{S}_M, \ \mathbb{P}\text{-a.s.}\big\}.$$ Next, we state an important lemma regarding the convergence of the sequence  $\int_0^{\cdot}h_n(s)\d s$, which is useful in proving compactness as well as weak convergence results.

\begin{lemma}[Lemma 3.2, \cite{BD1}]
	Let $\{h_n\}$ be a sequence of elements from $\mathcal{A}_M,$ for some $0<M <+\infty$. Let the sequence $\{h_n\}$ converges in distribution to $h$ with respect to the weak topology on 	$\mathrm{L}^2(0,T;\H_0)$. Then $\int_0^{\cdot}h_n(s)\d s$ converges in distribution as $\C([0,T];\H)$-valued processes to $\int_0^{\cdot}h(s)\d s$ as $n\to\infty$.
\end{lemma} Let $\mathscr{E}$ denote a Polish space, and for $\e>0$, let $\mathcal{G}^{\e} :\C([0,T];\H)\to\mathscr{E}$ be a measurable map. Let us define $$\mathrm{X}^{\e} =\mathcal{G}^{\e}(\W(\cdot)).$$ We are interested in the large deviation principle for $\mathrm{X}^{\e}$ as $\e\to 0$.
\begin{hypothesis}\label{hyp1}
	There exists a measurable map $\mathcal{G}^0 : \C ([0, T ] ;\H )\to\mathscr{ E}$ such that the following
	hold:
	\begin{enumerate}
		\item [(i)] Let $\{h^{\e} :\e>0\}\subset \mathcal{A}_{M},$  for some $M<+\infty$. Let $h^{\e}$ converge in distribution as an
		$\mathcal{S}_M$-valued random elements to $h$ as $\e\to0$. Then $\mathcal{G}^{\e}(\W(\cdot)+\frac{1}{\sqrt{\e}}\int_0^{\cdot}h^{\e}(s)\d s)$ converges in distribution to $\mathcal{G}^0(\int_0^{\cdot}h(s)\d s)$ as $\e\to0$. 
		\item [(ii)] For every $M<+\infty$, the set $$\mathcal{K}_M=\left\{\mathcal{G}^0\left(\int_0^{\cdot}h(s)\d s\right):h\in \mathcal{S}_M\right\}$$ is a compact subset of $\mathscr{E}$.
	\end{enumerate}
\end{hypothesis}
For each $f\in\mathscr{E}$, we define
\begin{align}\label{rate}
\I(f):=\inf_{\left\{h\in\mathrm{L}^2(0,T;\H_0):f=\mathcal{G}^0\left(\int_0^{\cdot}h(s)\d s\right)\right\}}\left\{\frac{1}{2}\int_0^T\|h(s)\|_0^2\d s\right\},
\end{align}
where infimum over an empty set is taken as $+\infty$. Next, we state an important result due to Budhiraja and Dupuis \cite{BD1}.
\begin{theorem}[Budhiraja-Dupuis principle, Theorem 4.4, \cite{BD1}]\label{BD}
	Let $\mathrm{X}^{\e}= \mathcal{G}^{\e}(\W(\cdot))$. If $\{\mathcal{G}^{\e}\}$ satisfies the Hypothesis \ref{hyp1}, then the family $\{\mathrm{X}^{\e}:\e>0\}$ satisfies the Laplace principle in $\mathscr{E}$ with rate function $\I$ given by (\ref{rate}).
\end{theorem}
It should be noted that Hypothesis \ref{hyp1} (i) is a statement on the weak convergence of a certain family of random variables and is at the core of weak convergence approach to the study of large deviations. Hypothesis \ref{hyp1} (ii) says that the level sets of the rate function are compact.

\subsection{LDP for SCBF equations} Let us recall that the system (\ref{2.18a})
has an $\mathscr{F}_t$-adapted pathwise unique strong solution $\u^{\e}(t)$  in the Polish space $$\mathscr{E}=\C([0,T];\H)\cap\mathrm{L}^2(0,T;\V)\cap\mathrm{L}^{r+1}(0,T;\widetilde\L^{r+1}), \ \mathbb{P}\text{-a.s.}$$ The solution to the system (\ref{2.18a}) denoted by $\u^{\e}(\cdot)$ can be written as $\mathcal{G}^{\e}(\W (\cdot))$, for a Borel measurable function $\mathcal{G}^{\e} : \C([0, T ]; \H)\to \mathscr{E}$ (see Corollary 4.2, Chapter X, \cite{VF}, see \cite{BD1} also). Our main goal is to verify that such a $\mathcal{G}^{\e}$ satisfies Hypothesis \ref{hyp1}. Then, applying the Theorem \ref{BD}, the LDP for $\big\{\u^{\e} : \e > 0\big\}$ in $\mathscr{E}$ can be established. Let us now state and prove our main theorem.

\begin{theorem}\label{thm4.14}
	Under the Hypothesis \ref{hyp}, $\{\u^{\e}:\e>0\}$ obeys an LDP on $\C([0, T ]; \H ) \cap \mathrm{L}^2(0, T ; \V )\cap\mathrm{L}^{r+1}(0,T;\wi\L^{r+1})$ with the rate function $\I$.
\end{theorem}

The LDP for $\big\{\u^{\e} : \e> 0\big\}$ in $\mathscr{E}$ is proved in the following way.  We show the well-posedness of certain controlled deterministic and controlled stochastic  equations in $\mathscr{E}$. These results help us to prove the two main results on the compactness of the level sets and weak convergence of the stochastic controlled equation, which verifies the Hypothesis \ref{hyp1}.

\begin{theorem}\label{thm5.9}
	Let $h\in\mathrm{L}^2(0,T;\H_0)$ and $\Phi(\cdot,\cdot)$ satisfy the Hypothesis \ref{hyp}. Then the following deterministic control system:
	\begin{equation}\label{5.4y}
	\left\{
	\begin{aligned}
	\d\u_h(t)&=-\left[\mu\A\u_h(t)+\B(\u_h(t))+\beta\mathcal{C}(\u_h(t))-\Phi(t,\u_h(t))h(t)\right]\d t,\ \mathrm{ in }\  \V'+\wi\L^{\frac{r+1}{r}},\\
	\u_h(0)&=\u_0\in\H,
	\end{aligned}
	\right.
	\end{equation}
	has a \emph{unique weak solution}  in $\C([0,T];\H)\cap\mathrm{L}^2(0,T;\V)\cap\mathrm{L}^{r+1}(0,T;\wi\L^{r+1})$, and 
	\begin{align}\label{5.5y}
&	\sup_{h\in \mathcal{S}_M}\left\{\sup_{t\in[0,T]}\|\u_h(t)\|_{\H}^2+\mu\int_0^T\|\u_h(t)\|_{\V}^2\d t+\beta\int_0^T\|\u_h(t)\|_{\wi\L^{r+1}}^{r+1}\d t\right\}\nonumber\\&\leq \left(\|\u_0\|_{\H}^2+KM\right)e^{2(T+M)}. 
	\end{align}
\end{theorem}
\begin{proof}
The existence and uniqueness of weak solution in the Leray-Hopf sense (satisfying the energy equality) of the system (\ref{5.4y}) can be proved using the  monotonicty as well as demicontinuous properties of the linear and nonlinear operators and  the Minty-Browder technique as in Theorem 3.4, \cite{MTM7}. We need to show \eqref{5.5y} only. Taking the inner product with $\u_h(\cdot)$ to the first equation in \eqref{5.4y}, we find 
	\begin{align}\label{6.7}
&\frac{1}{2}\frac{\d}{\d t}\|\u_h(t)\|_{\H}^2+\mu\|\u_h(t)\|_{\V}^2+\beta\|\u(t)\|_{\wi\L^{r+1}}^{r+1}=\langle\Phi(t,\u_h(t))h(t),\u_h(t)\rangle. 
	\end{align}
since $\langle\B(\u_h),\u_h\rangle=0$. Using the Cauchy-Schwarz and H\"older inequalities, and Hypothesis \ref{hyp} (H.2), we get 
	\begin{align}\label{6.8}
	|\langle\Phi(\cdot,\u_h)h,\u_h\rangle|&\leq\|\Phi(\cdot,\u_h)h\|_{\H}\|\u_h\|_{\H}\leq\|\Phi(\cdot,\u_h)\|_{\mathcal{L}_{\Q}}\|h\|_0\|\u_h\|_{\H}\nonumber\\&\leq\frac{1}{2}\|\u_h\|_{\H}^2+\frac{1}{2}\|\Phi(\cdot,\u_h)\|_{\mathcal{L}_{\Q}}^2\|h\|_0^2\nonumber\\&\leq\frac{1}{2}\|\u_h\|_{\H}^2+\frac{K}{2}\|h\|_0^2+\frac{K}{2}\|\u_h\|^2_{\H}\|h\|_0^2. 
	\end{align}
	Substituting \eqref{6.8} in \eqref{6.7}, we obtain 
	\begin{align}\label{6.9}
&	\|\u_h(t)\|_{\H}^2+2\mu\int_0^t\|\u_h(s)\|_{\V}^2\d s+2\beta\int_0^t\|\u(s)\|_{\wi\L^{r+1}}^{r+1}\d s\nonumber\\&\leq\|\u_0\|_{\H}^2+K\int_0^t\|h(s)\|_0^2\d s+\int_0^t\|\u_h(s)\|_{\H}^2\d s+K\int_0^t\|\u_h(s)\|^2_{\H}\|h(s)\|_0^2\d s.
	\end{align}
	Applying Gronwall's inequality in \eqref{6.9}, we find 
	\begin{align*}
	\|\u_h(t)\|_{\H}^2\leq\left(\|\u_0\|_{\H}^2+K\int_0^T\|h(t)\|_0^2\d t\right)e^{T+K\int_0^T\|h(t)\|_0^2\d t},
	\end{align*}
	for all $t\in[0,T]$. Thus, taking $h\in \mathcal{S}_M$, we finally obtain \eqref{5.5y}. 
	\end{proof}

We are now in a position to verify the Hypothesis \ref{hyp1} (ii). 
\begin{theorem}[Compactness]\label{compact}
	Let $M <+\infty$ be a fixed positive number. Let $$\mathcal{K}_M :=\big\{  \u_h \in\C([0,T];\H)\cap \mathrm{L}^2(0,T;\V )\cap\mathrm{L}^{r+1}(0,T;\wi\L^{r+1}):h\in \mathcal{S}_M\big\},$$
	where $\u_h$ is the unique Leray-Hopf weak solution of the deterministic controlled equation (\ref{5.4y}), with $r\geq 3$ ($\beta\mu>1,$ for $r=3$) and $\u_h (0) = \u_0 \in\H$ in $\mathscr{E} = \C([0,T];\H)\cap\mathrm{L}^2(0,T;\V)\cap\mathrm{L}^{r+1}(0,T;\wi\L^{r+1})$. Then $\mathcal{K}_M$ is compact in $\mathscr{E}$.
\end{theorem}
\begin{proof}
	Let us consider a sequence $\{\u_{h_n}\}$ in $\mathcal{K}_M$, where $\u_{h_n}$ corresponds to the solution of (\ref{5.4y}) with control $h_n \in \mathcal{S}_M$ in place of $h$, that is, 
	\begin{equation}\label{5.20z}
	\left\{
	\begin{aligned}
	\d\u_{h_n}(t)&=-\left[\mu\A\u_{h_n}(t)+\B(\u_{h_n}(t))+\beta\mathcal{C}(\u_{h_n}(t))-\Phi(t,\u_{h_n}(t))h_n(t)\right]\d t,\\
	\u_{h_n}(0)&=\u_0\in\H.
	\end{aligned}
	\right.
	\end{equation}
	Then, by using the weak compactness of $\mathcal{S}_M$, there exists a subsequence of $\{h_n\}$, (still denoted by $\{h_n\}$), which converges weakly to $ h\in \mathcal{S}_M$ in $\mathrm{L}^2(0, T ; \H_0)$. Using the estimate (\ref{5.5y}), we obtain 
	\begin{equation}\label{5.29}
	\left\{
	\begin{aligned}
	\u_{h_n}&\xrightarrow{w^*}\u_h\text{ in }\mathrm{L}^{\infty}(0,T;\H),\\ 
	\u_{h_n}&\xrightarrow{w}\u_h\text{ in }\mathrm{L}^2(0,T;\V),\\
	\u_{h_n}&\xrightarrow{w}\u_h\text{ in }\mathrm{L}^{r+1}(0,T;\wi\L^{r+1}),
	\end{aligned}
	\right.
	\end{equation}
where $\u_h(\cdot)$ satisfies (\ref{5.4y}).	In order to prove that $\mathcal{K}_M$ is compact, we need to prove that $\u_{h_n}\to\u_h$ in $\mathscr{E}$ as $n\to\infty$. In other words, it is required to show that 
	\begin{align}\label{5.30z}
	\sup_{t\in[0,T]}\|\u_{h_n}(t)-\u_{h}(t)\|_{\H}^2+\int_0^T\|\u_{h_n}(t)-\u_{h}(t)\|_{\V}^2\d t+\int_0^T\|\u_{h_n}(t)-\u_h(t)\|_{\wi\L^{r+1}}^{r+1}\d t\to 0,
	\end{align}
	as $n\to\infty$. Recall that for the system (\ref{5.4y}), the energy estimate in (\ref{5.5y}) holds true. Let us now define $\w_{h_n}^h:=\u_{h_n}-\u_{h}$, so that $\w_{h_n}^h$ satisfies: 
	\begin{equation}\label{5.22z}
	\left\{
	\begin{aligned}
	\d\w_{h_n}^h(t)&=-\left[\mu\A\w_{h_n}^h(t)+\B(\u_{h_n}(t))-\B(\u_h(t))+\beta(\mathcal{C}(\u_{h_n}(t))-\mathcal{C}(\u_h(t)))\right]\d t\\&\quad +\left[\Phi(t,\u_{h_n}(t))h_n(t)-\Phi(t,\u_{h}(t))h(t)\right]\d t,\\
	\w_{h_n}^h(0)&=\mathbf{0}.
	\end{aligned}
	\right.\end{equation}
	Taking the inner product with $\w_{h_n}^h(\cdot)$ to the system (\ref{5.22z}), we get 
	\begin{align}\label{5.23z}
	&\|\w_{h_n}^h(t)\|_{\H}^2+2\mu\int_0^t\|\w_{h_n}^h(s)\|_{\V}^2\d s+2\beta\int_0^t\langle\mathcal{C}(\u_{h_n}(s))-\mathcal{C}(\u_h(s)),\u_{h_n}(s)-\u_h(s)\rangle\d s\nonumber\\&=-2\int_0^t\langle\B(\u_{h_n}(s))-\B(\u_h(s)),\w_{h_n}^h(s)\rangle\d s\nonumber\\&\quad+2\int_0^t(\Phi(s,\u_{h_n}(s))h_n(s)-\Phi(s,\u_{h}(s))h(s),\w_{h_n}^h(s))\d s.
	\end{align}
From \eqref{2.23}, we easily have 
\begin{align}\label{6.27}
\beta	\langle\mathcal{C}(\u_{h_n})-\mathcal{C}(\u_{h}),\u_{h_n}-\u_{h}\rangle \geq \frac{\beta}{2}\||\u_{h_n}|^{\frac{r-1}{2}}(\u_{h_n}-\u_{h})\|_{\H}^2+\frac{\beta}{2}\||\u_{h}|^{\frac{r-1}{2}}(\u_{h_n}-\u_{h})\|_{\H}^2. 
\end{align}
Note that $\langle\B(\u_{h_n},\u_{h_n}-\u_{h}),\u_{h_n}-\u_{h}\rangle=0$ and it implies that
\begin{equation}\label{6.28}
\begin{aligned}
\langle \B(\u_{h_n})-\B(\u_{h}),\u_{h_n}-\u_{h}\rangle &=\langle\B(\u_{h_n},\u_{h_n}-\u_{h}),\u_{h_n}-\u_{h}\rangle +\langle \B(\u_{h_n}-\u_{h},\u_{h}),\u_{h_n}-\u_{h}\rangle \nonumber\\&=\langle\B(\u_{h_n}-\u_{h},\u_{h}),\u_{h_n}-\u_{h}\rangle=-\langle\B(\u_{h_n}-\u_{h},\u_{h_n}-\u_{h}),\u_{h}\rangle.
\end{aligned}
\end{equation} 
Using H\"older's and Young's inequalities, we estimate $|\langle\B(\u_{h_n}-\u_{h},\u_{h_n}-\u_{h}),\u_{h}\rangle|$ as  
\begin{align}\label{6.29}
|\langle\B(\u_{h_n}-\u_{h},\u_{h_n}-\u_{h}),\u_{h}\rangle|&\leq\|\u_{h_n}-\u_{h}\|_{\V}\|\u_{h}(\u_{h_n}-\u_{h})\|_{\H}\nonumber\\&\leq\frac{\mu }{2}\|\u_{h_n}-\u_{h}\|_{\V}^2+\frac{1}{2\mu }\|\u_{h}(\u_{h_n}-\u_{h})\|_{\H}^2.
\end{align}
We take the term $\|\u_{h}(\u_{h_n}-\u_{h})\|_{\H}^2$ from \eqref{6.29} and use H\"older's and Young's inequalities to estimate it as (see \cite{KWH} also)
\begin{align}\label{6.30}
&\int_{\mathcal{O}}|\u_{h}(x)|^2|\u_{h_n}(x)-\u_{h}(x)|^2\d x\nonumber\\&=\int_{\mathcal{O}}|\u_{h}(x)|^2|\u_{h_n}(x)-\u_{h}(x)|^{\frac{4}{r-1}}|\u_{h_n}(x)-\u_{h}(x)|^{\frac{2(r-3)}{r-1}}\d x\nonumber\\&\leq\left(\int_{\mathcal{O}}|\u_{h}(x)|^{r-1}|\u_{h_n}(x)-\u_{h}(x)|^2\d x\right)^{\frac{2}{r-1}}\left(\int_{\mathcal{O}}|\u_{h_n}(x)-\u_{h}(x)|^2\d x\right)^{\frac{r-3}{r-1}}\nonumber\\&\leq\frac{\beta\mu }{2}\left(\int_{\mathcal{O}}|\u_{h}(x)|^{r-1}|\u_{h_n}(x)-\u_{h}(x)|^2\d x\right)\nonumber\\&\quad+\frac{r-3}{r-1}\left(\frac{4}{\beta\mu (r-1)}\right)^{\frac{2}{r-3}}\left(\int_{\mathcal{O}}|\u_{h_n}(x)-\u_{h}(x)|^2\d x\right),
\end{align}
for $r>3$. Combining \eqref{6.27} and \eqref{6.30}, we find 
\begin{align}\label{630}
&\beta\langle\mathcal{C}(\u_{h_n})-\mathcal{C}(\u_{h}),\u_{h_n}-\u_{h}\rangle+\langle\B(\u_{h_n}-\u_{h},\u_{h_n}-\u_{h}),\u_{h}\rangle\nonumber\\&\geq\frac{\beta}{4}\||\u_{h_n}|^{\frac{r-1}{2}}(\u_{h_n}-\u_{h})\|_{\H}^2+\frac{\beta}{2}\||\u_{h}|^{\frac{r-1}{2}}(\u_{h_n}-\u_{h})\|_{\H}^2\nonumber\\&\quad-\frac{r-3}{2\mu(r-1)}\left(\frac{4}{\beta\mu (r-1)}\right)^{\frac{2}{r-3}}\left(\int_{\mathcal{O}}|\u_{h_n}(x)-\u_{h}(x)|^2\d x\right)-\frac{\mu}{2} \|\u_{h_n}-\u_{h}\|_{\V}^2.
\end{align}
Using \eqref{a215},  we have 
\begin{align*}
\frac{2^{2-r}\beta}{4}\|\u_{h_n}-\u_{h}\|_{\wi\L^{r+1}}^{r+1}\leq\frac{\beta}{4}\||\u_{h_n}|^{\frac{r-1}{2}}(\u_{h_n}-\u_h)\|_{\L^2}^2+\frac{\beta}{4}\||\u_{h}|^{\frac{r-1}{2}}(\u_{h_n}-\u_h)\|_{\L^2}^2.
\end{align*}
Thus, from \eqref{630}, it is immediate that 
\begin{align}\label{622}
&\beta\langle\mathcal{C}(\u_{h_n})-\mathcal{C}(\u_{h}),\u_{h_n}-\u_{h}\rangle+\langle\B(\u_{h_n}-\u_{h},\u_{h_n}-\u_{h}),\u_{h}\rangle\nonumber\\&\geq \frac{\beta}{2^r}\|\u_{h_n}-\u_{h}\|_{\wi\L^{r+1}}^{r+1}-\frac{\wi\eta}{2}\|\u_{h_n}-\u_h\|_{\H}^2-\frac{\mu}{2} \|\u_{h_n}-\u_{h}\|_{\V}^2,
\end{align}
where \begin{align}\label{623}\wi\eta=\frac{r-3}{\mu(r-1)}\left(\frac{4}{\beta\mu (r-1)}\right)^{\frac{2}{r-3}}.\end{align}
For $r> 3$, from \eqref{5.23z}, one can easily get 
\begin{align}\label{633}
	&\|\w_{h_n}^h(t)\|_{\H}^2+\mu\int_0^t\|\w_{h_n}^h(s)\|_{\V}^2\d s+\frac{\beta}{2^{r-1}}\int_0^t\|\w_{h_n}^h(s)\|_{\wi\L^{r+1}}^{r+1}\d s\nonumber\\&\leq\wi\eta\int_0^t\|\w_{h_n}^h(s)\|_{\H}^2\d s+ 2\int_0^t((\Phi(s,\u_{h_n}(s))-\Phi(s,\u_{h}(s))h_n(s),\w_{h_n}^h(s))\d s\nonumber\\&\quad+ 2\int_0^t(\Phi(s,\u_{h}(s))(h_n(s)-h(s)),\w_{h_n}^h(s))\d s\nonumber\\&=:\wi\eta\int_0^t\|\w_{h_n}^h(s)\|_{\H}^2\d s+I_1+I_2.
\end{align}
Using Cauchy-Schwarz, H\"older's and Young's inequalities and Hypothesis \ref{hyp} (H.3), we estimate $I_1$ as  
	\begin{align}\label{634}
I_1	&\leq 2\int_0^t|(\Phi(s,\u_{h_n}(s))-\Phi(s,\u_{h}(s))h_n(s),\w_{h_n}^h(s))|\d s\nonumber\\&\leq 2\int_0^t\|\Phi(s,\u_{h_n}(s))-\Phi(s,\u_{h}(s)\|_{\mathcal{L}_{\Q}}\|h_n(s)\|_{0}\|\w_{h_n}^h(s)\|_{\H}\d s\nonumber\\&\leq 2L\int_0^t\|h_n(s)\|_0\|\w_{h_n}^h(s)\|_{\H}^2\d s\nonumber\\&\leq L\int_0^t\|\w_{h_n}^h(s)\|_{\H}^2\d s+L\int_0^t\|h_n(s)\|_0^2\|\w_{h_n}^h(s)\|_{\H}^2\d s.
	\end{align}
	Making use of the Cauchy-Schwarz and Young inequalities, we estimate $I_2$ as 
	\begin{align}\label{635}
	I_2&\leq 2\int_0^t\|\Phi(s,\u_{h}(s))(h_n(s)-h(s))\|_{\H}\|\w_{h_n}^h(s)\|_{\H}\d s\nonumber\\&\leq \int_0^t\|\w_{h_n}^h(s)\|_{\H}^2\d s+\int_0^t\|\Phi(s,\u_{h}(s))(h_n(s)-h(s))\|_{\H}^2\d s.
	\end{align}
	Combining \eqref{634} and \eqref{635} and substituting it in \eqref{633}, we find 
	\begin{align}\label{636}
	&\|\w_{h_n}^h(t)\|_{\H}^2+\mu\int_0^t\|\w_{h_n}^h(s)\|_{\V}^2\d s+\frac{\beta}{2^{r-1}}\int_0^t\|\w_{h_n}^h(s)\|_{\wi\L^{r+1}}^{r+1}\d s\nonumber\\&\leq (\wi\eta+1+L)\int_0^t\|\w_{h_n}^h(s)\|_{\H}^2\d s+L\int_0^t\|h_n(s)\|_0^2\|\w_{h_n}^h(s)\|_{\H}^2\d s\nonumber\\&\quad+\int_0^t\|\Phi(s,\u_{h}(s))(h_n(s)-h(s))\|_{\H}^2\d s.
	\end{align}
	A application of Gronwall's inequality in \eqref{636} yields 
	\begin{align}\label{637}
&\sup_{t\in[0,T]}	\|\w_{h_n}^h(t)\|_{\H}^2+\mu\int_0^T\|\w_{h_n}^h(t)\|_{\V}^2\d t+\frac{\beta}{2^{r-1}}\int_0^T\|\w_{h_n}^h(t)\|_{\wi\L^{r+1}}^{r+1}\d t\nonumber\\&\leq \left(\int_0^T\|\Phi(s,\u_{h}(s))(h_n(s)-h(s))\|_{\H}^2\d s\right)e^{(\wi\eta+1+L)T}\exp\left(L\int_0^T\|h_n(t)\|_0^2\d t\right)\nonumber\\&\leq \left(\int_0^T\|\Phi(s,\u_{h}(s))(h_n(s)-h(s))\|_{\H}^2\d s\right)e^{(\wi\eta+1+L(1+M))T},
	\end{align}
	since $\{h_n\}\in \mathcal{S}_M$. 
	It should be noted that the operator $\Phi(\cdot,\cdot)\Q^{\frac{1}{2}}$ is Hilbert-Schmidt in $\H$, and hence it is a compact operator on $\H$. Furthermore, we know that compact operator maps weakly convergent sequences into strongly convergent sequences. Since $\{h_n\}$ converges weakly to $ h\in \mathcal{S}_M$ in $\mathrm{L}^2(0, T ; \H_0)$, we infer that $$\int_0^T\|\Phi(s,\u_{h}(s))(h_n(s)-h(s))\|_{\H}^2\d s\to 0 \ \text{ as } \ n\to\infty.$$ Thus, from \eqref{637}, we obtain 
	\begin{align}\label{638}
	&\sup_{t\in[0,T]}	\|\w_{h_n}^h(t)\|_{\H}^2+\mu\int_0^T\|\w_{h_n}^h(t)\|_{\V}^2\d t+\frac{\beta}{2^{r-1}}\int_0^T\|\w_{h_n}^h(t)\|_{\wi\L^{r+1}}^{r+1}\d t\to 0\  \text{ as } n\ \to \infty,
	\end{align}
	which concludes the proof for $r>3$.

	For $r=3$, from \eqref{2.23}, we have 
	\begin{align}\label{6.33}
	\beta\langle\mathcal{C}(\u_{h_n})-\mathcal{C}(\u_{h}),\u_{h_n}-\u_{h}\rangle\geq\frac{\beta}{2}\|\u_{h_n}(\u_{h_n}-\u_{h})\|_{\H}^2+\frac{\beta}{2}\|\u_{h}(\u_{h_n}-\u_{h})\|_{\H}^2. 
	\end{align}
	We estimate $|\langle\B(\u_{h_n}-\u_{h},\u_{h_n}-\u_{h}),\u_{h}\rangle|$ using H\"older's and Young's inequalities as 
	\begin{align}\label{6.34}
	|\langle\B(\u_{h_n}-\u_{h},\u_{h_n}-\u_{h}),\u_{h}\rangle|&\leq\|\u_{h}(\u_{h_n}-\u_{h})\|_{\H}\|\u_{h_n}-\u_{h}\|_{\V} \nonumber\\&\leq\frac{\mu}{2} \|\u_{h_n}-\u_{h}\|_{\V}^2+\frac{1}{2\mu }\|\u_{h}(\u_{h_n}-\u_{h})\|_{\H}^2.
	\end{align}
	Combining \eqref{6.33} and \eqref{6.34}, we obtain 
	\begin{align}\label{632}
	&\beta\langle\mathcal{C}(\u_{h_n})-\mathcal{C}(\u_{h}),\u_{h_n}-\u_{h}\rangle+\langle\B(\u_{h_n}-\u_{h},\u_{h_n}-\u_{h}),\u_{h}\rangle\nonumber\\&\geq \frac{\beta}{2}\|\u_{h_n}(\u_{h_n}-\u_{h})\|_{\H}^2+\frac{1}{2}\left(\beta-\frac{1}{\mu}\right)\|\u_{h}(\u_{h_n}-\u_{h})\|_{\H}^2-\frac{\mu}{2} \|\u_{h_n}-\u_{h}\|_{\V}^2\nonumber\\&\geq \frac{1}{2}\left(\beta-\frac{1}{\mu}\right)\|\u_{h_n}-\u_{h}\|_{\wi\L^4}^4-\frac{\mu}{2} \|\u_{h_n}-\u_{h}\|_{\V}^2. 
	\end{align}
	Thus,  we infer that 
		\begin{align}\label{642}
	&\|\w_{h_n}^h(t)\|_{\H}^2+\mu\int_0^t\|\w_{h_n}^h(s)\|_{\V}^2\d s+\left(\beta-\frac{1}{\mu}\right)\int_0^t\|\w_{h_n}^h(s)\|_{\wi\L^{4}}^{4}\d s\nonumber\\&\leq (1+L)\int_0^t\|\w_{h_n}^h(s)\|_{\H}^2\d s+L\int_0^t\|h_n(s)\|_0^2\|\w_{h_n}^h(s)\|_{\H}^2\d s\nonumber\\&\quad+\int_0^t\|\Phi(s,\u_{h}(s))(h_n(s)-h(s))\|_{\H}^2\d s.
	\end{align}
	Hence, for $\beta\mu >1$, arguing similarly as in the case of $r>3$, we finally obtain the required result. 
\end{proof}

Let us now verify Hypothesis \ref{hyp1} (i). We first establish the existence and uniqueness result of the following stochastic controlled SCBF equations. 
\begin{theorem}\label{thm5.10}
	For any $h\in\mathcal{A}_M$, $0<M<+\infty$, under the Hypothesis \ref{hyp}, the stochastic control problem:
	\begin{equation}\label{5.4z}
	\left\{
	\begin{aligned}
	\d\u^{\e}_h(t)&=-\left[\mu\A\u^{\e}_h(t)+\B(\u^{\e}_h(t))+\beta\mathcal{C}(\u_{\e}^h(t))-\Phi(t,\u^{\e}_h(t))h(t)\right]\d t\\&\quad+\sqrt{\e}\Phi(t,\u^{\e}_h(t))\d\W(t),\\
	\u^{\e}_h(0)&=\u_0\in\H,
	\end{aligned}
	\right.
	\end{equation}
	has a \emph{pathwise unique  strong solution} in $\mathrm{L}^2(\Omega;\mathscr{E})$, where $\mathscr{E}=\C([0,T];\H)\cap\mathrm{L}^2(0,T;\V)\cap\mathrm{L}^{r+1}(0,T;\wi\L^{r+1})$ with $\mathscr{F}_t$-adapted paths in $\mathscr{E}$, $\mathbb{P}$-a.s. Furthermore, $\u^{\e}_h(\cdot)$ satisfies: 
	\begin{align}\label{5.5z}
	&\sup_{0<\e\leq \e_0,\ h\in \mathcal{A}_M}\E\left[\sup_{t\in[0,T]}\|\u^{\e}_h(t)\|_{\H}^2+2\wi\mu\int_0^T\|\u^{\e}_h(t)\|_{\V}^2\d t+2\beta\int_0^{\T}\|\u^{\e}_h(t)\|_{\wi\L^{r+1}}^{r+1}\d t\right]\nonumber\\&\qquad\qquad\qquad\leq \left(2\|\u_0\|_{\H}^2+2K(4M+13K\e_0)T\right)e^{8MKT},
	\end{align}
	where $\wi\mu =\mu-\frac{13K\e}{\uplambda_1}\geq 0$ and $\e_0=\frac{\mu\uplambda_1}{13K}$.
\end{theorem}
\begin{proof}
	The existence and uniqueness of pathwise strong solution satisfying the energy equality to the system (\ref{5.4z}) can be obtained similarly as in Theorem 3.7, \cite{MTM8}, by using the  monotonicty as well as demicontinuous properties of the linear and nonlinear operators and a stochastic generalization of the Minty-Browder technique.  
	
	Let us define a sequence of stopping times to be 
	\begin{align*}
	\tau_N:=\inf_{t\geq 0}\left\{t:\|\u^{\e}_h(t)\|_{\H}>N\right\}.
	\end{align*}
	Since $\u^{\e}_h(\cdot)$ satisfies It\^o's formula  (Theorem 3.7, \cite{MTM8}), we obtain 
	\begin{align}\label{5.7z}
	&\|\u^{\e}_h(\t)\|_{\H}^2+2\mu\int_0^{\t}\|\u^{\e}_h(s)\|_{\V}^2\d s+2\beta\int_0^t\|\u^{\e}_h(s)\|_{\wi\L^{r+1}}^{r+1}\d s\nonumber\\&=\|\u_0\|_{\H}^2+ 2\int_0^{\t}(\Phi(s,\u^{\e}_h(s))h(s),\u^{\e}_h(s))\d s +\e\int_0^{\t}\|\Phi(s,\u^{\e}_h(s))\|_{\mathcal{L}_{\Q}}^2\d s\nonumber\\&\quad+2\sqrt{\e}\int_0^{\t}(\Phi(s,\u^{\e}_h(s))\d\W(s),\u^{\e}_h(s)).
	\end{align}
	Taking supremum from $0$ to $T$ and then taking expectation in (\ref{5.7z}), we get 
	\begin{align}\label{5.8z}
	&\E\left[\sup_{t\in[0,\T]}\|\u^{\e}_h(t)\|_{\H}^2+2\mu\int_0^{\T}\|\u^{\e}_h(t)\|_{\V}^2\d t+2\beta\int_0^{\T}\|\u^{\e}_h(t)\|_{\wi\L^{r+1}}^{r+1}\d t\right]\nonumber\\&\leq \|\u_0\|_{\H}^2+2\E\left[\int_0^{\T}|(\Phi(t,\u^{\e}_h(t))h(t),\u^{\e}_h(t))|\d t\right]+\e\E\left[\int_0^{\T}\|\Phi(t,\u^{\e}_h(t))\|_{\mathcal{L}_{\Q}}^2\d t\right]\nonumber\\&\quad+2\E\left[\sup_{t\in[0,\T]}\left|\int_0^{\t}(\sqrt{\e}\Phi(s,\u^{\e}_h(s))\d\W(s),\u^{\e}_h(s))\right|\right]\nonumber\\&=:\|\u_0\|_{\H}^2+\e\E\left[\int_0^{\T}\|\Phi(t,\u^{\e}_h(t))\|_{\mathcal{L}_{\Q}}^2\d t\right]+\sum_{k=1}^2I_k.
	\end{align}
	We estimate $I_1$ using Cauchy-Schwarz, H\"older's and Young's inequalities as 
	\begin{align}\label{5.9z}
	I_1&\leq \E\left[\int_0^{\T}\|\Phi(t,\u^{\e}_h(t))\|_{\mathcal{L}_{\Q}}\|h(t)\|_{0}\|\u^{\e}_h(t)\|_{\H}\d t\right]\nonumber\\&\leq \frac{1}{8}\E\left[\sup_{t\in[0,\T]}\|\u^{\e}_h(t)\|_{\H}^2\right]+2\E\left(\int_0^{\T}\|\Phi(t,\u^{\e}_h(t))\|_{\mathcal{L}_{\Q}}\|h(t)\|_{0}\d t\right)^2\nonumber\\&\leq\frac{1}{8}\E\left[\sup_{t\in[0,\T]}\|\u^{\e}_h(t)\|_{\H}^2\right]+2 \E\left[\left(\int_0^{\T}\|\Phi(t,\u^{\e}_h(t))\|_{\mathcal{L}_{\Q}}^2\d t\right)\left(\int_0^{\T}\|h(t)\|_{0}^2\d t\right)\right]\nonumber\\&\leq \frac{1}{8}\E\left[\sup_{t\in[0,\T]}\|\u^{\e}_h(t)\|_{\H}^2\right]+2M \E\left[\int_0^{\T}\|\Phi(t,\u^{\e}_h(t))\|_{\mathcal{L}_{\Q}}^2\d t\right],
	\end{align}
	where we used the fact that $h\in\mathcal{A}_M$.
	Using  Burkholder-Davis-Gundy (\cite{DLB,BD}), H\"older's and Young's inequalities, we estimate $I_2$ as
	\begin{align}\label{5.10z}
	I_2&\leq\sqrt{3\e}\E\left[\int_0^{\T}\|\Phi(t,\u^{\e}_h(t))\|_{\mathcal{L}_{\Q}}^2\|\u^{\e}_h(s)\|_{\H}^2\d s\right]^{1/2}\nonumber\\&\leq\sqrt{3\e}\E\left[\|\u^{\e}_h(s)\|_{\H}\left(\int_0^{\T}\|\Phi(t,\u^{\e}_h(t))\|_{\mathcal{L}_{\Q}}^2\d s\right)^{1/2}\right]\nonumber\\&\leq \frac{1}{8}\E\left[\sup_{t\in[0,\T]}\|\u^{\e}_h(t)\|_{\H}^2\right]+6\e\E\left[\int_0^{\T}\|\Phi(t,\u^{\e}_h(t))\|_{\mathcal{L}_{\Q}}^2\d t\right].
	\end{align}
	Substituting (\ref{5.9z}) and (\ref{5.10z}) in (\ref{5.8z}), we obtain 
	\begin{align}\label{5.11z}
	&\E\left[\sup_{t\in[0,\T]}\|\u^{\e}_h(t)\|_{\H}^2+2\mu\int_0^{\T}\|\u^{\e}_h(t)\|_{\V}^2\d t+2\beta\int_0^{\T}\|\u^{\e}_h(t)\|_{\wi\L^{r+1}}^{r+1}\d t\right]\nonumber\\&\leq 2\|\u_0\|_{\H}^2+K\left(8M+26\e\right)\E\left[\int_0^{\T}\left(1+\|\u^{\e}_h(t)\|_{\H}^2\right)\d t\right],
	\end{align}
	where we used the Hypothesis \ref{hyp} (H.2). 
	Thus,  from (\ref{5.11z}), we get 
	\begin{align}\label{5.11q}
	&\E\left[\sup_{t\in[0,\T]}\|\u^{\e}_h(t)\|_{\H}^2+2\left(\mu-\frac{13K\e}{\uplambda_1}\right)\int_0^{\T}\|\u^{\e}_h(t)\|_{\V}^2\d t+2\beta\int_0^{\T}\|\u^{\e}_h(t)\|_{\wi\L^{r+1}}^{r+1}\d t\right]\nonumber\\&\leq 2\|\u_0\|_{\H}^2+2K(4M+13K\e)T+8MK\E\left[\int_0^{\T}\|\u^{\e}_{h}(t)\|_{\H}^2\d t\right].
	\end{align}
	For $0<\e\leq\e_0=\frac{\mu\uplambda_1}{13K}$,	an application of Gronwall's inequality in (\ref{5.11q}) yields 
	\begin{align}\label{5.12z}
	\E\left[\sup_{t\in[0,\T]}\|\u^{\e}_h(t)\|_{\H}^2\right]\leq\left(2\|\u_0\|_{\H}^2+2K(4M+13K\e_0)T\right)e^{8MKT}.
	\end{align}
	Passing  $N\to\infty$ in (\ref{5.12z}), using the monotone convergence theorem and then substituting it in (\ref{5.11z}), we finally obtain  (\ref{5.5z}).
\end{proof}

For  all $h\in\mathrm{L}^2(0,T;\H_0)$, let $\u_h(\cdot)$ be the unique weak solution of the deterministic control equation:
\begin{equation}\label{443}
\left\{
\begin{aligned}
\d\u_h(t)&=-\left[\mu\A\u_h(t)+\B(\u_h(t))+\beta\mathcal{C}(\u_h(t))-\Phi(t,\u_h(t))h(t)\right]\d t,\\
\u_h(0)&=\u_0\in\H.
\end{aligned}
\right.
\end{equation}
Note that  $\int_0^{\cdot}h(s)\d s\in\C([0,T];\H)$. Let us define $\mathcal{G}^0 : \C([0,T] ; \H) \to \mathscr{E}$  by
$$\u_h(\cdot)=\mathcal{G}^0\left(\int_0^{\cdot}h(s)\d s\right), \text{ for some }h\in\mathrm{L}^2(0,T;\H_0).$$ 

Let $\u_{h^{\e}}^{\e}(\cdot)$ solve the following stochastic control system:
\begin{equation}\label{442}
\left\{
\begin{aligned}
\d\u^{\e}_{h^{\e}}(t)&=-\left[\mu\A\u^{\e}_{h^{\e}}(t)+\B(\u^{\e}_{h^{\e}}(t))+\beta\mathcal{C}(\u^{\e}_{h^{\e}}(t))-\Phi(t,\u^{\e}_{h^{\e}}(t))h^{\e}(t)\right]\d t\\&\quad+\sqrt{\e}\Phi(t,\u^{\e}_{h^{\e}}(t))\d\W(t),\\
\u^{\e}_{h^{\e}}(0)&=\u_0\in\H.
\end{aligned}
\right.
\end{equation}
Using Theorem \ref{thm5.10},  the system \eqref{442} has a pathwise unique strong solution $\u^{\e}_{h^{\e}}(\cdot)$ with paths in $$\mathscr{E}=\C([0,T];\H)\cap\mathrm{L}^2(0,T;\V)\cap\mathrm{L}^{r+1}(0,T;\wi\L^{r+1}), \ \mathbb{P}\text{-a.s.}$$ Since, 
\begin{align*}
\mathbb{E}\left(\exp\left\{-\frac{1}{\sqrt{\e}}\int_0^T(h^{\e}(t),\d\W(t))_0-\frac{1}{2\e}\int_0^T\|h^{\e}(t)\|_0^2\d t \right\}\right)=1,
\end{align*}
the measure $\widehat{\mathbb{P}}$ defined by 
\begin{align*}
\d\widehat{\mathbb{P}}(\omega)=\exp\left\{-\frac{1}{\sqrt{\e}}\int_0^T(h^{\e}(t),\d\W(t))_0-\frac{1}{2\e}\int_0^T\|h^{\e}(t)\|_0^2\d t\right\}\d{\mathbb{P}}(\omega)
\end{align*}
is a probability measure on $(\Omega,\mathscr{F},\mathbb{P})$. Moreover, $\widehat{\mathbb{P}}(\omega)$ is mutually absolutely continuous with respect to $\mathbb{P}(\omega)$ and by using Girsanov's theorem (Theorem 10.14, \cite{DZ}),  we have the process 
\begin{align*}
\widehat{\W}(t):=\W(t)+\frac{1}{\sqrt{\e}}\int_0^th^{\e}(s)\d s, \ t\in[0,T],
\end{align*}
is a $\Q$-Wiener process with respect to $\{\mathscr{F}_t\}_{t\geq 0}$ on the probability space $(\Omega,\mathscr{F},\widehat{\mathbb{P}})$. Thus, we know that (\cite{BD,MRBS}) $$\u^{\e}_{h^{\e}}(\cdot)= \mathcal{G}^{\e}\left(\W(\cdot) +\frac{1}{\sqrt{\e}}\int_0^{\cdot}h^{\e}(s)\d s\right)$$ is the unique strong solution of \eqref{2.18a} with $\W(\cdot)$ replaced by $\widehat{\W}(\cdot)$, on $(\Omega,\mathscr{F},\{\mathscr{F}_t\}_{t\geq 0},\widehat{\mathbb{P}})$. Moreover, the system \eqref{2.18a} with $\widehat{\W}(\cdot)$ is same as the system \eqref{442}, and since $\widehat{\mathbb{P}}$ and $\mathbb{P}$ are mutually absolutely continuous, we further find that $\u^{\e}_{h^{\e}}(\cdot)$ is the unique strong solution of \eqref{442} on $(\Omega,\mathscr{F},\{\mathscr{F}_t\}_{t\geq 0},{\mathbb{P}})$. 

The well known Skorokhod's representation theorem (see \cite{AVS}) states that if $\mu_n, n=1,2,\ldots,$ and $\mu_0$ are  probability measures on complete separable metric space (Polish space)  such that  $\mu_n\xrightarrow{w}\mu$, as $ n \to \infty$, then there exist a probability space $(\widetilde{\Omega},\widetilde{\mathscr{F}},\widetilde{\mathbb{P}})$ and a sequence of measurable random elements $\mathrm{X}_n$ such that $\mathrm{X}_n\to\mathrm{X},$ $\widetilde{\mathbb{P}}$-a.s., and $\mathrm{X}_n$ has the distribution function $\mu_n$, $n = 0, 1, 2, \ldots $  ($\mathrm{X}_n\sim\mu_n$), that is, the law of $\mathrm{X}_n$ is $\mu_n$. We use Skorokhod's representation theorem in the next theorem. 
\begin{theorem}[Weak convergence]\label{weak}
	Let $\big\{h^{\e} : \e > 0\big\}\subset \mathcal{A}_M$ converges in distribution to $h$ with respect to the weak topology on $\mathrm{L}^2(0,T;\H_0)$. Then $ \mathcal{G}^{\e}\left(\W(\cdot) +\frac{1}{\sqrt{\e}}\int_0^{\cdot}h^{\e}(s)\d s\right)$ converges in distribution to $\mathcal{G}^0\left(\int_0^{\cdot}h(s)\d s\right)$ in $\mathscr{E}$, as $\e\to0$.
\end{theorem}
\begin{proof}
Let $h_{\e}$ converge to $h$ in distribution as random elements taking values in $\mathcal{S}_M,$ where $\mathcal{S}_M$ is equipped with the weak topology. 
Since $\mathcal{A}_M$ is  Polish (see section \ref{sec4.2} and \cite{BD2}) and $\big\{h^{\e} : \e > 0\big\}\subset \mathcal{A}_M$ converges in distribution to $h$ with respect to the weak topology on $\mathrm{L}^2(0,T;\H_0)$, the Skorokhod representation theorem can be used  to construct a probability space $(\widetilde{\Omega},\widetilde{\mathscr{F}},(\widetilde{\mathscr{F}}_t)_{0\leq t\leq T},\widetilde{\mathbb{P}})$  and processes $(\wi h^{\e},\wi h, \wi\W^{\e})$ such that the distribution of $(\wi h^{\e}, \wi\W^{\e})$
	is same as that of $(h^{\e}, \W^{\e})$, and $\wi h^{\e}\to \wi h$,  $\widetilde{\mathbb{P}}$-a.s., in the weak topology of $\mathcal{S}_M$. Thus
	$\int_0^{t} \wi h^{\e}(s)\d s\to\int_0^{t} \wi h(s)\d s$ weakly in $\H_0$, $\widetilde{\mathbb{P}}$-a.s., for all $t\in[0,T]$. In the following sequel, without loss of 
	generality, we write $(\Omega,\mathscr{F},\mathbb{P})$ as the probability space and $(h^{\e},h,\W)$ as processes, though strictly speaking, one should write $(\widetilde{\Omega},\widetilde{\mathscr{F}},\widetilde{\mathbb{P}})$ and $(\wi h^{\e}, \wi h, \widetilde{\W}^{\e})$, respectively for probability space and processes.	
	
	Let us define $\w^{\e}_{h^{\e}}:=\u^{\e}_{h^{\e}}-\u_h$, where $\w^{\e}_{h^{\e}}(\cdot)$ satisfies:
	\begin{equation}
	\left\{
	\begin{aligned}
	\d\w^{\e}_{h^{\e}}(t)&=-\left[\mu\A\w^{\e}_{h^{\e}}(t)+\B(\u^{\e}_{h^{\e}}(t))-\B(\u_h(t))+\beta(\mathcal{C}(\u^{\e}_{h^{\e}}(t))-\mathcal{C}(\u_h(t)))\right.\\&\qquad\left.-\Phi(t,\u^{\e}_{h^{\e}}(t))h^{\e}(t)+\Phi(t,\u_h(t))h(t)\right]\d t+\sqrt{\e}\Phi(t,\u^{\e}_{h^{\e}}(t))\d\W(t),\\
	\w^{\e}_{h^{\e}}(0)&=\mathbf{0},
	\end{aligned}
	\right.
	\end{equation}
	and 
	\begin{align*}
&\widetilde{\mathbb{P}}(\widetilde{\w}^{\e}_{h^{\e}}\in\mathrm{C}([0,T];\H)\cap\mathrm{L}^2(0,T;\V)\cap\mathrm{L}^{r+1}(0,T;\wi\L^{r+1}))\nonumber\\&=\mathbb{P}(\w^{\e}_{h^{\e}}\in\mathrm{C}([0,T];\H)\cap\mathrm{L}^2(0,T;\V)\cap\mathrm{L}^{r+1}(0,T;\wi\L^{r+1}))=1,
	\end{align*}
	where $\widetilde{\w}^{\e}_{h^{\e}}=\widetilde{\u}^{\e}_{h^{\e}}-\widetilde{\u}_h$. We need to prove that
	\begin{align*}
	\sup_{t\in[0,T]}\|\w^{\e}_{h^{\e}}(t)\|_{\H}^2+\int_0^T\|\w^{\e}_{h^{\e}}(t)\|_{\V}^2\d t+\int_0^T\|\w^{\e}_{h^{\e}}(t)\|_{\wi\L^{r+1}}^{r+1}\d t\to 0,
	\end{align*} 
	in probability as $\e\to 0$. Using It\^o's formula, we obtain 
	\begin{align}\label{446}
	&\|\w^{\e}_{h^{\e}}(t)\|_{\H}^2+2\mu\int_0^t\|\w^{\e}_{h^{\e}}(s)\|_{\V}^2\d s+2\beta\int_0^t\langle\mathcal{C}(\u^{\e}_{h^{\e}}(s))-\mathcal{C}(\u_h(s)),\w^{\e}_{h^{\e}}(s)\rangle\d s\nonumber\\&=-2\int_0^t\langle \B(\u^{\e}_{h^{\e}}(s))-\B(\u_h(s)),\w^{\e}_{h^{\e}}(s)\rangle \d s\nonumber\\&\quad +2\int_0^t(\Phi(s,\u^{\e}_{h^{\e}}(s))h^{\e}(s)-\Phi(s,\u_h(s))h(s),\w^{\e}_{h^{\e}}(s))\d s+\e\int_0^t\|\Phi(s,\u^{\e}_{h_{\e}}(s))\|_{\mathcal{L}_{\Q}}^2\d s\nonumber\\&\quad +2\int_0^t(\sqrt{\e}\Phi(s,\u^{\e}_{h^{\e}}(s))\d\W(s),\w^{\e}_{h^{\e}}(s)).
	\end{align}
	Let us define a sequence of stopping times to be
	\begin{align*}
	\tau_{N}^{\e}:=\inf_{t\geq 0}\left\{t:\|\u^{\e}_{h^{\e}}(t)\|_{\H}>N \ \text{ or } \ \|\u_h(t)\|_{\H}> N\right\}.
	\end{align*}
	Let us fix any $\e_0$ as in Theorem \ref{thm5.10}. Then, we show that  $\sup\limits_{0<\e<\e_0,\ h,h_{\e}\in\mathcal{A}_M}\mathbb{P}\{\omega\in\Omega:\tau_{N,\e}=T\}=1$ as $N\to\infty$. Using Markov's inequality and energy estimates, we have 
	\begin{align}\label{447}
	&\sup\limits_{0<\e<\e_0,\ h,h_{\e}\in\mathcal{A}_M}\mathbb{P}\{\omega\in\Omega:\tau_{N,\e}=T\}\nonumber\\&=\sup\limits_{0<\e<\e_0,\ h,h_{\e}\in\mathcal{A}_M}\mathbb{P}\bigg\{\omega\in\Omega:\sup_{0\leq t\leq T}\|\u_h(t)\|_{\H}^2+\sup_{0\leq t\leq T}\|\u^{\e}_{h_{\e}}(t)\|_{\H}^2\leq 2N^2  \bigg\}\nonumber\\&\geq 1-\frac{1}{2N^2}\sup\limits_{0<\e<\e_0,\ h,h_{\e}\in\mathcal{A}_M}\mathbb{E}\bigg[\sup_{0\leq t\leq T}\|\u_h(t)\|_{\H}^2+\sup_{0\leq t\leq T}\|\u^{\e}_{h_{\e}}(t)\|_{\H}^2\bigg] \nonumber\\&\geq 1-\frac{C}{N^2}(1+\|\u_0\|_{\H}^2),
	\end{align}
	where $C$ is constant depending on $M, \mu, K,T$, etc (see \eqref{5.5y} and \eqref{5.5z}). For $r>3$, we can use \eqref{622} in \eqref{446} and then take supremum in $0\leq t \leq \T^{\e}$ to find 
		\begin{align}\label{455}
	&\sup_{t\in[0,\T^{\e}]}\|\w^{\e}_{h^{\e}}(t)\|_{\H}^2+\mu\int_0^{{\T^{\e}}}\|\w^{\e}_{h^{\e}}(t)\|_{\V}^2\d t+\frac{\beta}{2^{r-1}}\int_0^{\T}\|\w^{\e}_{h^{\e}}(t)\|_{\wi\L^{r+1}}^{r+1}\d t\nonumber\\&\leq \wi\eta\int_0^{{\T^{\e}}}\|\w^{\e}_{h^{\e}}(t)\|_{\H}^2\d t+\e\int_0^{{\T^{\e}}}\|\Phi(t,\u^{\e}_{h_{\e}}(t))\|_{\mathcal{L}_{\Q}}^2\d t\nonumber\\&\quad +2 \int_0^{{\T^{\e}}}|(\Phi(t,\u^{\e}_{h^{\e}}(t))h^{\e}(t)-\Phi(t,\u_h(t))h(t),\w^{\e}_{h^{\e}}(t))|\d t \nonumber\\&\quad  +2\sup_{t\in[0,\T^{\e}]}\left|\int_0^t(\sqrt{\e}\Phi(s,\u^{\e}_{h^{\e}}(s))\d\W(s),\w^{\e}_{h^{\e}}(s))\right|,
	\end{align}
where $\wi\eta$ is defined in \eqref{623}. We estimate the third term from the right hand side of the inequality \eqref{455} using Cauchy-Schwarz and Young's inequalities, and Hypothesis \ref{hyp} (H.3) as 
	\begin{align}\label{456}
	&2\int_0^{{\T^{\e}}}|(\Phi(t,\u^{\e}_{h^{\e}}(t))h^{\e}(t)-\Phi(t,\u_h(t))h(t),\w^{\e}_{h^{\e}}(t))|\d t \nonumber\\& \leq 2\int_0^{{\T^{\e}}}|((\Phi(t,\u^{\e}_{h^{\e}}(t))-\Phi(t,\u_h(t)))h^{\e}(t),\w^{\e}_{h^{\e}}(t))|\d t  \nonumber\\&\quad +2\int_0^{{\T^{\e}}}|(\Phi(t,\u_h(t))(h^{\e}(t)-h(t)),\w^{\e}_{h^{\e}}(t))|\d t \nonumber\\&\leq 2\int_0^{{\T^{\e}}}\|\Phi(t,\u^{\e}_{h^{\e}}(t))-\Phi(t,\u_h(t))\|_{\mathcal{L}_{\Q}}\|h^{\e}(t)\|_{0}\|\w^{\e}_{h^{\e}}(t)\|_{\H}\d t \nonumber\\&\quad+2 \int_0^{{\T^{\e}}}\|\Phi(t,\u_h(t))(h^{\e}(t)-h(t))\|_{\H}\|\w^{\e}_{h^{\e}}(t)\|_{\H}\d t \nonumber\\&\leq L\int_0^{{\T^{\e}}}\left(1+\|h^{\e}(t)\|_{0}^2\right)\|\w^{\e}_{h^{\e}}(t)\|_{\H}^2\d t +\int_0^{{\T^{\e}}}\|\w^{\e}_{h^{\e}}(t)\|_{\H}^2\d t\nonumber\\&\quad +\int_0^{{\T^{\e}}}\|\Phi(t,\u_h(t))(h^{\e}(t)-h(t))\|_{\H}^2\d t.
	\end{align}
	Making use of the Hypothesis \ref{hyp} (H.2) and (H.3),  it can be easily seen that 
	\begin{align}\label{457}
	&\int_0^{{\T^{\e}}}\|\Phi(t,\u^{\e}_{h_{\e}}(t))\|_{\mathcal{L}_{\Q}}^2\d t\nonumber\\&\leq 2\int_0^{{\T^{\e}}} \|\Phi(t,\u^{\e}_{h_{\e}}(t))-\Phi(t,\u_{h}(t))\|_{\mathcal{L}_{\Q}}^2\d t +2 \int_0^{{\T^{\e}}} \|\Phi(t,\u_{h}(t))\|_{\mathcal{L}_{\Q}}^2\d t \nonumber\\&\leq 2L\int_0^{{\T^{\e}}}\|\w^{\e}_{h_{\e}}(t)\|_{\H}^2\d t+2K\int_0^{{\T^{\e}}}\left(1+\|\u_{h}(t)\|_{\H}^2\right)\d t.
	\end{align}
		Using \eqref{456} and \eqref{457} in \eqref{455}, we deduce that 
	\begin{align}\label{458}
	&\sup_{t\in[0,\T^{\e}]}\|\w^{\e}_{h^{\e}}(t)\|_{\H}^2+\mu\int_0^{{\T^{\e}}}\|\w^{\e}_{h^{\e}}(t)\|_{\V}^2\d t+\frac{\beta}{2^{r-1}}\int_0^{\T}\|\w^{\e}_{h^{\e}}(t)\|_{\wi\L^{r+1}}^{r+1}\d t\nonumber\\&\leq \int_0^{{\T^{\e}}}\left(\wi\eta+L\left(1+\|h^{\e}(t)\|_{0}^2\right)+2L\e\right)\|\w^{\e}_{h^{\e}}(t)\|_{\H}^2\d t\nonumber\\&\quad+2\e K\int_0^{{\T^{\e}}}\left(1+\|\u_{h}(t)\|_{\H}^2\right)\d t +\int_0^{{\T^{\e}}}\|\Phi(t,\u_h(t))(h^{\e}(t)-h(t))\|_{\H}^2\d t\nonumber\\&\quad  +2\sup_{t\in[0,\T^{\e}]}\left|\int_0^t(\sqrt{\e}\Phi(s,\u^{\e}_{h^{\e}}(s))\d\W(s),\w^{\e}_{h^{\e}}(s))\right|.
	\end{align}
	An application of Gronwall's inequality in \eqref{458} yields 
	\begin{align}\label{459}
	&\sup_{t\in[0,\T^{\e}]}\|\w^{\e}_{h^{\e}}(t)\|_{\H}^2+\mu\int_0^{{\T^{\e}}}\|\w^{\e}_{h^{\e}}(t)\|_{\V}^2\d t+\frac{\beta}{2^{r-1}}\int_0^{\T}\|\w^{\e}_{h^{\e}}(t)\|_{\wi\L^{r+1}}^{r+1}\d t\nonumber\\&\leq \bigg\{2\e K\int_0^{\T^{\e}}\left(1+\|\u_{h}(t)\|_{\H}^2\right)\d t +\int_0^{\T^{\e}}\|\Phi(t,\u_h(t))(h^{\e}(t)-h(t))\|_{\H}^2\d t\nonumber\\&\quad  +2\sup_{t\in[0,\T^{\e}]}\left|\int_0^t(\sqrt{\e}\Phi(s,\u^{\e}_{h^{\e}}(s))\d\W(s),\w^{\e}_{h^{\e}}(s))\right|\bigg\}\nonumber\\&\qquad\times\exp\left\{\int_0^T\left(\wi\eta+L\left(1+\|h^{\e}(t)\|_{0}^2\right)+2L\e\right)\d t\right\} \nonumber\\&\leq  \bigg\{2\e K\left(T+\sup_{t\in[0,T]}\|\u_{h}(t)\|_{\H}^2\right) +\int_0^{T}\|\Phi(t,\u_h(t))(h^{\e}(t)-h(t))\|_{\H}^2\d t\nonumber\\&\quad  +2\sqrt{\e}\sup_{t\in[0,\T^{\e}]}\left|\int_0^t(\Phi(s,\u^{\e}_{h^{\e}}(s))\d\W(s),\w^{\e}_{h^{\e}}(s))\right|\bigg\}e^{\left(\wi\eta+L+2L\e\right)T+M}, \ \mathbb{P}\text{-a.s.},
	\end{align}
since $h^{\e}\in\mathcal{A}_M$, $\mathbb{P}$-a.s.	Using energy estimates \eqref{5.5y}, we also know that 
	\begin{align}\label{460}
&	\sup_{h\in \mathcal{S}_M}\left\{\sup_{t\in[0,T]}\|\u_h(t)\|_{\H}^2+\mu\int_0^T\|\u_h(t)\|_{\V}^2\d t+\beta\int_0^T\|\u_h(t)\|_{\wi\L^{r+1}}^{r+1}\d t\right\}\nonumber\\&\leq \left(\|\u_0\|_{\H}^2+KM\right)e^{2(T+M)}. 
	\end{align}
	Once again, we use the fact that compact operators maps weakly convergent sequences into strongly convergent sequences. Since $\Phi(\cdot,\cdot)$ is compact and $\big\{h^{\e} : \e > 0\big\}\subset \mathcal{A}_M$ converges in distribution to $h$ with respect to the weak topology on $\mathrm{L}^2(0,T;\H_0)$, we get  
	\begin{align}\label{461}
	\int_0^{T}\|\Phi(t,\u_h(t))(h^{\e}(t)-h(t))\|_{\H}^2\d t\to 0, \ \text{ as }\ \e\to 0,\ \mathbb{P}\text{-a.s.}
	\end{align}
	Using Burkholder-Davis-Gundy, H\"older's and Young's inequalities and Hypothesis \ref{hyp} (H.2), we find 
	\begin{align}\label{462}
	&2\sqrt{\e}\mathbb{E}\left[\sup_{t\in[0,\T^{\e}]}\left|\int_0^t(\Phi(s,\u^{\e}_{h^{\e}}(s))\d\W(s),\w^{\e}_{h^{\e}}(s))\right|\right]\nonumber\\&\leq 2\sqrt{3\e}\mathbb{E}\left[\int_0^{\T^{\e}}\|\Phi(t,\u^{\e}_{h^{\e}}(t))\|_{\mathcal{L}_{\Q}}^2\|\w^{\e}_{h^{\e}}(t)\|_{\H}^2\d t\right]^{1/2} \nonumber\\&\leq 2\sqrt{3\e } \mathbb{E}\left[\sup_{t\in[0,\T^{\e}]}\|\w^{\e}_{h^{\e}}(t)\|_{\H}\left(\int_0^{\T^{\e}}\|\Phi(t,\u^{\e}_{h^{\e}}(t))\|_{\mathcal{L}_{\Q}}^2\d t\right)^{1/2}\right]\nonumber\\&\leq \sqrt{3\e } \mathbb{E}\left[\sup_{t\in[0,\T^{\e}]}\|\w^{\e}_{h^{\e}}(t)\|_{\H}^2+K\int_0^{\T^{\e}}\left(1+\|\u^{\e}_{h^{\e}}(t)\|_{\H}^2\right)\d t\right]\nonumber\\&\leq \sqrt{3\e }\left[(2+K)\mathbb{E}\left(\sup_{t\in[0,T]}\|\u^{\e}_{h^{\e}}(t)\|_{\H}^2\right)+2\sup_{t\in[0,T]}\|\u_h(t)\|_{\H}^2+KT\right],
	\end{align}
	and the right hand side of \eqref{462} is finite using \eqref{460} and \eqref{5.5z}. Thus, using \eqref{462} and Markov's inequality, we have 
	\begin{align}\label{463}
	\lim_{\e\to 0}\sqrt{\e}\sup_{t\in[0,\T^{\e}]}\left|\int_0^t(\Phi(s,\u^{\e}_{h^{\e}}(s))\d\W(s),\w^{\e}_{h^{\e}}(s))\right|=0, \ \mathbb{P}\text{-a.s.}
	\end{align}
	Passing $N\to\infty$ and $\e\to 0$ in \eqref{459} and using \eqref{447}, we finally obtain 
	\begin{align}\label{464}
	\sup_{t\in[0,T]}\|\w^{\e}_{h^{\e}}(t)\|_{\H}^2+\mu\int_0^{T}\|\w^{\e}_{h^{\e}}(t)\|_{\V}^2\d t+\frac{\beta}{2^{r-1}}\int_0^{T}\|\w^{\e}_{h^{\e}}(t)\|_{\wi\L^{r+1}}^{r+1}\d t\to 0, \ \mathbb{P}\text{-a.s.},
	\end{align}
	as $\e\to 0$. 
	
	For $r=3$ and $\beta\mu>1$, one can use the estimate \eqref{632} to get the required result. 
	\end{proof}

\section{Large Deviations for Short Time}\label{sec5}\setcounter{equation}{0}  In this section, we study the  LDP for the solutions to the system \eqref{32} for short time, which in the finite dimensional case is the celebrated Varadhan’s large deviation estimate.  Short time LDP for solution of the stochastic quasigeostrophic equation with multiplicative noise is obtained in \cite{DYJD}, stochastic generalized porous media equations is established in \cite{RWW} and stochastic 2D Oldroyd models is derived in \cite{MTM6}. We discuss the large deviations for the family $\{\u(\e^2t):\e\in(0,1]\}$ of solutions to \eqref{32} in $\C([0,T];\H)$ instead of $\mathscr{E}=\mathrm{C}([0,T];\H)\cap\mathrm{L}^2(0,T;\V)\cap\mathrm{L}^{r+1}(0,T;\wi\L^{r+1})$. Let us define $\widetilde{\u}^{\e}(t)=\u(\e^2t)$. Then, $\widetilde{\u}^{\e}(\cdot)$ satisfies: 
\begin{equation}\label{51}
\left\{
\begin{aligned}
\d\widetilde{\u}^{\e}(t)&=-\e^2\left[\mu\A\widetilde{\u}^{\e}(t)+\B(\widetilde{\u}^{\e}(t))+\beta\mathcal{C}(\widetilde{\u}^{\e}(t))\right]\d t+\e\Phi(t,\widetilde{\u}^{\e}(t))\d\widetilde{\W}(t),\\
\widetilde{\u}^{\e}(0)&=\u_0\in\H,
\end{aligned}
\right.
\end{equation}
where $\widetilde{\W}(t)=(1/\e){\W}(\e^2t)$ is a $\Q$-Wiener process. Note that the laws of $\W(t)$ and $\widetilde{\W}(t)$ are same, using the self scaling property of Wiener process. Then, there exists a measurable map $\widetilde{\mathcal{G}}_{\e}:\C([0,T];\H)\to\C([0,T];\H)$ such that $\widetilde{\u}^{\e}(\cdot)=\widetilde{\mathcal{G}}_{\e}(\widetilde{\W}(\cdot))$. 
Hence, we have the following result on large deviations for short time under the following assumption on the noise coefficient: 
\begin{hypothesis}\label{hyp3}
	\begin{itemize}
		\item [(H.1)] The functions $\A^{1/2}\Phi\in\C([0,T]\times\V;\mathcal{L}_{\Q}(\H))$ and $\Phi\in\C([0,T]\times\wi\L^{r+1};\gamma(\H_0,\wi\L^{r+1}))$.
	\item[(H.2)] 
	There exists a positive	constant $\wi K$  such that for all $t\in[0,T]$ and $\u\in\V$, 
	\begin{align*}
	\|\A^{1/2}\Phi(t, \u)\|^{2}_{\mathcal{L}_{\Q}} &	\leq \wi K\left(1 +\|\u\|_{\V}^{2}\right). 
	\end{align*}
		\item[(H.3)] 
	There exists a positive	constant $\widehat{K}$ such that for all $t\in[0,T]$ and $\u\in\wi\L^{r+1}$, 
	\begin{align*}
	\|\Phi(t,\u)\|^{r+1}_{\gamma(\H_0,\wi\L^{r+1})}&\leq \widehat{K}\left(1 +\|\u\|_{\wi\L^{r+1}}^{r+1}\right). 
	\end{align*}
	\end{itemize}
\end{hypothesis}
In the Hypothesis  \ref{hyp3}, $\gamma(\H_0,\wi\L^{r+1})$ denotes the space of all $\gamma$-radonifying operators from $\H_0=\Q^{\frac{1}{2}}\H$ to $\wi\L^{r+1}$ (see \cite{N10}). From Proposition 3.14, \cite{N10}, it is well-known that every operator $\Phi\in\gamma(\H_0,\wi\L^{r+1})$ is compact. 
\begin{theorem}\label{thm5.2}
	Let $\u(t)$ be the unique strong solution to the system \eqref{32}. Then the family $\{\u(\e^2t),\e\in(0,1]\}$ satisfies LDP in $\C([0,T];\H)$ with a rate function 
	\begin{align}\label{rate1}
	\I(g):=\frac{1}{2}\inf\left\{\int_0^T\|h(s)\|_0^2\d s:g(t)=\u_0+\int_0^t\Phi(s,g(s))h(s)\d s\right\},
	\end{align}
	with the convention that $\inf\emptyset=\infty$. 
\end{theorem}
\begin{proof}
For fixed $M$ and for $h\in \mathcal{S}_M$, let $\widetilde{\u}_h$ satisfies the system:
	\begin{equation}\label{52}
	\left\{
	\begin{aligned}
	\d\widetilde{\u}_h(t)&=\Phi(t,\widetilde{\u}_h(t))h(t)\d t,\\
	\u_h(0)&=\u_0\in\H.
	\end{aligned}
	\right.
	\end{equation}
	Taking the inner product with $\widetilde{\u}_h(\cdot)$ to the first equation in \eqref{52}, we find 
	\begin{align*}
	\frac{1}{2}\frac{\d}{\d t}\|\widetilde{\u}_h(t)\|_{\H}^2&=(\Phi(t,\widetilde{\u}_h(t))h(t),\widetilde{\u}_h(t))\leq\|\Phi(t,\widetilde{\u}_h(t))h(t)\|_{\H}\|\widetilde{\u}_h(t)\|_{\H}\nonumber\\&\leq\frac{1}{2}\|\widetilde{\u}_h(t)\|_{\H}^2+\frac{1}{2}\|\Phi(t,\widetilde{\u}_h(t))\|_{\mathcal{L}_{\Q}}^2\|h(t)\|_{0}^2,
	\end{align*}
	where we used Cauchy-Schwarz and Young's inequalities. Integrating the inequality from $0$ to $t$ and then using the Hypothesis \ref{hyp} (H.2), we obtain 
	\begin{align*}
	\|\widetilde{\u}_h(t)\|_{\H}^2\leq\|\u_0\|_{\H}^2+\int_0^t\|\wi\u_h(s)\|_{\H}^2\d s+K\int_0^t(1+\|\wi\u_h(s)\|_{\H}^2)\|h(s)\|_0^2\d s,
	\end{align*}
	for all $t\in[0,T]$. Applying Gronwall's inequality, we further have 
	\begin{align}\label{76}
	\sup_{t\in[0,T]}\|\widetilde{\u}_h(t)\|_{\H}^2&\leq\left(\|\u_0\|_{\H}^2+K\int_0^T\|h(t)\|_0^2\d t\right)\exp\left(T+K\int_0^T\|h(t)\|_0^2\d t\right)\nonumber\\&\leq\left(\|\u_0\|_{\H}^2+KM\right)e^{T+KM},
	\end{align}
	since $h\in\mathcal{S}_M$. Moreover, for $h\in\mathcal{S}_M$, we get 
	\begin{align*}
	\int_0^T\left\|\frac{\d\wi\u_h(t)}{\d t}\right\|_{\H}^2\d t&\leq \int_0^T\|\Phi(t,\wi\u_h(t))\|_{\mathcal{L}_{\Q}}^2\|h(t)\|_{0}^2\d t\leq K\int_0^T(1+\|\wi\u_h(t)\|_{\H}^2)\|h(t)\|_{0}^2\d t\nonumber\\&\leq KM(1+\left(\|\u_0\|_{\H}^2+KM\right)e^{T+KM}),
	\end{align*}
using \eqref{76}.	Clearly, the system \eqref{52} has a unique weak solution in $\C([0,T];\H)$ under the assumption on $\Phi$ given in the Hypothesis \ref{hyp}. We define $\widetilde{\mathcal{G}}^0:\C([0,T];\H)\to \C([0,T];\H)$ by $\widetilde{\u}_h(\cdot)=\widetilde{\mathcal{G}}^0\left(\int_0^{\cdot}h(s)\d s\right), \text{ for some }h\in\mathrm{L}^2(0,T;\H_0).$ 	In order to prove the Theorem, we need to verify Hypothesis \ref{hyp1}. 
	\vskip 0.2 cm
\noindent\textbf{Part I: Compactness:} 	
One can easily show as in Theorem \ref{compact} that the set $$\widetilde{\mathcal{K}}_M:=\left\{\widetilde{\u}_{h}\in\C([0,T];\H):h\in \mathcal{S}_M\right\}$$ is compact, where $\widetilde{\u}_h(\cdot)$ is the unique solution in $\C([0,T];\H)$ of the deterministic control system \eqref{52}. 
	\vskip 0.2 cm	
\noindent\textbf{Part  II: Weak convergence:}	Our next aim is to establish the weak convergence result. That is, we show that  $ \widetilde{\mathcal{G}}^{\e}(\widetilde{\W}(\cdot) +\int_0^{\cdot}h^{\e}(s)\d s)$ converges in distribution to $\widetilde{\mathcal{G}}^0(\int_0^{\cdot}h(s)\d s)$ in $\C([0,T];\H)$, as $\e\to0$, whenever $\big\{h^{\e} : \e > 0\big\}\subset \mathcal{A}_M$ converges in distribution to $h$ with respect to the weak topology on $\mathrm{L}^2(0,T;\H_0)$. Here $\widetilde{\u}_{h^{\e}}^{\e}(\cdot)= \widetilde{\mathcal{G}}^{\e}(\widetilde{\W}(\cdot) +\int_0^{\cdot}h^{\e}(s)\d s)$ is the unique strong solution to the system:
	\begin{equation}\label{53}
	\left\{
	\begin{aligned}
	\d\widetilde{\u}^{\e}_{h^{\e}}(t)&=-\e^2\left[\mu\A\widetilde{\u}^{\e}_{h^{\e}}(t)+\B(\widetilde{\u}^{\e}_{h^{\e}}(t))+\beta\mathcal{C}(\widetilde{\u}^{\e}_{h^{\e}}(t))\right]\d t+\Phi(t,\widetilde{\u}^{\e}_{h^{\e}}(t))h^{\e}(t)\d t\\&\quad+\e\Phi(t,\widetilde{\u}^{\e}_{h^{\e}}(t))\d\widetilde{\W}(t),\\
	\widetilde{\u}^{\e}_{h^{\e}}(0)&=\u_0\in\H.
	\end{aligned}
	\right.
	\end{equation}
	with paths in $\C([0,T];\H)\cap\mathrm{L}^2(0,T;\V)\cap\mathrm{L}^{r+1}(0,T;\wi\L^{r+1})$, $\mathbb{P}$-a.s.
	
\noindent\emph{Step 1. Approximation of the system \eqref{52}.}	Since $\mathcal{V}\subset\V\cap\wi\L^{r+1}\subset\H$ and $\mathcal{V}$ is dense in $\H$ implies $\V\cap\wi\L^{r+1}$ is also dense in $\H$.  Let $\{\u^n(0)\}$ be a sequence in $\V\cap\wi\L^{r+1}$ such that $\|\u^n(0)-\u_0\|_{\H}\to 0$ as $n\to\infty$. Next, we consider  the following system: 
		\begin{equation}\label{75}
	\left\{
	\begin{aligned}
	\d\widetilde{\u}_h^n(t)&=\Phi(t,\widetilde{\u}_h^n(t))h(t)\d t,\\
\wi	\u^n_h(0)&=\u^n(0)\in\V\cap\wi\L^{r+1}.
	\end{aligned}
	\right.
	\end{equation}
A calculation similar to \eqref{76} yields 
\begin{align}\label{7.8}
\sup_{t\in[0,T]}\|\widetilde{\u}_h^n(t)\|_{\H}^2&\leq\left(\|\widetilde{\u}_h^n(0)\|_{\H}^2+K\int_0^T\|h(t)\|_0^2\d t\right)\exp\left(T+K\int_0^T\|h(t)\|_0^2\d t\right)\nonumber\\&\leq\left(C+\|\u_0\|_{\H}^2+KM\right)e^{T+KM},
\end{align}
since $h\in\mathcal{S}_M$ and every convergent sequence is bounded. 
Taking the inner product with $\A\widetilde{\u}_h^n(\cdot)$ to the first equation in \eqref{75}, we obtain 
	\begin{align*}
	\frac{1}{2}\frac{\d}{\d t}\|\widetilde{\u}_h^n(t)\|_{\V}^2&=(\A^{1/2}\Phi(t,\widetilde{\u}_h^n(t))h(t),\A^{1/2}\widetilde{\u}_h^n(t))\nonumber\\&\leq\|\A^{1/2}\Phi(t,\widetilde{\u}_h^n(t))\|_{\mathcal{L}_{\Q}}\|h(t)\|_{0}\|\A^{1/2}\widetilde{\u}_h^n(t)\|_{\H}\nonumber\\&\leq\frac{1}{2}\|\widetilde{\u}_h^n(t)\|_{\V}^2+\frac{1}{2}\|\A^{1/2}\Phi(t,\widetilde{\u}_h^n(t))\|_{\mathcal{L}_{\Q}}^2\|h(t)\|_0^2. 
	\end{align*}
	Integrating the above inequality from $0$ to $t$ and then using Hypothesis \ref{hyp3} (H.2), we find 
	\begin{align*}
	\|\widetilde{\u}_h^n(t)\|_{\V}^2\leq\|\u^n(0)\|_{\V}^2+\int_0^t\|\widetilde{\u}_h^n(s)\|_{\V}^2\d s+\wi K\int_0^t(1+\|\widetilde{\u}_h^n(s)\|_{\V}^2)\|h(s)\|_0^2\d s,
	\end{align*}
	for all $t\in[0,T]$. Applying Gronwall's inequality, we get 
	\begin{align*}
	\sup_{t\in[0,T]}\|\widetilde{\u}_h^n(t)\|_{\V}^2&\leq\left(\|\u^n(0)\|_{\V}^2+\wi K\int_0^T\|h(t)\|_0^2\d t\right)\exp\left(T+\wi K\int_0^T\|h(t)\|_0^2\d t\right)\nonumber\\&\leq\left(\|\u^n(0)\|_{\V}^2+\wi KM\right)e^{T+\wi KM},
	\end{align*}
	where we used the fact that $h\in\mathcal{S}_M$.
	Taking the inner product with $|\widetilde{\u}_h^n(\cdot)|^{r-1}\widetilde{\u}_h^n(\cdot)$ to the first equation in \eqref{75}, we obtain 
	\begin{align*}
	\frac{1}{r+1}\frac{\d}{\d t}\|\widetilde{\u}_h^n(t)\|_{\wi\L^{r+1}}^{r+1}&=(\Phi(t,\widetilde{\u}_h^n(t))h(t),|\widetilde{\u}_h^n(t)|^{r-1}\widetilde{\u}_h^n(t))\nonumber\\&\leq\|\Phi(t,\widetilde{\u}_h^n(t))h(t)\|_{\wi\L^{r+1}}\||\widetilde{\u}_h^n(t)|^{r-1}\widetilde{\u}_h^n(t)\|_{\wi\L^{\frac{r+1}{r}}}\nonumber\\&\leq\|\Phi(t,\widetilde{\u}_h^n(t))\|_{\gamma(\H_0,\wi\L^{r+1})}\|h(t)\|_0\|\widetilde{\u}_h^n(t)\|_{\wi\L^{r+1}}^r\nonumber\\&\leq\frac{1}{r+1}\|h(t)\|_0^{\frac{r+1}{r}}\|\widetilde{\u}_h^n(t)\|_{\wi\L^{r+1}}^{r+1}+\frac{r^r}{r+1}\|\Phi(t,\widetilde{\u}_h^n(t))\|_{\gamma(\H_0,\wi\L^{r+1})}^{r+1},
	\end{align*}
	where we used H\"older's and Young's inequalities. Integrating the above inequality from $0$ to $t$ and then using Hypothesis \ref{hyp3} (H.3), we find 
	\begin{align*}
	\|\widetilde{\u}_h^n(t)\|_{\wi\L^{r+1}}^{r+1}\leq\|\u^n(0)\|_{\wi\L^{r+1}}^{r+1}+\int_0^t\|h(s)\|_0^{\frac{r+1}{r}}\|\widetilde{\u}_h^n(s)\|_{\wi\L^{r+1}}^{r+1}\d s+r^r\widehat{K}\int_0^t(1+\|\widetilde{\u}_h^n(s)\|_{\wi\L^{r+1}}^{r+1})\d s.
	\end{align*}
	Applying Gronwall's inequality and then using the fact that $h\in\mathcal{S}_M$, we get 
	\begin{align*}
	\sup_{t\in[0,T]}\|\widetilde{\u}_h^n(t)\|_{\wi\L^{r+1}}^{r+1}&\leq\left(\|\u^n(0)\|_{\wi\L^{r+1}}^{r+1}+r^r\widehat{K}T\right)\exp\left(r^r\widehat{K}T+\int_0^T\|h(t)\|_0^{\frac{r+1}{r}}\d t\right)\nonumber\\&\leq\left(\|\u^n(0)\|_{\wi\L^{r+1}}^{r+1}+r^r\widehat{K}T\right)e^{r^r\widehat{K}T+T^{\frac{r-1}{2r}}M^{\frac{r+1}{2r}}},
	\end{align*}
and hence we obtain $\widetilde{\u}_h^n\in\mathrm{L}^{\infty}(0,T;\V\cap\wi\L^{r+1})$. 	Furthermore, we find 
	\begin{align*}
&	\int_0^T\left(\left\|\frac{\d\widetilde{\u}_h^n(t)}{\d t}\right\|^2_{\V}+\left\|\frac{\d\widetilde{\u}_h^n(t)}{\d t}\right\|_{\wi\L^{r+1}}^{2}\right)\d t\nonumber\\&\leq\int_0^T\|\Phi(t,\widetilde{\u}_h^n(t))h(t)\|_{\V}^2\d t+\int_0^T\|\Phi(t,\widetilde{\u}_h^n(t))h(t)\|_{\wi\L^{r+1}}^2\d t\nonumber\\&\leq \widetilde{K}\int_0^T(1+\|\widetilde{\u}_h^n(t)\|_{\V}^2)\|h(t)\|_0^2\d t+\widehat{K}\int_0^T(1+\|\widetilde{\u}_h^n(t)\|_{\wi\L^{r+1}}^2)\|h(t)\|_0^{2}\d t\nonumber\\&\leq M\Bigg\{\widetilde{K}\left[1+\left(\|\u^n(0)\|_{\V}^2+\wi KM\right)e^{T+\wi KM}\right]\nonumber\\&\qquad+\widehat{K}\left[1+\left(\left(\|\u^n(0)\|_{\wi\L^{r+1}}^{r+1}+r^r\widehat{K}T\right)e^{r^r\widehat{K}T+T^{\frac{r-1}{2r}}M^{\frac{r+1}{2r}}}\right)^{\frac{2}{r+1}}\right]\Bigg\},
	\end{align*}
	and thus we get $\partial_t\widetilde{\u}_h^n\in\mathrm{L}^{2}(0,T;\V\cap\wi\L^{r+1})$. The fact that $\widetilde{\u}_h^n\in\mathrm{L}^{\infty}(0,T;\V\cap\wi\L^{r+1})$ and  $\partial_t\widetilde{\u}_h^n\in\mathrm{L}^{2}(0,T;\V\cap\wi\L^{r+1})$ implies that $\widetilde{\u}_h^n\in\mathrm{C}([0,T];\V\cap\wi\L^{r+1})$ (Theorem 2, Chapter 5, page 286, \cite{Evans}).  Note that, $\widetilde{\u}_{h}^n(\cdot)-\widetilde{\u}_{h}(\cdot)$  satisfies: 
	\begin{equation}\label{713}
\left\{
\begin{aligned}
\d(\widetilde{\u}_{h}^n(t)-\widetilde{\u}_h(t))&=[\Phi(t,\widetilde{\u}_h^n(t))-\Phi(t,\widetilde{\u}_h(t))]h(t)\d t,\\
(\widetilde{\u}_{h}^n(0)-\u_h(0))&=\u^n(0)-\u_0\in\H.
\end{aligned}
\right.
\end{equation}
Taking the inner product with $\widetilde{\u}_{h}^n(\cdot)-\widetilde{\u}_{h}(\cdot)$ to the first equation in \eqref{713}, we get 
\begin{align*}
\frac{1}{2}\frac{\d}{\d t}\|\widetilde{\u}_{h}^n(t)-\widetilde{\u}_{h}(t)\|_{\H}^2&=([\Phi(t,\widetilde{\u}_h^n(t))-\Phi(t,\widetilde{\u}_h(t))]h(t),\widetilde{\u}_{h}^n(t)-\widetilde{\u}_{h}(t))\nonumber\\&\leq\|\Phi(t,\widetilde{\u}_h^n(t))-\Phi(t,\widetilde{\u}_h(t))\|_{\mathcal{L}_{\Q}}\|h(t)\|_0\|\widetilde{\u}_{h}^n(t)-\widetilde{\u}_{h}(t)\|_{\H}\nonumber\\&\leq L\|h(t)\|_0\|\widetilde{\u}_{h}^n(t)-\widetilde{\u}_{h}(t)\|_{\H}^2\nonumber\\&\leq\frac{L}{2}\|\widetilde{\u}_{h}^n(t)-\widetilde{\u}_{h}(t)\|_{\H}^2+\frac{L}{2}\|h(t)\|_0^2\|\widetilde{\u}_{h}^n(t)-\widetilde{\u}_{h}(t)\|_{\H}^2,
\end{align*}
where we used Hypothesis \ref{hyp} (H.3). Integrating the above inequality from $0$ to $t$, we find 
\begin{align*}
\|\widetilde{\u}_{h}^n(t)-\widetilde{\u}_{h}(t)\|_{\H}^2&\leq\|\u^n(0)-\u_0\|_{\H}^2+L\int_0^t\|\widetilde{\u}_{h}^n(s)-\widetilde{\u}_{h}(s)\|_{\H}^2\d s\nonumber\\&\quad+L\int_0^t\|h(s)\|_0^2\|\widetilde{\u}_{h}^n(s)-\widetilde{\u}_{h}(s)\|_{\H}^2\d s,
\end{align*}
for all $t\in[0,T]$. Applying Gronwall's inequality, we have 
\begin{align}\label{716}
\|\widetilde{\u}_{h}^n(t)-\widetilde{\u}_{h}(t)\|_{\H}^2\leq\|\u^n(0)-\u_0\|_{\H}^2e^{LT}\exp\left(L\int_0^T\|h(t)\|_0^2\d t\right)\leq\|\u^n(0)-\u_0\|_{\H}^2e^{L(T+M)}, 
\end{align}
for all $t\in[0,T]$, since $h\in\mathcal{S}_M$. Thus, taking $n\to\infty$ in \eqref{716}, we obtain 
\begin{align}\label{717}
\sup_{t\in[0,T]}\|\widetilde{\u}_{h}^n(t)-\widetilde{\u}_{h}(t)\|_{\H}\to 0\ \text{ as } \ n\to\infty.
\end{align}

\noindent\emph{Step 2. Weak convergence arguments.}	Let us now define \begin{align*}\widehat{\w}^{\e}_{h^{\e}}(\cdot):=\widetilde{\u}^{\e}_{h^{\e}}(\cdot)-\widetilde{\u}_{h}(\cdot)=(\widetilde{\u}^{\e}_{h^{\e}}(\cdot)-\widetilde{\u}_{h}^n(\cdot))+(\widetilde{\u}_{h}^n(\cdot)-\widetilde{\u}_{h}(\cdot))=:\widetilde{\w}^{\e}_{h^{\e}}(\cdot)+\widetilde{\w}^{n}_{h}(\cdot).\end{align*} From \eqref{717}, we know that $\sup\limits_{t\in[0,T]}\|\widetilde{\w}^{n}_{h}(t)\|_{\H}^2\to 0$ as $n\to\infty$.  Note that  $\widetilde{\w}^{\e}_{h^{\e}}(\cdot)$ satisfies: 
	\begin{equation}\label{54}
	\left\{
	\begin{aligned}
	\d\widetilde{\w}^{\e}_{h^{\e}}(t)&=-\e^2\left[\mu\A\widetilde{\u}^{\e}_{h^{\e}}(t)+\B(\widetilde{\u}^{\e}_{h^{\e}}(t))+\beta\mathcal{C}(\widetilde{\u}^{\e}_{h^{\e}}(t))\right]\d t\\&\quad +[\Phi(t,\widetilde{\u}^{\e}_{h^{\e}}(t))h^{\e}(t)-\Phi(t,\widetilde{\u}_h^n(t))h(t)]\d t+\e\Phi(t,\widetilde{\u}^{\e}_{h^{\e}}(t))\d\widetilde{\W}(t),\\
	\widetilde{\w}^{\e}_{h^{\e}}(0)&=\mathbf{0}. 
	\end{aligned}
	\right.
	\end{equation}
	Let us define a sequence of stopping times 
	\begin{align*}
	\tau_N^{n,\e}:=\inf_{t\geq 0}\left\{t:	\|\widetilde{\u}^{\e}_{h^{\e}}(t)\|_{\H}>N \ \text{ or } \ \|\widetilde{\u}^{n}_{h}(t)\|_{\V}>N \ \text{ or } \ \|\wi\u_h^n(t)\|_{\wi\L^{r+1}}>N \right\}.
	\end{align*}
	Applying It\^o's formula to the process $\|\widetilde{\w}^{\e}_{h^{\e}}(\cdot)\|_{\H}^2$, we obtain
	\begin{align}\label{55}
	\|\widetilde{\w}^{\e}_{h^{\e}}(\tt)\|_{\H}^2&=-2\e^2\int_0^{\tt}\langle \mu\A\widetilde{\u}^{\e}_{h^{\e}}(s)+\B(\widetilde{\u}^{\e}_{h^{\e}}(s))+\beta\mathcal{C}(\widetilde{\u}^{\e}_{h^{\e}}(s)),\widetilde{\w}^{\e}_{h^{\e}}(s)\rangle\d s\nonumber\\&\quad +2\int_0^{\tt}( [\Phi(s,\widetilde{\u}^{\e}_{h^{\e}}(s))h^{\e}(s)-\Phi(s,\widetilde{\u}_h^n(s))h(s)],\widetilde{\w}^{\e}_{h^{\e}}(s))\d s\nonumber\\&\quad +\e^2\int_0^{\tt}\|\Phi(s,\widetilde{\u}^{\e}_{h^{\e}}(s))\|_{\mathcal{L}_{\Q}}^2\d s +2\e\int_0^t(\Phi(t,\widetilde{\u}^{\e}_{h^{\e}}(t))\d\widetilde{\W}(t), \widetilde{\w}^{\e}_{h^{\e}}(s)).
	\end{align}
	We estimate $\mu\langle\A\widetilde{\u}^{\e}_{h^{\e}},\widetilde{\w}^{\e}_{h^{\e}}\rangle$ as 
	\begin{align}\label{77}
	\mu\langle\A\widetilde{\u}^{\e}_{h^{\e}},\widetilde{\w}^{\e}_{h^{\e}}\rangle=\mu\|\widetilde{\w}^{\e}_{h^{\e}}\|_{\V}^2+\mu(\nabla\wi\u_h^n,\nabla\widetilde{\w}^{\e}_{h^{\e}})\geq\frac{\mu}{2}\|\widetilde{\w}^{\e}_{h^{\e}}\|_{\V}^2-\frac{\mu}{2}\|\wi\u_h^n\|_{\V}^2. 
	\end{align}
	Furthermore, we estimate $\langle\B(\widetilde{\u}^{\e}_{h^{\e}}),\widetilde{\w}^{\e}_{h^{\e}}\rangle$ as 
	\begin{align}\label{78}
	|\langle\B(\widetilde{\u}^{\e}_{h^{\e}}),\widetilde{\w}^{\e}_{h^{\e}}\rangle|&\leq\|\widetilde{\u}^{\e}_{h^{\e}}\|_{\V}\|\widetilde{\u}^{\e}_{h^{\e}}\widetilde{\w}^{\e}_{h^{\e}}\|_{\H}\leq\frac{1}{2}\|\widetilde{\u}^{\e}_{h^{\e}}\|_{\V}^2+\frac{1}{2}\|\widetilde{\u}^{\e}_{h^{\e}}\widetilde{\w}^{\e}_{h^{\e}}\|_{\H}^2\nonumber\\&\leq\frac{1}{2}\|\widetilde{\u}^{\e}_{h^{\e}}\|_{\V}^2+\frac{\beta}{2}\||\widetilde{\u}^{\e}_{h^{\e}}|^{\frac{r-1}{2}}\widetilde{\w}^{\e}_{h^{\e}}\|_{\H}^2+\frac{r-3}{2(r-1)}\left(\frac{2}{\beta (r-1)}\right)^{\frac{2}{r-3}}\|\widetilde{\w}^{\e}_{h^{\e}}\|_{\H}^2,
	\end{align}
for $r>3$,	where we performed a calculation similar to \eqref{6.30}. Let us now estimate $\beta\langle\mathcal{C}(\widetilde{\u}^{\e}_{h^{\e}}),\widetilde{\w}^{\e}_{h^{\e}}\rangle$ as 
\begin{align}\label{79}
\beta\langle\mathcal{C}(\widetilde{\u}^{\e}_{h^{\e}}),\widetilde{\w}^{\e}_{h^{\e}}\rangle&=\beta\int_{\mathcal{O}}|\widetilde{\u}^{\e}_{h^{\e}}(x)|^{r-1}\widetilde{\u}^{\e}_{h^{\e}}(x)\cdot\widetilde{\w}^{\e}_{h^{\e}}(x)\d x\nonumber\\&= \beta\int_{\mathcal{O}}|\widetilde{\u}^{\e}_{h^{\e}}(x)|^{r-1}|\widetilde{\w}^{\e}_{h^{\e}}(x)|^2\d x+\beta\int_{\mathcal{O}}|\widetilde{\u}^{\e}_{h^{\e}}(x)|^{r-1}\widetilde{\u}_{h}^n(x)\cdot\widetilde{\w}^{\e}_{h^{\e}}(x)\d x.
\end{align}
Combining \eqref{78} and \eqref{79}, we obtain 
\begin{align}\label{723}
&\langle\B(\widetilde{\u}^{\e}_{h^{\e}}),\widetilde{\w}^{\e}_{h^{\e}}\rangle+\beta\langle\mathcal{C}(\widetilde{\u}^{\e}_{h^{\e}}),\widetilde{\w}^{\e}_{h^{\e}}\rangle\nonumber\\&\geq -\frac{1}{2}\|\widetilde{\u}^{\e}_{h^{\e}}\|_{\V}^2-\frac{r-3}{2(r-1)}\left(\frac{2}{\beta (r-1)}\right)^{\frac{2}{r-3}}\|\widetilde{\w}^{\e}_{h^{\e}}\|_{\H}^2\nonumber\\&\quad+\frac{\beta}{2}\int_{\mathcal{O}}|\widetilde{\u}^{\e}_{h^{\e}}(x)|^{r-1}|\widetilde{\w}^{\e}_{h^{\e}}(x)|^2\d x+\beta\int_{\mathcal{O}}|\widetilde{\u}^{\e}_{h^{\e}}(x)|^{r-1}\widetilde{\u}_{h}^n(x)\cdot\widetilde{\w}^{\e}_{h^{\e}}(x)\d x\nonumber\\&=-\frac{1}{2}\|\widetilde{\u}^{\e}_{h^{\e}}\|_{\V}^2-\frac{r-3}{2(r-1)}\left(\frac{2}{\beta (r-1)}\right)^{\frac{2}{r-3}}\|\widetilde{\w}^{\e}_{h^{\e}}\|_{\H}^2+\frac{\beta}{2}\|\widetilde{\u}^{\e}_{h^{\e}}\|_{\wi\L^{r+1}}^{r+1}-\frac{\beta}{2}\||\widetilde{\u}^{\e}_{h^{\e}}|^{\frac{r-1}{2}}\wi\u_h^n\|_{\H}^2.
\end{align}
But we know that 
\begin{align*}
\||\widetilde{\u}^{\e}_{h^{\e}}|^{\frac{r-1}{2}}\wi\u_h^n\|_{\H}^2&\leq\||\widetilde{\u}^{\e}_{h^{\e}}|^{\frac{r-1}{2}}\|_{\wi\L^{\frac{2(r+1)}{r-1}}}^2\|\wi\u_h^n\|_{\wi\L^{r+1}}^2=\|\widetilde{\u}^{\e}_{h^{\e}}\|_{\wi\L^{r+1}}^{r-1}\|\wi\u_h^n\|_{\wi\L^{r+1}}^2\nonumber\\&\leq\|\widetilde{\u}^{\e}_{h^{\e}}\|_{\wi\L^{r+1}}^{r+1}+\left(\frac{2}{r+1}\right)\left(\frac{r-1}{r+1}\right)^{\frac{r-1}{2}}\|\wi\u_h^n\|_{\wi\L^{r+1}}^{r+1},
\end{align*}
using H\"older's and Young's inequalities. Thus, from \eqref{723},  we get 
\begin{align}\label{725}
&\langle\B(\widetilde{\u}^{\e}_{h^{\e}}),\widetilde{\w}^{\e}_{h^{\e}}\rangle+\beta\langle\mathcal{C}(\widetilde{\u}^{\e}_{h^{\e}}),\widetilde{\w}^{\e}_{h^{\e}}\rangle\nonumber\\&\geq -\frac{1}{2}\|\widetilde{\u}^{\e}_{h^{\e}}\|_{\V}^2-\frac{r-3}{2(r-1)}\left(\frac{2}{\beta (r-1)}\right)^{\frac{2}{r-3}}\|\widetilde{\w}^{\e}_{h^{\e}}\|_{\H}^2-\left(\frac{\beta}{r+1}\right)\left(\frac{r-1}{r+1}\right)^{\frac{r-1}{2}}\|\wi\u_h^n\|_{\wi\L^{r+1}}^{r+1}.
\end{align}
Substituting \eqref{77} and \eqref{725} in \eqref{55}, we obtain 
\begin{align*}
	\|\widetilde{\w}^{\e}_{h^{\e}}(\tt)\|_{\H}^2&\leq \e^2\int_0^{\tt}\left(\mu\|\wi\u_h^n(s)\|_{\V}^2+\|\widetilde{\u}^{\e}_{h^{\e}}(s)\|_{\V}^2+\zeta\|\widetilde{\w}^{\e}_{h^{\e}}(s)\|_{\H}^2+\wi\zeta\|\wi\u_h^n(s)\|_{\wi\L^{r+1}}^{r+1}\right)\d s\nonumber\\&\quad +2\int_0^{\tt}( [\Phi(s,\widetilde{\u}^{\e}_{h^{\e}}(s))-\Phi(s,\widetilde{\u}_h^n(s))]h^{\e}(s),\widetilde{\w}^{\e}_{h^{\e}}(s))\d s\nonumber\\&\quad +2\int_0^{\tt}( \Phi(s,\widetilde{\u}_h^n(s))(h^{\e}(s)-h(s)),\widetilde{\w}^{\e}_{h^{\e}}(s))\d s\nonumber\\&\quad +\e^2\int_0^{\tt}\|\Phi(s,\widetilde{\u}^{\e}_{h^{\e}}(s))\|_{\mathcal{L}_{\Q}}^2\d s +2\e\int_0^{\tt}(\Phi(t,\widetilde{\u}^{\e}_{h^{\e}}(t))\d\widetilde{\W}(t), \widetilde{\w}^{\e}_{h^{\e}}(s))\nonumber\\&\leq \e^2\int_0^{\tt}\left(\mu\|\wi\u_h^n(s)\|_{\V}^2+\|\widetilde{\u}^{\e}_{h^{\e}}(s)\|_{\V}^2+\zeta\|\widetilde{\w}^{\e}_{h^{\e}}(s)\|_{\H}^2+\wi\zeta\|\wi\u_h^n(s)\|_{\wi\L^{r+1}}^{r+1}\right)\d s\nonumber\\&\quad +\int_0^{\tt}\left[L\left(1+\|h^{\e}(s)\|_{0}^2\right)+1\right]\|\widetilde{\w}^{\e}_{h^{\e}}(s)\|_{\H}^2\d s\nonumber\\&\quad+\int_0^{\tt}\|\Phi(s,\widetilde{\u}_h(s))(h^{\e}(s)-h(s))\|_{\H}^2\d s +K\e^2\int_0^{\tt}\left(1+\|\widetilde{\u}^{\e}_{h^{\e}}(s)\|_{\H}^2\right)\d s \nonumber\\&\quad+2\e\int_0^{\tt}(\Phi(t,\widetilde{\u}^{\e}_{h^{\e}}(t))\d\widetilde{\W}(t), \widetilde{\w}^{\e}_{h^{\e}}(s)),
\end{align*}
where $\zeta=\frac{r-3}{r-1}\left(\frac{2}{\beta (r-1)}\right)^{\frac{2}{r-3}}$ and $\wi\zeta=\left(\frac{2\beta}{r+1}\right)\left(\frac{r-1}{r+1}\right)^{\frac{r-1}{2}}$. 
	An application of Gronwall's inequality in \eqref{55} yields
	\begin{align}\label{56}
&	\sup_{t\in[0,\TT]}\|\widetilde{\w}^{\e}_{h^{\e}}(t)\|_{\H}^2\nonumber\\&\leq\Bigg\{ \e^2\bigg[\int_0^{\TT}\left(\mu\|\wi\u_h^n(t)\|_{\V}^2+\|\widetilde{\u}^{\e}_{h^{\e}}(t)\|_{\V}^2+\wi\zeta\|\wi\u_h^n(t)\|_{\wi\L^{r+1}}^{r+1}\right)\d  t\nonumber\\&\qquad+T\bigg(K+\sup_{t\in[0,\TT]}\|\widetilde{\u}^{\e}_{h^{\e}}(t)\|_{\H}^2\bigg)\bigg]+\int_0^{\TT}\|\Phi(t,\widetilde{\u}_h(t))(h^{\e}(t)-h(t))\|_{\H}^2\d t\nonumber\\&\qquad+2\e\sup_{t\in[0,\TT]}\left|\int_0^t(\Phi(t,\widetilde{\u}^{\e}_{h^{\e}}(t))\d\widetilde{\W}(t), \widetilde{\w}^{\e}_{h^{\e}}(s))\right|\Bigg\}e^{(M+L+1+\e^2\zeta)T}\nonumber\\&\leq\Bigg\{ \e^2\bigg[\mu\int_0^{T}\|\widetilde{\u}^{\e}_{h^{\e}}(t)\|_{\V}^2\d  t+T\bigg(N^2+\wi\zeta N^{r+1}+K+\sup_{t\in[0,T]}\|\widetilde{\u}^{\e}_{h^{\e}}(t)\|_{\H}^2\bigg)\bigg]\nonumber\\&\qquad+\int_0^T\|\Phi(t,\widetilde{\u}_h(t))(h^{\e}(t)-h(t))\|_{\H}^2\d t\nonumber\\&\qquad+2\e\sup_{t\in[0,\TT]}\left|\int_0^t(\Phi(t,\widetilde{\u}^{\e}_{h^{\e}}(t))\d\widetilde{\W}(t), \widetilde{\w}^{\e}_{h^{\e}}(s))\right|\Bigg\}e^{(M+L+1+\e^2\zeta)T}.
	\end{align}
	Note that $\widetilde{\u}^{\e}_{h^{\e}}(\cdot)$ has paths in $\C([0,T];\H)\cap\mathrm{L}^2(0,T;\V)\cap\mathrm{L}^{r+1}(0,T;\wi\L^{r+1})$, $\mathbb{P}$-a.s. Let us now show that  $\sup\limits_{0<\e<\e_0,\ h,h_{\e}\in\mathcal{A}_M}\mathbb{P}\{\omega\in\Omega:\tau_{N}^{n,\e}=T\}=1$ as $N\to\infty$, for each $n\in\mathbb{N}$. Using Markov's inequality and energy estimates, we have 
	\begin{align}\label{732}
	&\sup\limits_{0<\e<\e_0,\ h,h_{\e}\in\mathcal{A}_M}\mathbb{P}\{\omega\in\Omega:\tau_{N}^{n,\e}=T\}\nonumber\\&=\sup\limits_{0<\e<\e_0,\ h,h_{\e}\in\mathcal{A}_M}\mathbb{P}\bigg\{\omega\in\Omega:\sup_{0\leq t\leq T}\left(\|\widetilde{\u}^{\e}_{h^{\e}}(t)\|_{\H}^2+\|\widetilde{\u}^{n}_{h}(t)\|_{\V}+\|\widetilde{\u}^{n}_{h}(t)\|_{\wi\L^{r+1}}\right)\leq N^2+2N  \bigg\}\nonumber\\&\geq 1-\frac{1}{N^2+2N}\sup\limits_{0<\e<\e_0,\ h,h_{\e}\in\mathcal{A}_M}\left\{\mathbb{E}\bigg[\sup_{0\leq t\leq T}\|\widetilde{\u}^{\e}_{h^{\e}}(t)\|_{\H}^2+\sup_{0\leq t\leq T}\|\widetilde{\u}^{n}_{h}(t)\|_{\V}+\sup_{0\leq t\leq T}\|\widetilde{\u}^{n}_{h}(t)\|_{\wi\L^{r+1}}\bigg]\right\} \nonumber\\&\geq 1-\frac{C}{N^2+2N}(1+\|\u_0\|_{\H}^2+\|\u^n(0)\|_{\V}+\|\u^n(0)\|_{\wi\L^{r+1}})\nonumber\\&\to 1\ \text{ as } \ N\to\infty,
	\end{align} for each $n$. Using Burkholder-Davis-Gundy, H\"older's and Young's inequalities and Hypothesis \ref{hyp} (H.2) as in \eqref{462}, we obtain 
	\begin{align}\label{57}
	&2\e\mathbb{E}\left[\sup_{t\in[0,\TT]}\left|\int_0^t(\Phi(t,\widetilde{\u}^{\e}_{h^{\e}}(t))\d\widetilde{\W}(t), \widetilde{\w}^{\e}_{h^{\e}}(s))\right|\right]\nonumber\\&\leq \sqrt{3 }\e\left[(2+K)\mathbb{E}\left(\sup_{t\in[0,T]}\|\widetilde{\u}^{\e}_{h^{\e}}(t)\|_{\H}^2\right)+2\sup_{t\in[0,\TT]}\|\widetilde{\u}_h^n(t)\|_{\H}^2+KT\right]\nonumber\\&\leq \sqrt{3 }\e\left[C(2+K)(1+\|\u_0\|_{\H}^2)+2N^2+KT\right],
	\end{align}
	and the right hand side of \eqref{57} is finite. 
	Using \eqref{732}, \eqref{57} and \eqref{461} in \eqref{56} and then using energy estimates satisfied by $\widetilde{\u}^{\e}_{h^{\e}}(\cdot)$ and $\widetilde{\u}_h^n(\cdot)$  (see \eqref{5.5z} and \eqref{7.8}), we  obtain $$\sup\limits_{t\in[0,T]}\|\widetilde{\w}^{\e}_{h^{\e}}(t)\|_{\H}^2 \ \to 0, \ \mathbb{P}\text{-a.s.},$$ as $\e\to 0$ and $N\to \infty$, and hence $\widetilde{\w}^{\e}_{h^{\e}}\to 0$ in $\C([0,T];\H)$, $\mathbb{P}\text{-a.s.},$ as $\e\to 0$ (if needed one can choose $N=\left(\frac{1}{\e}\right)^{\frac{1}{r+1}}$). Finally, we note that 
	\begin{align*}
\sup_{t\in[0,T]}	\|\widehat{\w}^{\e}_{h^{\e}}(t)\|_{\H}\leq\sup_{t\in[0,T]}\|\widetilde{\w}^{\e}_{h^{\e}}(t)\|_{\H}+\sup_{t\in[0,T]}\|\widetilde{\w}^{n}_{h}(t)\|_{\H}.\end{align*} Taking $\e\to 0$ and then $n\to\infty$, we obtain $\sup\limits_{t\in[0,T]}	\|\widehat{\w}^{\e}_{h^{\e}}(t)\|_{\H}\to 0$ as $\e\to 0$,
which  completes the weak convergence proof and the Theorem follows for $r>3$. 

The case of $r=3$ and $\beta\mu>1$  can be shown in a similar way using the calculation given  in \eqref{632}. 
\end{proof}
\begin{remark}
	If $\Phi(\cdot,\cdot)$ is an invertible operator, then the rate function has the form $$\I(g)=\frac{1}{2}\int_0^T\|\Phi(t,g(t))^{-1}\dot{g}(t)\|_{0}^2\d t.$$ 
\end{remark}

\section{Exponential Inequalities}\label{sec6}\setcounter{equation}{0}  In this section, we consider the SCBF equations perturbed by additive Gaussian noise.  We first derive an exponential inequality for the energy of the solution trajectory $\u(\cdot)$ of the system:
\begin{equation}\label{4}
\left\{
\begin{aligned}
\d\u(t)+\mu \A\u(t)+\B(\u(t))+\beta\mathcal{C}(\u(t))&=\Phi(t)\d\W(t), \ t\in(0,T),\\
\u(0)&=\u_0,
\end{aligned}
\right.
\end{equation}
when it exceeds a threshold $R>0$ by time $T$.  Specifically, for any fixed $R>0$, we define 
\begin{align*}
\tau_R:=\inf_{t\in[0,T]}\Big\{t:\|\u(t)\|_{\H}>R\Big\},
\end{align*}
the first time of exit for the solution from the $R$-ball in the space $\H$. We follow the works 
\cite{Chow,Chow1,Hsu}, etc to get the required results. In order to obtain the global solvability (existence and uniqueness of pathwise strong  solution) of the system \eqref{4}, the noise coefficient $\Phi:\Omega\times[0,T]\to\mathcal{L}_{\Q}(\H)$ satisfies the following assumption:
\begin{hypothesis}\label{assm2.4}
	The noise coefficient $\Phi:\Omega\times[0,T]\to\mathcal{L}_{\Q}(\H)$ is assumed to be predictable with \begin{align}\label{72}\mathbb{E}\left[\int_0^T\mathrm{Tr}(\Phi(s)\Q\Phi^*(s))\d s\right]<+\infty.\end{align} 
\end{hypothesis}

Let us now give  Doob's martingale inequality, which is used  in this section. 
\begin{theorem}[Theorem II.1.7, \cite{DRM}]
	Let $(\mathscr{F}_t)_{t \ge 0}$ be a filtration on the probability space $(\Omega, \mathscr{F},\mathbb{P})$ and let $(\M(t))_{t \ge 0}$ be a real valued right continuous martingale or positive submartingale with respect to the filtration $(\mathscr{F}_t)_{t \ge 0}$. 
	Then, for arbitrary $\uplambda> 0$, 
	$$ \mathbb {P} \left[\sup _{0\leq t\leq T}|\M(t)|\geq \uplambda\right]\leq \uplambda^{-p}\sup _{0\leq t\leq T} \mathbb{E}\left[|\M(t)|^p\right].$$
\end{theorem}
For the one-dimensional Brownian motion $(\mathrm{B}(t))_{t\geq 0}$, we have  (Proposition II.1.8, \cite{DRM})
\begin{align*}
\mathbb {P} \left[\sup _{0\leq t\leq T}|\mathrm{B}(t)|\geq \uplambda\right]\leq \exp\left(-\frac{\uplambda^2}{2T}\right). 
\end{align*}
For any $\theta\in\R$, note that $e^{\theta\mathrm{B}(t)-\frac{\theta^2 t}{2}}$ is a positive martingale with $\mathbb{E}\left[e^{\theta  \mathrm{B}(t)-\frac{\theta^2  t}{2}}\right]= \mathbb{E}\left[e^{\theta \mathrm{B}(0)-\frac{\theta^2 0}{2}}\right]=1$, and using Doob's martingale inequality, we find 
\begin{align}\label{27a}
\mathbb {P} \left[\sup _{0\leq t\leq T}\left(\theta \mathrm{B}(t)-\frac{\theta^2  t}{2}\right)\geq \uplambda\right]&=\mathbb {P} \left[\sup _{0\leq t\leq T}\exp\left(\theta \mathrm{B}(t)-\frac{\theta^2 t}{2}\right)\geq e^{\uplambda}\right]\nonumber\\&\leq e^{-\uplambda}\E\left[\exp\left(\theta \mathrm{B}(t)-\frac{\theta^2 t}{2}\right)\right]=e^{-\uplambda}. 
\end{align}

With the above motivation, if we  define 
\begin{align}\label{27}
\vartheta_{\theta}(t)= \theta \int_0^t(\u(s),\Phi(s)\d\W(s))-\frac{\theta ^2}{2}\int_0^t\|\Q^{\frac{1}{2}}\Phi^*(s)\u(s)\|_{\H}^2\d s,
\end{align}
then we have 
\begin{lemma}[Lemma 4.2, \cite{Chow}]
	For every real $\theta$, the exponential process $\Z(t)=\exp(\vartheta_{\theta}(t))$ is a martingale and $\E[\Z(t)]=1$. 
\end{lemma}
\begin{proposition}\label{prop3.1}
	Assume that there exists a bounded, strictly positive, increasing function $\rho_t$ on $(0,T]$ such that 
	\begin{align}\label{a3.2}
	\int_0^t\|\Phi^*(s)\Phi(s)\|_{\mathcal{L}(\H,\H)}\d s\leq \rho_t, \ \mathbb{P}\text{-a.s. for all } t\in(0,T].
	\end{align}
	Then, for any $R>0$, the pathwise strong solution $\u(\cdot)$ to the system \eqref{4} satisfies: 
	\begin{align}\label{a3.4}
	&\mathbb{P}\left\{\sup_{0\leq t\leq T}\|\u(t)\|_{\H}>R\right\}\leq \exp\left\{\|\u_0\|_{\H}^2+\mathrm{Tr}(\Q)\rho_T\right\}\exp\left\{-\frac{R^2}{e^{4\mathrm{Tr}(\Q)\rho_T}}\right\}.
	\end{align}
\end{proposition}
\begin{proof}
	Using  It\^o's formula to the process $\|\u(\cdot)\|_{\H}^2$ (see \eqref{a34}), we find 
	\begin{align*}
	&	\|\u(t)\|_{\H}^2+2\mu\int_0^t\|\u(s)\|_{\V}^2\d s+2\beta\int_0^t\|\u(s)\|_{\wi\L^{r+1}}^{r+1}\d s\nonumber\\&=\|\u_0\|_{\H}^2+2\int_0^t(\Phi(s)\d\W(s),\u(s))+\int_0^t\mathrm{Tr}(\Phi(s)\Q\Phi^*(s))\d s.
	\end{align*}
	The above equality implies 
	\begin{align}\label{a3.13}
	\|\u(t)\|_{\H}^2&\leq \|\u_0\|_{\H}^2+\vartheta_2(t)+2\int_0^t\|\Q^{\frac{1}{2}}\Phi^*(s)\u(s)\|_{\H}^2\d s+\int_0^t\mathrm{Tr}(\Phi(s)\Q\Phi^*(s))\d s,
	\end{align}
	where 
	\begin{align}\label{a314}
	\vartheta_2(t)=2\int_0^t(\Phi(s)\d\W(s),\u(s))-2\int_0^t\|\Q^{\frac{1}{2}}\Phi^*(s)\u(s)\|_{\H}^2\d s.
	\end{align}
	But we know that
	\begin{align}\label{a3.14a}
	\|\Q^{\frac{1}{2}}\Phi^*\u\|_{\H}&\leq \|\Q^{\frac{1}{2}}\Phi^*\|_{\mathcal{L}(\H,\H)}\|\u\|_{\H}\leq 
	\mathrm{Tr}(\Phi\Q\Phi^*)^{1/2}\|\u\|_{\H}\leq \mathrm{Tr}(\Q)^{1/2}\|\Phi^*\Phi\|_{\mathcal{L}(\H,\H)}^{1/2}\|\u\|_{\H}. 
	\end{align}
	Thus, from \eqref{a3.13}, it is immediate that 
	\begin{align}\label{a3.15}
	\|\u(t)\|_{\H}^2&\leq \|\u_0\|_{\H}^2+\vartheta_2(t)+2\mathrm{Tr}(\Q)\int_0^t\|\Phi^*(s)\Phi(s)\|_{\mathcal{L}(\H,\H)}\|\u(s)\|_{\H}^2\d s\nonumber\\&\quad+\mathrm{Tr}(\Q)\int_0^t\|\Phi^*(s)\Phi(s)\|_{\mathcal{L}(\H,\H)}\d s.
	\end{align}
	An application of Gronwall's inequality in \eqref{a3.15} yields 
	\begin{align}\label{a3.16}
	\|\u(t)\|_{\H}^2& \leq \left\{ \|\u_0\|_{\H}^2+\vartheta_2(t)+\mathrm{Tr}(\Q)\int_0^t\|\Phi^*(s)\Phi(s)\|_{\mathcal{L}(\H,\H)}\d s\right\}\nonumber\\&\quad +\bigg\{2\mathrm{Tr}(\Q) \int_0^t\|\Phi^*(s)\Phi(s)\|_{\mathcal{L}(\H,\H)}\left[ \|\u_0\|_{\H}^2+\vartheta_2(s) +\mathrm{Tr}(\Q)\int_0^s\|\Phi^*(r)\Phi(r)\|_{\mathcal{L}(\H,\H)}\d r\right]\nonumber\\&\qquad\times\exp\left[2\mathrm{Tr}(\Q)\int_s^t\|\Phi^*(r)\Phi(r)\|_{\mathcal{L}(\H,\H)}\d r\right]\bigg\}\d s. 
	\end{align}
	Taking supremum over time on both sides of \eqref{a3.16}, we arrive at 
	\begin{align}\label{a3.17}
	\sup_{t\in[0,T]}\|\u(t)\|_{\H}^2& \leq \left(1+2\mathrm{Tr}(\Q)\rho_Te^{2\mathrm{Tr}(\Q)\rho_T}\right) \left( \|\u_0\|_{\H}^2+\sup_{0\leq t\leq T}\vartheta_2(t)+\mathrm{Tr}(\Q)\rho_T\right)\nonumber\\&\leq e^{4\mathrm{Tr}(\Q)\rho_T} \left( \|\u_0\|_{\H}^2+\sup_{0\leq t\leq T}\vartheta_2(t)+\mathrm{Tr}(\Q)\rho_T\right),
	\end{align}
	where $\rho_T$ is defined in \eqref{a3.2}.  Remember that $\rho_t$ is bounded for all $t\in(0,T]$. Thus, for any fixed $R>0$, using \eqref{27a}, we infer that 
	\begin{align*}
	\mathbb{P}\left\{\sup_{0\leq t\leq T}\|\u(t)\|_{\H}>R\right\}&= \mathbb{P}\left\{\sup_{0\leq t\leq T}\|\u(t)\|_{\H}^2>R^2\right\}\nonumber\\&\leq \mathbb{P}\left\{ e^{4\mathrm{Tr}(\Q)\rho_T}\left( \|\u_0\|_{\H}^2+\sup_{0\leq t\leq T}\vartheta_2(t)+\mathrm{Tr}(\Q)\rho_T\right)>R^2\right\}\nonumber\\&= \mathbb{P}\left\{\sup_{0\leq t\leq T}\vartheta_2(t)>R^2e^{-4\mathrm{Tr}(\Q)\rho_T}-\left(\|\u_0\|_{\H}^2+\mathrm{Tr}(\Q)\rho_T\right) \right\}\nonumber\\&=\mathbb{P}\left\{\sup_{0\leq t\leq T}\exp[\vartheta_2(t)]>\exp\left[R^2e^{-4\mathrm{Tr}(\Q)\rho_T}-\left(\|\u_0\|_{\H}^2+\mathrm{Tr}(\Q)\rho_T\right)\right]\right\}\nonumber\\&\leq \exp\left[-\frac{R^2}{e^{4\mathrm{Tr}(\Q)\rho_T}}+\left(\|\u_0\|_{\H}^2+\mathrm{Tr}(\Q)\rho_T\right)\right],
	\end{align*}
	which completes the proof of \eqref{a3.4}. 
\end{proof}

Let us now find the probability of the deviations of the trajectory $\u(\cdot)$ from its unperturbed state $\u^0(\cdot)$, that satisfies:
\begin{equation}\label{a3.29}
\left\{
\begin{aligned}
\d\u^0(t)+[\mu \A\u^0(t)+\B(\u^0(t))+\beta\mathcal{C}(\u^0(t))]\d t&=\mathbf{0}, \ t\in(0,T),\\
\u^0(0)&=\u_0,
\end{aligned}
\right.
\end{equation}
From \cite{CLF,KT2,MTM7}, we know that the system \eqref{a3.29} has a unique weak solution $\u\in\mathrm{C}([0,T];\H)\cap\mathrm{L}^2(0,T;\V)\cap\mathrm{L}^{r+1}(0,T;\wi\L^{r+1})$ (see Theorem \ref{exis2} also).

\begin{proposition}\label{prop3.3}
	Let the assumptions of Proposition \ref{prop3.1} be satisfied. Then for any $R>0$, we have 
	\begin{align*}
	\mathbb{P}\left\{\sup_{0\leq t\leq T}\|\u(t)-\u^0(t)\|_{\H}>R\right\}\leq e^{\mathrm{Tr}(\Q)\rho_T} \exp\left[-\frac{R^2}{e^{4\mathrm{Tr}(\Q)\rho_T}}\right].
	\end{align*}
\end{proposition}
\begin{proof}
	Let us define $\widetilde{\u}(\cdot):=\u(\cdot)-\u^0(\cdot),$ so that $\widetilde{\u}(\cdot)$ satisfies:
	\begin{equation}\label{a3.30}
	\left\{
	\begin{aligned}
	\d\wi\u(t)+[\mu \A\wi\u(t)+\B(\u(t))-\B(\u^0(t))+\beta(\mathcal{C}(\u(t))-\mathcal{C}(\u^0(t)))]\d t&=\Phi(t)\W(t), \\
	\u^0(0)&=\u_0,
	\end{aligned}
	\right.
	\end{equation}
	for $ t\in(0,T)$. Applying It\^o's formula to the process $\|\widetilde{\u}(\cdot)\|_{\H}^2$ (see \eqref{a34}), we  find 
	\begin{align}\label{a3.31}
	&	\|\widetilde{\u}(t)\|_{\H}^2+2\mu\int_0^t\|\widetilde{\u}(s)\|_{\V}^2\d s\nonumber\\&=-2\int_0^t[(\B(\mathbf{u}(s))-\B({\u^0}(s)),\widetilde{\u}(s))+\beta\langle\mathcal{C}(\u(s))-\mathcal{C}(\u^0(s)),\wi\u(s)\rangle]\d s\nonumber\\&\quad+2\int_0^t(\Phi(s)\d\W(s),\widetilde{\u}(s))+\int_0^t\mathrm{Tr}(\Phi(s)\Q\Phi^*(s))\d s.
	\end{align}
	Using \eqref{2.23}, we get 
	\begin{align}\label{a620}
	\beta\langle\mathcal{C}(\u)-\mathcal{C}(\u^0),\u-\u^0\rangle\geq\frac{\beta}{2}\||\u^0|^{\frac{r-1}{2}}(\u-\u_0)\|_{\H}^2.
	\end{align}
	We estimate $\langle\B(\u)-\B(\u^0),\u-\u^0\rangle=\langle\B(\u-\u^0,\u^0),\u-\u^0\rangle$ as 
	\begin{align}\label{a3.32}
	| \langle\B(\u)-\B(\u^0),\u-\u^0\rangle|&=|\langle\B(\u-\u^0,\u-\u^0),\u^0\rangle|\leq\|\u-\u^0\|_{\V}\|\u^0(\u-\u^0)\|_{\H}\nonumber\\&\leq\frac{\mu}{2}\|\u-\u^0\|_{\V}^2+\frac{1}{2\mu}\|\u^0(\u-\u^0)\|_{\H}^2.\end{align} 
	Now, we estimate $\|\u^0(\u-\u^0)\|_{\H}^2$ as 
	\begin{align}\label{a622}
	\|\u^0(\u-\u^0)\|_{\H}^2&=\int_{\mathcal{O}}|\u^0(x)|^2|\u(x)-\u^0(x)|^2\d x\nonumber\\&=\int_{\mathcal{O}}|\u^0(x)|^2|\u(x)-\u^0(x)|^{\frac{4}{r-1}}|\u(x)-\u^0(x)|^{\frac{2(r-3)}{r-1}}\d x\nonumber\\&\leq\left(\int_{\mathcal{O}}|\u^0(x)|^{r-1}|\u(x)-\u^0(x)|^2\d x\right)^{\frac{2}{r-1}}\left(\int_{\mathcal{O}}|\u(x)-\u^0(x)|^2\d x\right)^{\frac{r-3}{r-1}}\nonumber\\&\leq{\beta\mu }\left(\int_{\mathcal{O}}|\u^0(x)|^{r-1}|\u(x)-\u^0(x)|^2\d x\right)\nonumber\\&\quad+\frac{r-3}{r-1}\left(\frac{2}{\beta\mu (r-1)}\right)^{\frac{2}{r-3}}\left(\int_{\mathcal{O}}|\u(x)-\u^0(x)|^2\d x\right).
	\end{align}
	Combining \eqref{a620}-\eqref{a622}, we obtain 
	\begin{align}\label{a623}
	&\beta\langle\mathcal{C}(\u)-\mathcal{C}(\u^0),\u-\u^0\rangle+\langle\B(\u)-\B(\u^0),\u-\u^0\rangle\nonumber\\&\geq-\frac{\mu}{2}\|\u-\u^0\|_{\V}^2-\frac{\eta}{2}\|\u-\u^0\|_{\H}^2,
	\end{align}
	where $\eta=\frac{r-3}{\mu(r-1)}\left(\frac{2}{\beta\mu (r-1)}\right)^{\frac{2}{r-3}}$. 	Thus, from \eqref{a3.31}, we obtain 
	\begin{align}\label{a3.38}
	\|\widetilde{\u}(t)\|_{\H}^2+\mu\int_0^t\|\wi\u(s)\|_{\V}^2\d s&\leq\eta\int_0^t\|\wi\u(s)\|_{\H}^2\d s+\vartheta_2(t)+2\int_0^t\|\Q^{\frac{1}{2}}\Phi^*(s)\u(s)\|_{\H}^2\d s\nonumber\\&\quad+\int_0^t\mathrm{Tr}(\Phi(s)\Q\Phi^*(s))\d s,
	\end{align}
	where $\vartheta_2(t)$ is defined in \eqref{a314}. Using \eqref{a3.14a} in \eqref{a3.38}, we further get 
	\begin{align}\label{a3.39}
	\|\widetilde{\u}(t)\|_{\H}^2&\leq \eta\int_0^t\|\wi\u(s)\|_{\H}^2\d s+\vartheta_2(t)+2\mathrm{Tr}(\Q)\int_0^t\|\Phi^*(s)\Phi(s)\|_{\mathcal{L}(\H,\H)}\|\u(s)\|_{\H}^2\d s\nonumber\\&\quad+\mathrm{Tr}(\Q)\int_0^t\|\Phi^*(s)\Phi(s)\|_{\mathcal{L}(\H,\H)}\d s.
	\end{align}
	An application of Gronwall's inequality in \eqref{a3.39} yields 
	\begin{align}\label{a3.40}
	\|\widetilde{\u}(t)\|_{\H}^2&\leq \left\{\vartheta_2(t)+\mathrm{Tr}(\Q)\int_0^t\|\Phi^*(s)\Phi(s)\|_{\mathcal{L}(\H,\H)}\d s\right\}\nonumber\\&\quad+2\mathrm{Tr}(\Q)\int_0^t\bigg\{\|\Phi^*(s)\Phi(s)\|_{\mathcal{L}(\H,\H)} \left[\vartheta_2(s)+\mathrm{Tr}(\Q)\int_0^s\|\Phi^*(r)\Phi(r)\|_{\mathcal{L}(\H,\H)}\d r\right] \nonumber\\&\qquad\times \exp\left[2\mathrm{Tr}(\Q)\int_s^t\|\Phi^*(s)\Phi(s)\|_{\mathcal{L}(\H,\H)}\right]\bigg\}\d s.
	\end{align}
	Taking supremum over time from $0$ to $T$ in \eqref{a3.40}, we find 
	\begin{align*}
	\sup_{0\leq t\leq T}\|\widetilde{\u}(t)\|_{\H}^2&\leq\left(1+2\mathrm{Tr}(\Q)\rho_Te^{2\mathrm{Tr}(\Q)\rho_T}\right)\left(\sup_{0\leq t\leq T}\vartheta_2(t)+\mathrm{Tr}(\Q)\rho_T\right) \nonumber\\&\leq e^{4\mathrm{Tr}(\Q)\rho_T}\left(\sup_{0\leq t\leq T}\vartheta_2(t)+\mathrm{Tr}(\Q)\rho_T\right).
	\end{align*}
	Hence, for any fixed $R>0$, using \eqref{27a}, we infer that 
	\begin{align*}
	\mathbb{P}\left\{\sup_{0\leq t\leq T}\|\widetilde{\u}(t)\|_{\H}>R\right\}&= \mathbb{P}\left\{\sup_{0\leq t\leq T}\|\widetilde{\u}(t)\|_{\H}^2>R^2\right\}\nonumber\\&\leq \mathbb{P}\left\{e^{4\mathrm{Tr}(\Q)\rho_T}\left(\sup_{0\leq t\leq T}\vartheta_2(t)+\mathrm{Tr}(\Q)\rho_T\right)>R^2\right\}\nonumber\\&=\mathbb{P}\left\{\sup_{0\leq t\leq T}\exp[\vartheta_2(t)]>\exp\left[R^2e^{-4\mathrm{Tr}(\Q)\rho_T}-\mathrm{Tr}(\Q)\rho_T\right]\right\}\nonumber\\&\leq \exp\left[-\frac{R^2}{e^{4\mathrm{Tr}(\Q)\rho_T}}+\mathrm{Tr}(\Q)\rho_T\right],
	\end{align*}
	which completes the proof. 
\end{proof}

\section{Relation with Large Deviations}\label{sec7}\setcounter{equation}{0} 
In this section, we examine exit times of solutions of the SCBF equations \eqref{4} from the $R$-ball by using small noise asymptotic granted  by large deviations theory.  The following Theorem shows that the LDP is preserved under continuous mappings, and is known as the \emph{contraction principle}. 

\begin{theorem}[Contraction principle, Theorem 4.2.1, \cite{DZ}]\label{thm4.6}
	Let $\mathscr{E}$ and $\mathscr{G}$ be Hausdorff
	topological spaces and $f : \mathscr{E}\to \mathscr{G}$ a continuous function. Let us consider a good rate function $I : \mathscr{E}\to  [0,\infty]$. 
	\begin{enumerate}
		\item [(a)] For each $y \in\mathscr{G}$, define
		\begin{align*}
		\mathrm{J}(y):=\inf\left\{\mathrm{I}(x):x\in\mathscr{E},\ y=f(x)\right\}.
		\end{align*}
		\item [(b)] If $\mathrm{I}$ controls the LDP associated with a family of probability measures
		$\{\mu_{\e}\}$ on $\mathscr{E}$, then $\mathrm{I}$ controls the LDP associated with the family of probability measures $\{\mu_{\e}\circ f^{-1}\}$ on $\mathscr{G}$. 
	\end{enumerate}
\end{theorem}

\subsection{Exponential estimates and LDP}
Let us consider the following stochastic Stokes' problem:
\begin{equation}\label{a4.1}
\left\{
\begin{aligned}
\d\z(t)+\mu \A\z(t)&=\Phi(t)\d\W(t), \ t\in(0,T),\\
\z(0)&=\mathbf{0}.
\end{aligned}
\right.
\end{equation}
Under the Hypothesis \ref{72}, we can show that there exists a unique pathwise strong solution to the system \eqref{a4.1} with trajectories in $\mathrm{C}([0,T];\H)\cap\mathrm{L}^2(0,T;\V)$, $\mathbb{P}$-a.s., and satisfies the following energy estimate: 
\begin{align*}
\mathbb{E}\left[\sup_{0\leq t\leq T}\|\z(t)\|_{\H}^2+4\mu\int_0^T\|\z(t)\|_{\V}^2\d t\right]\leq 14\mathbb{E}\left[\int_0^T\mathrm{Tr}(\Phi(t)\Q\Phi^*(t))\d t\right].
\end{align*}
 Using Sobolev embedding theorem, we know that 
\begin{align*}
\H^{2\alpha}(\mathcal{O})\subset\L^{r+1}(\mathcal{O}), \ \text{ for } \ 0\leq \alpha\leq\frac{n(r-1)}{4(r+1)}. 
\end{align*}
Using the interpolation and Sobolev inequalities, we obtain 
\begin{align}\label{7.5}
\|\u\|_{\wi\L^{r+1}}&\leq\|\u\|^{\frac{r-1}{r+1}}_{\wi\L^{\frac{n(r+1)}{n+2}}}\|\u\|^{\frac{2}{r+1}}_{\wi\L^{\frac{n(r+1)}{n-r+1}}}\leq C\|\A^{\alpha}\u\|_{\H}^{\frac{r-1}{r+1}}\|\A^{\frac{1}{2}+\alpha}\u\|_{\H}^{\frac{2}{r+1}}, 
\end{align}
for $0\leq\alpha\leq\frac{n}{4}\left(\frac{n(r-1)-4}{n(r+1)}\right)$ and $3\leq r\leq n+1$. For $n=2$, we can take $\alpha =0,$ for $r\in[1,3]$, and hence we need only Hypothesis \ref{assm2.4} to obtain the results established in this section. Once again, using interpolation and  inequalities, we find  
\begin{align}\label{7.6}
\|\u\|_{\wi\L^{r+1}}&\leq\|\u\|^{\frac{r-1}{r+1}}_{\wi\L^{r-1}}\|\u\|^{\frac{2}{r+1}}_{\wi\L^{\infty}}\leq C\|\A^{\alpha}\u\|_{\H}^{\frac{r-1}{r+1}}\|\A^{\frac{1}{2}+\alpha}\u\|_{\H}^{\frac{2}{r+1}}, 
\end{align}
for $\frac{n}{4}-\frac{1}{2}<\alpha\leq\frac{n}{4}\left(\frac{r-3}{r-1}\right)$ and $r> n+1$. 

In this section, we need the following additional assumption on the noise coefficient.
\begin{hypothesis}\label{assm7.7}
	The noise coefficient $\A^{\alpha}\Phi:\Omega\times[0,T]\to\mathcal{L}_{\Q}(\H)$ is assumed to be predictable with \begin{align}\label{a7.5}\mathbb{E}\left[\int_0^T\mathrm{Tr}(\A^{\alpha}\Phi(s)\Q(\A^{\alpha}\Phi)^*(s))\d s\right]<+\infty,\end{align} for the range of $\alpha$ given in \eqref{7.5} and \eqref{7.6}. 
\end{hypothesis}
Applying Ito's formula to the process $\|\A^{\alpha}\z(\cdot)\|_{\H}^p$, for $p\geq 2$, we find 
\begin{align*}
&\|\A^{\alpha}\z(t)\|_{\H}^p+p\mu\int_0^t\|\A^{\alpha}\z(s)\|_{\H}^{p-2}\|\A^{\frac{1}{2}+\alpha}\z(s)\|_{\H}^2\d s\nonumber\\&=p\int_0^t\|\A^{\alpha}\z(s)\|_{\H}^{p-2}(\A^{\alpha}\Phi(s)\d\W(s),\A^{\alpha}\z(s))\nonumber\\&\quad+\frac{p(p-1)}{2}\int_0^t\|\A^{\alpha}\z(s)\|_{\H}^{p-2}\mathrm{Tr}(\A^{\alpha}\Phi(s)\Q(\A^{\alpha}\Phi)^*(s))\d s.
\end{align*}
Taking supremum over time $0\leq t\leq T$ and then taking expectation, we obtain 
\begin{align}\label{a7.7}
&\E\left[\sup_{0\leq t\leq T}\|\A^{\alpha}\z(t)\|_{\H}^p+2\mu\int_0^T\|\A^{\alpha}\z(t)\|_{\H}^{p-2}\|\A^{\frac{1}{2}+\alpha}\z(t)\|_{\H}^2\d t\right]\nonumber\\&\leq\frac{p(p-1)}{2}\E\left[\int_0^T\|\A^{\alpha}\z(t)\|_{\H}^{p-2}\mathrm{Tr}(\A^{\alpha}\Phi(t)\Q(\A^{\alpha}\Phi)^*(t))\d t\right]\nonumber\\&\quad+p\E\left[\sup_{t\in[0,T]}\left|\int_0^t\|\A^{\alpha}\z(s)\|_{\H}^{p-2}(\A^{\alpha}\Phi(s)\d\W(s),\A^{\alpha}\z(s))\right|\right].
\end{align}
Applying H\"older's and Young's inequalities, we estimate the first term from the right hand side of the inequality \eqref{a7.7} as 
\begin{align}\label{a7.8}
&\frac{p(p-1)}{2}\E\left[\int_0^T\|\A^{\alpha}\z(t)\|_{\H}^{p-2}\mathrm{Tr}(\A^{\alpha}\Phi(t)\Q(\A^{\alpha}\Phi)^*(t))\d t\right]\nonumber\\&\leq \frac{p(p-1)}{2}\E\left[\sup_{0\leq t\leq T}\|\A^{\alpha}\z(t)\|_{\H}^{p-2}\int_0^T\mathrm{Tr}(\A^{\alpha}\Phi(t)\Q(\A^{\alpha}\Phi)^*(t))\d t\right]\nonumber\\&\leq\frac{1}{4}\E\left[\sup_{0\leq t\leq T}\|\A^{\alpha}\z(t)\|_{\H}^{p}\right]\nonumber\\&\quad+(p-1)\left(2(p-1)(p-2)\right)^{\frac{p-2}{2}}\E\left[\left(\int_0^T\mathrm{Tr}(\A^{\alpha}\Phi(t)\Q(\A^{\alpha}\Phi)^*(t))\d t\right)^{p/2}\right]\nonumber\\&\leq \frac{1}{4}\E\left[\sup_{0\leq t\leq T}\|\A^{\alpha}\z(t)\|_{\H}^{p}\right]+C_p\left[\E\left(\int_0^T\mathrm{Tr}(\A^{\alpha}\Phi(t)\Q(\A^{\alpha}\Phi)^*(t))\d t\right)\right]^{p/2},
\end{align}
where $C_p=(p-1)\left(2(p-1)(p-2)\right)^{\frac{p-2}{2}}$. Using Burkholder-Davis-Gundy, H\"older's and Young's inequalities, we estimate the final term from the right hand side of the inequality \eqref{a7.7} as 
\begin{align}\label{a7.9}
&p\E\left[\sup_{t\in[0,T]}\left|\int_0^t\|\A^{\alpha}\z(s)\|_{\H}^{p-2}(\A^{\alpha}\Phi(s)\d\W(s),\A^{\alpha}\z(s))\right|\right]\nonumber\\&\leq\sqrt{3}p\E\left[\int_0^T\|\A^{\alpha}\z(s)\|_{\H}^{2(p-1)}\|\A^{\alpha}\Phi(s)\|_{\mathcal{L}_{\Q}}^2\d s\right]^{1/2}\nonumber\\&\leq\sqrt{3}p\E\left[\sup_{t\in[0,T]}\|\A^{\alpha}\z(s)\|_{\H}^{(p-1)}\left(\int_0^T\mathrm{Tr}(\A^{\alpha}\Phi(t)\Q(\A^{\alpha}\Phi)^*(t))\d t\right)^{1/2}\right]\nonumber\\&\leq\frac{1}{4}\E\left[\sup_{0\leq t\leq T}\|\A^{\alpha}\z(t)\|_{\H}^p\right]+\widetilde{C}_p\left[\E\left(\int_0^T\mathrm{Tr}(\A^{\alpha}\Phi(t)\Q(\A^{\alpha}\Phi)^*(t))\d t\right)\right]^{p/2},
\end{align}
where $\widetilde{C}_p=\sqrt{3}(4\sqrt{3}(p-1))^{p-1}$. Making use of \eqref{a7.8} and \eqref{a7.9} in \eqref{a7.7}, we deduce that 
\begin{align}\label{a710}
&\E\left[\sup_{0\leq t\leq T}\|\A^{\alpha}\z(t)\|_{\H}^p+2p\mu\int_0^T\|\A^{\alpha}\z(t)\|_{\H}^{p-2}\|\A^{\frac{1}{2}+\alpha}\z(t)\|_{\H}^2\d t\right]\nonumber\\&\leq 2(C_p+\widetilde{C}_p)\left[\E\left(\int_0^T\mathrm{Tr}(\A^{\alpha}\Phi(t)\Q(\A^{\alpha}\Phi)^*(t))\d t\right)\right]^{p/2}.
\end{align}
Thus, we obtain $\z\in\mathrm{L}^{2}(\Omega;\mathrm{L}^{\infty}(0,T;\D(\A^{\alpha})))\cap\mathrm{L}^2(0,T;\D(\A^{\frac{1}{2}+\alpha}))$ and $\z\in\C([0,T];\D(\A^{\alpha}))\cap\mathrm{L}^2(0,T;\D(\A^{\frac{1}{2}+\alpha}))$, $\mathbb{P}$-a.s. Thus, using \eqref{7.5} and \eqref{7.6}, it is immediate that 
\begin{align*}
\E\left[\int_0^T\|\z(t)\|_{\wi\L^{r+1}}^{r+1}\d t\right]&\leq C\E\left[\int_0^T\|\A^{\alpha}\z(t)\|_{\H}^{r-1}\|\A^{\frac{1}{2}+\alpha}\z(t)\|_{\H}^{2}\d t\right]\nonumber\\&\leq 2C(C_{r+1}+\widetilde{C}_{r+1})\left[\E\left(\int_0^T\mathrm{Tr}(\A^{\alpha}\Phi(t)\Q(\A^{\alpha}\Phi)^*(t))\d t\right)\right]^{\frac{r+1}{2}},
\end{align*}
where $C_{r+1}+\widetilde{C}_{r+1}=r\left(2r(r-1)\right)^{\frac{r-1}{2}}+\sqrt{3}(4\sqrt{3}r)^{r}$ and hence $\z\in\mathrm{L}^{r+1}(\Omega;\mathrm{L}^{r+1}(0,T;\wi\L^{r+1}))$.
\begin{example}
	For example, one can take $\Phi=\mathrm{I}$ and $\Q=\A^{-\e}$, for $\e>\frac{n}{2}+2\alpha$. The behavior of these eigenvalues of the Stokes operator is well known in the literature (for example see Theorem 2.2, Corollary 2.2, \cite{AAI}).  For all $k\geq 1$, we have  
	\begin{align}\label{a643}
	\uplambda_k\geq {C}_nk^{2/n}, \ \text{ where }\ 
	{C}_n=\frac{n}{2+n}\left(\frac{(2\pi)^n}{\omega_n(n-1)|\Omega|}\right)^{2/n}, \ \omega_n= \pi^{n/2}\Gamma(1+n/2), 
	\end{align}  and $|\Omega|$ is the $n$-dimensional Lebesgue measure of $\Omega$. Then, the condition \eqref{a7.5} is satisfied if 
	\begin{align*}
	\mathrm{Tr}(\A^{2\alpha-\e})=\sum_{k=1}^{\infty}(\A^{2\alpha-\e}e_k,e_k)=\sum_{k=1}^{\infty}\uplambda_k^{2\alpha-\e}\geq C_n^{2\alpha-\e}\sum_{k=1}^{\infty}k^{\frac{2(2\alpha-\e)}{n}}=C_n^{2\alpha-\e}\sum_{k=1}^{\infty}\frac{1}{k^{\frac{2(\e-2\alpha)}{n}}}<+\infty,
	\end{align*}
	provided $\e>\frac{n}{2}+2\alpha$. 
\end{example}

Let us define $\v(\cdot)=\u(\cdot)-\z(\cdot)$ and consider the system satisfied by $\v(\cdot)$ as 
\begin{eqnarray}\label{a4.3}
\left\{
\begin{aligned}
\d \v(t)+[\A\mathbf{v}(t)+\B(\mathbf{v}(t)+\z(t))+\beta\mathcal{C}(\v(t)+\z(t))]\d t&=\mathbf{0}, \ t\in(0,T),\\
\mathbf{v}(0)&=\mathbf{u}_0.
\end{aligned}
\right.
\end{eqnarray}
Taking the inner product with $\v(\cdot)$ to the first equation in \eqref{a4.3}, we obtain 
\begin{align}\label{a4.4}
&\frac{1}{2}\frac{\d}{\d t}\|\v(t)\|_{\H}^2+\mu\|\v(t)\|_{\V}^2=-\langle\B(\v(t)+\z(t)),\v(t)\rangle-\beta\langle\mathcal{C}(\v(t)+\z(t)),\v(t)\rangle. 
\end{align}
The above equality implies 
\begin{align}\label{a716}
&\frac{1}{2}\frac{\d}{\d t}\|\v(t)\|_{\H}^2+\mu\|\v(t)\|_{\V}^2+\beta\|\v(t)+\z(t)\|_{\wi\L^{r+1}}^{r+1}\nonumber\\&=\langle\B(\v(t)+\z(t)),\z(t)\rangle+\beta\langle\mathcal{C}(\v(t)+\z(t)),\z(t)\rangle.
\end{align}
It can be easily seen that 
\begin{align}\label{a717}
|\langle\mathcal{C}(\v+\z),\z\rangle|&\leq\||\v+\z|^{r-1}(\v+\z)\|_{\wi\L^{\frac{r+1}{r}}}\|\z\|_{\wi\L^{r+1}}=\|\v+\z\|_{\wi\L^{r+1}}^{r}\|\z\|_{\wi\L^{r+1}}\nonumber\\&\leq\frac{1}{4}\|\v+\z\|_{\wi\L^{r+1}}^{r+1}+\frac{1}{r+1}\left(\frac{4r}{r+1}\right)^r\|\z\|_{\wi\L^{r+1}}^{r+1}.
\end{align}
Calculations similar to \eqref{a3.32} and \eqref{a622} yield
\begin{align}\label{a718}
|\langle\B(\v+\z),\z\rangle|&=|\langle\B(\v+\z,\z),\v+\z\rangle|\leq\|\z\|_{\V}\|(\v+\z)(\v+\z)\|_{\H}\nonumber\\&\leq\|\z\|_{\V}^2+\frac{1}{4}\|(\v+\z)(\v+\z)\|_{\H}^2\leq\|\z\|_{\V}^2+\frac{\beta}{4}\|\v+\z\|_{\wi\L^{r+1}}^{r+1}+\frac{\widehat{\eta}}{4}\|\v+\z\|_{\H}^2,
\end{align}
where $\widehat\eta=\frac{r-3}{r-1}\left(\frac{2}{\beta (r-1)}\right)^{\frac{2}{r-3}}$. Combining \eqref{a717} and \eqref{a718}, substituting it in \eqref{a716} and then integrating from $0$ to $t$, we get 
\begin{align*}
&\|\v(t)\|_{\H}^2+2\mu\int_0^t\|\v(s)\|_{\V}^2\d s+\beta\int_0^t\|\v(s)+\z(s)\|_{\wi\L^{r+1}}^{r+1}\d s\nonumber\\&\leq\|\u_0\|_{\H}^2+\kappa\int_0^t\|\z(s)\|_{\wi\L^{r+1}}^{r+1}\d s+2\int_0^t\|\z(s)\|_{\V}^2\d s+\widehat{\eta}\int_0^t\|\z(s)\|_{\H}^2\d s+\widehat{\eta}\int_0^t\|\v(s)\|_{\H}^2\d s, 
\end{align*}
where $\kappa=\frac{2\beta}{r+1}\left(\frac{4r}{r+1}\right)^r$. Applying Gronwall's inequality, we get 
\begin{align}\label{a720}
&\sup_{t\in[0,T]}\|\v(t)\|_{\H}^2+2\mu\int_0^T\|\v(t)\|_{\V}^2\d t+\beta\int_0^T\|\v(t)+\z(t)\|_{\wi\L^{r+1}}^{r+1}\d t\nonumber\\&\leq \left\{\|\u_0\|_{\H}^2+\kappa\int_0^T\|\z(t)\|_{\wi\L^{r+1}}^{r+1}\d t+2\int_0^T\|\z(t)\|_{\V}^2\d t+\widehat{\eta}\int_0^T\|\z(t)\|_{\H}^2\d t\right\} e^{2\widehat{\eta}T},
\end{align}
for  each $\omega\in\Omega$, and hence $\v\in\mathrm{L}^{\infty}(0,T;\H)\cap\mathrm{L}^2(0,T;\V)\cap\mathrm{L}^{r+1}(0,T;\wi\L^{r+1})$, $\mathbb{P}$-a.s. The final regularity result holds true, since 
\begin{align*}
\int_0^T\|\v(t)\|_{\wi\L^{r+1}}^{r+1}\d t\leq 2^r\int_0^T\|\v(t)+\z(t)\|_{\wi\L^{r+1}}^{r+1}\d t+ 2^r\int_0^T\|\z(t)\|_{\wi\L^{r+1}}^{r+1}\d t<+\infty.
\end{align*}
Moreover, $\v$ is having a continuous modification with trajectories in $\C([0,T];\H)$, $\mathbb{P}$-a.s. 

The next Lemma is proved under the Hypothesis \ref{assm7.7}. 
\begin{lemma}\label{lem2.7}
	Let a function $\psi\in\mathrm{C}([0,T];\D(\A^{\alpha}))$ be given.  Let the map \begin{align}\label{a4.130}
	\Psi:\psi\to\v_{\psi}\end{align} be defined by 
	\begin{eqnarray}\label{a4.13}
	\left\{
	\begin{aligned}
	\d \v_{\psi}(t)+[\A\mathbf{v}_{\psi}(t)+\B(\mathbf{v}_{\psi}(t)+\psi(t))+\beta\mathcal{C}(\v_{\psi}(t)+\psi(t))]\d t&=0, \ t\in(0,T),\\
	\mathbf{v}_{\psi}(0)&=\mathbf{u}_0.
	\end{aligned}
	\right.
	\end{eqnarray}
	Then, the map $\Psi$ is a continuous map from $\mathrm{C}([0,T];\D(\A^{\alpha}))$ into the space $\mathrm{C}([0,T];\H)\cap\mathrm{L}^2(0,T;\V)\cap\mathrm{L}^{r+1}(0,T;\wi\L^{r+1})$. 
\end{lemma}
\begin{proof}
	We consider two functions $\psi_1$ and $\psi_2$ in $\mathrm{C}([0,T];\D(\A^{\alpha} ))$. Let us denote the corresponding weak solution of \eqref{a4.13} as $\v_{\psi_i}$, for $i=1,2$. Then the system $\v_{\psi}=\v_{\psi_1}-\v_{\psi_2}$ where $\psi:=\psi_1-\psi_2$ satisfies:
	\begin{eqnarray}\label{a4.14}
	\left\{
	\begin{aligned}
	\d\v_{\psi}(t)+[\A\mathbf{v}_{\psi}(t)+\B(\v_{\psi_1}(t)+\psi_1(t))-\B(\v_{\psi_2}(t)+\psi_2(t))&\\+\beta(\mathcal{C}(\v_{\psi_1}(t)+\psi_1(t))-\mathcal{C}(\v_{\psi_2}(t)+\psi_2(t)))]\d t&=\mathbf{0}, \ t\in\times(0,T),\\
	\v_{\psi}(0)&=\mathbf{0}.
	\end{aligned}
	\right.
	\end{eqnarray}
	Taking the inner product with $\v_{\psi}$ to the first equation in \eqref{a4.14}, we find 
	\begin{align}\label{a4.17}
	&\frac{1}{2}\frac{\d}{\d t}\|\v_{\psi}(t)\|_{\H}^2+\mu\|\v_{\psi}(t)\|_{\V}^2\nonumber\\&=-\langle\B(\v_{\psi_1}(t)+\psi_1(t))-\B(\v_{\psi_2}(t)+\psi_2(t)),\v_{\psi}(t)\rangle\nonumber\\&\quad-\beta\langle\mathcal{C}(\v_{\psi_1}(t)+\psi_1(t))-\mathcal{C}(\v_{\psi_2}(t)+\psi_2(t)),\v_{\psi}(t)\rangle\nonumber\\&=-\langle\B(\v_{\psi_1}(t)+\psi_1(t))-\B(\v_{\psi_2}(t)+\psi_2(t)),\v_{\psi}(t)+\psi(t)\rangle\nonumber\\&\quad-\beta\langle\mathcal{C}(\v_{\psi_1}(t)+\psi_1(t))-\mathcal{C}(\v_{\psi_2}(t)+\psi_2(t)),\v_{\psi}(t)+\psi(t)\rangle\nonumber\\&\quad+\langle\B(\v_{\psi_1}(t)+\psi_1(t))-\B(\v_{\psi_2}(t)+\psi_2(t)),\psi(t)\rangle\nonumber\\&\quad+\beta\langle\mathcal{C}(\v_{\psi_1}(t)+\psi_1(t))-\mathcal{C}(\v_{\psi_2}(t)+\psi_2(t)),\psi(t)\rangle\nonumber\\&=:\sum_{i=1}^4I_i.
	\end{align}
	Using H\"older's and Young's inequalities, we estimate $I_1$ as 
	\begin{align}\label{a729}
	I_1&=|\langle\B(\v_{\psi}+\psi,\v_{\psi}+\psi),\psi\rangle|=|\langle\B(\v_{\psi}+\psi,\psi),\v_{\psi}+\psi\rangle|\nonumber\\&\leq\|\psi\|_{\V}\|(\v_{\psi}+\psi)(\v_{\psi}+\psi)\|_{\H}\leq\frac{1}{2}\|\psi\|_{\V}^2+\frac{1}{2}\|(\v_{\psi}+\psi)(\v_{\psi}+\psi)\|_{\H}^2.
	\end{align}
	We estimate the final term from the right hand side of the inequality \eqref{a729} using H\"older's and Young's inequalities as 
	\begin{align}\label{a7.29}
&	\|(\v_{\psi}+\psi)(\v_{\psi}+\psi)\|_{\H}^2 \nonumber\\&=\int_{\mathcal{O}}|\v_{\psi}(x)+\psi(x)|^2|\v_{\psi}(x)+\psi(x)|^2\d x\nonumber\\&=\int_{\mathcal{O}}|\v_{\psi}(x)+\psi(x)|^2|\v_{\psi}(x)+\psi(x)|^{\frac{4}{r-1}}|\v_{\psi}(x)+\psi(x)|^{\frac{2(r-3)}{r-1}}\d x\nonumber\\&\leq\left(\int_{\mathcal{O}}|\v_{\psi}(x)+\psi(x)|^{r-1}|\v_{\psi}(x)+\psi(x)|^2\d x\right)^{\frac{2}{r-1}}\left(\int_{\mathcal{O}}|\v_{\psi}(x)+\psi(x)|^2\d x\right)^{\frac{r-3}{r-1}}\nonumber\\&\leq\frac{\beta}{2^{r-3}}\|\v_{\psi}+\psi\|_{\wi\L^{r+1}}^{r+1}+\frac{r-3}{r-1}\left(\frac{2^{r-2}}{\beta(r-1)}\right)^{\frac{2}{r-3}}\|\v_{\psi}+\psi\|_{\H}^2,
	\end{align}
	for $r>3$. Thus, $I_1$ can be estimated as 
	\begin{align*}
	I_1&\leq\frac{1}{2}\|\psi\|_{\V}^2+ \frac{\beta}{2^{r-2}}\|\v_{\psi}+\psi\|_{\wi\L^{r+1}}^{r+1}+\varrho\|\v_{\psi}+\psi\|_{\H}^2,
	\end{align*}
where $\varrho=\frac{r-3}{2(r-1)}\left(\frac{2^{r-2}}{\beta(r-1)}\right)^{\frac{2}{r-3}}$.	We estimate $I_2$ using \eqref{2.23} as 
	\begin{align*}
	I_2&\leq -\frac{\beta}{2}\||\v_{\psi_1}+\psi_1|^{\frac{r-1}{2}}(\v_{\psi}+\psi)\|_{\H}^2-\frac{\beta}{2}\||\v_{\psi_2}+\psi_2|^{\frac{r-1}{2}}(\v_{\psi}+\psi)\|_{\H}^2.
	\end{align*}
	Now, we estimate $I_3$ as 
	\begin{align}\label{a731}
	I_3&=\langle\B(\v_{\psi_1}+\psi_1,\v_{\psi}+\psi),\psi\rangle+\langle\B(\v_{\psi}+\psi,\v_{\psi_2}+\psi_2),\psi\rangle\nonumber\\&=-\langle\B(\v_{\psi_1}+\psi_1,\psi),\v_{\psi}+\psi\rangle-\langle\B(\v_{\psi}+\psi,\psi),\v_{\psi_2}+\psi_2\rangle\nonumber\\&\leq\|\psi\|_{\V}\|(\v_{\psi_1}+\psi_1)(\v_{\psi}+\psi)\|_{\H}+\|\psi\|_{\V}\|(\v_{\psi_2}+\psi_2)(\v_{\psi}+\psi)\|_{\H}\nonumber\\&\leq\|\psi\|_{\V}^2+\frac{1}{2}\|(\v_{\psi_1}+\psi_1)(\v_{\psi}+\psi)\|_{\H}^2+\frac{1}{2}\|(\v_{\psi_2}+\psi_2)(\v_{\psi}+\psi)\|_{\H}^2.
	\end{align}
	The estimation of the term $\|(\v_{\psi_1}+\psi_1)(\v_{\psi}+\psi)\|_{\H}^2$ is similar to \eqref{a7.29}, and we estimate it as 
	\begin{align}\label{a732}
	&	\|(\v_{\psi_1}+\psi_1)(\v_{\psi}+\psi)\|_{\H}^2\nonumber\\&\leq\frac{\beta}{4}\||\v_{\psi_1}+\psi_1|^{\frac{r-1}{2}}(\v_{\psi}+\psi)\|_{\H}^2+\frac{r-3}{r-1}\left(\frac{8}{\beta(r-1)}\right)^{\frac{2}{r-3}}\|\v_{\psi}+\psi\|_{\H}^2,
	\end{align}
	for $r>3$. 	Similarly, we estimate $\|(\v_{\psi_2}+\psi_2)(\v_{\psi}+\psi)\|_{\H}^2$ as 
	\begin{align}\label{a733}
	\|(\v_{\psi_1}+\psi_1)(\v_{\psi}+\psi)\|_{\H}^2\leq\frac{\beta}{4}\||\v_{\psi_2}+\psi_2|^{\frac{r-1}{2}}(\v_{\psi}+\psi)\|_{\H}^2+\rho\|\v_{\psi}+\psi\|_{\H}^2,
	\end{align}
	where $\rho=\frac{r-3}{r-1}\left(\frac{8}{\beta(r-1)}\right)^{\frac{2}{r-3}}$. Using \eqref{a732} and \eqref{a733} in \eqref{a731}, we find 
	\begin{align}\label{a734}
	I_3\leq\|\psi\|_{\V}^2+\frac{\beta}{8}\||\v_{\psi_1}+\psi_1|^{\frac{r-1}{2}}(\v_{\psi}+\psi)\|_{\H}^2+\frac{\beta}{8}\||\v_{\psi_2}+\psi_2|^{\frac{r-1}{2}}(\v_{\psi}+\psi)\|_{\H}^2+\rho\|\v_{\psi}+\psi\|_{\H}^2,
	\end{align}
	for $r>3$. Finally, we estimate $I_4$ using Taylor's formula (Theorem 7.9.1, \cite{PGC}) as 
	\begin{align}\label{a735}
	I_4&=\beta\left<\int_0^1\mathcal{C}'(\theta(\v_{\psi_1}+\psi_1)+(1-\theta)(\v_{\psi_2}+\psi_2))(\v_{\psi}+\psi)\d\theta,\psi\right>\nonumber\\&=\beta r\left<\int_0^1|\theta(\v_{\psi_1}+\psi_1)+(1-\theta)(\v_{\psi_2}+\psi_2)|^{r-1}(\v_{\psi}+\psi)\d\theta,\psi\right>\nonumber\\&\leq\beta r\left<\sup_{0<\theta<1}|\theta(\v_{\psi_1}+\psi_1)+(1-\theta)(\v_{\psi_2}+\psi_2)|^{r-1}|\v_{\psi}+\psi|,|\psi|\right>\nonumber\\&\leq\beta r2^{r-2}\langle\left(|\v_{\psi_1}+\psi_1|^{r-1}+|\v_{\psi_2}+\psi_2|^{r-1}\right)|\v_{\psi}+\psi|,|\psi|\rangle\nonumber\\&\leq\frac{\beta}{8}\||\v_{\psi_1}+\psi_1|^{\frac{r-1}{2}}(\v_{\psi}+\psi)\|_{\H}^2+\frac{\beta}{8}\||\v_{\psi_2}+\psi_2|^{\frac{r-1}{2}}(\v_{\psi}+\psi)\|_{\H}^2\nonumber\\&\quad+\beta r^2 2^{2r-3}\||\v_{\psi_1}+\psi_1|^{\frac{r-1}{2}}\psi\|_{\H}^2+\beta r^2 2^{2r-3}\||\v_{\psi_2}+\psi_2|^{\frac{r-1}{2}}\psi\|_{\H}^2\nonumber\\&\leq \frac{\beta}{8}\||\v_{\psi_1}+\psi_1|^{\frac{r-1}{2}}(\v_{\psi}+\psi)\|_{\H}^2+\frac{\beta}{8}\||\v_{\psi_2}+\psi_2|^{\frac{r-1}{2}}(\v_{\psi}+\psi)\|_{\H}^2\nonumber\\&\quad+\beta r^2 2^{2r-3}\|\v_{\psi_1}+\psi_1\|_{\wi\L^{r+1}}^{r-1}\|\psi\|_{\wi\L^{r+1}}^2+\beta r^2 2^{2r-3}\|\v_{\psi_2}+\psi_2\|_{\wi\L^{r+1}}^{r-1}\|\psi\|_{\wi\L^{r+1}}^2.
	\end{align}
	From \eqref{a215}, we infer that 
	\begin{align*}
	\frac{2^{2-r}\beta}{2}\|\u-\v\|_{\wi\L^{r+1}}^{r+1}\leq\frac{\beta}{2}\||\u|^{\frac{r-1}{2}}(\u-\v)\|_{\L^2}^2+\frac{\beta}{2}\||\v|^{\frac{r-1}{2}}(\u-\v)\|_{\L^2}^2, 
	\end{align*}
	and hence we get 
	\begin{align*}
	\frac{\beta}{2^{r-1}}\|\v_{\psi}+\psi\|_{\wi\L^{r+1}}^{r+1}\leq\frac{\beta}{2}\||\v_{\psi_1}+\psi_1|^{\frac{r-1}{2}}(\v_{\psi}+\psi)\|_{\H}^2+\frac{\beta}{2}\||\v_{\psi_2}+\psi_2|^{\frac{r-1}{2}}(\v_{\psi}+\psi)\|_{\H}^2.
	\end{align*}
	Combining \eqref{a729}-\eqref{a735} and substituting it in \eqref{a4.17}, we get 
	\begin{align}\label{a736}
	&\|\v_{\psi}(t)\|_{\H}^2+2\mu\int_0^t\|\v_{\psi}(s)\|_{\V}^2\d s+\frac{\beta}{2^{r-2}}\int_0^t\|\v_{\psi}(s)+\psi(s)\|_{\wi\L^{r+1}}^{r+1}\d s\nonumber\\&\leq 3\int_0^t\|\psi(s)\|_{\V}^2\d s+4\varrho\int_0^t\|{\psi}(s)\|_{\H}^2\d s+4\varrho\int_0^t\|\v_{\psi}(s)\|_{\H}^2\d s\nonumber\\&\quad+4\rho\int_0^t\|\v_{\psi}(s)\|_{\H}^2\d s+4\rho\int_0^t\|\psi(s)\|_{\H}^2\d s+\beta r^2 2^{3r-3}\int_0^t\|\psi(s)\|_{\wi\L^{r+1}}^{r+1}\d s\nonumber\\&\quad+\beta r^2 2^{3r-4}\int_0^t\|\v_{\psi_1}(s)\|_{\wi\L^{r+1}}^{r-1}\|\psi(s)\|_{\wi\L^{r+1}}^2\d s+\beta r^2 2^{3r-4}\int_0^t\|\v_{\psi_2}(s)\|_{\wi\L^{r+1}}^{r-1}\|\psi(s)\|_{\wi\L^{r+1}}^2\d s
	\end{align}
	Applying Gronwall's inequality in \eqref{a736}, we obtain 
	\begin{align}\label{a737}
	\sup_{t\in[0,T]}\|\v_{\psi}(t)\|_{\H}^2&\leq\Bigg\{3\int_0^T\|\psi(t)\|_{\V}^2\d t+4(\rho+\varrho)\int_0^T\|\psi(t)\|_{\H}^2\d t+\beta r^2 2^{3r-3}\int_0^T\|\psi(t)\|_{\wi\L^{r+1}}^{r+1}\d t\nonumber\\&\quad+\beta r^2 2^{3r-4}\left(\int_0^T\|\v_{\psi_1}(t)\|_{\wi\L^{r+1}}^{r+1}\d t\right)^{\frac{r-1}{r+1}}\left(\int_0^T\|\psi(t)\|_{\wi\L^{r+1}}^{r+1}\d t\right)^{\frac{2}{r+1}}\nonumber\\&\quad+\beta r^2 2^{3r-4}\left(\int_0^T\|\v_{\psi_2}(t)\|_{\wi\L^{r+1}}^{r+1}\d t\right)^{\frac{r-1}{r+1}}\left(\int_0^T\|\psi(t)\|_{\wi\L^{r+1}}^{r+1}\d t\right)^{\frac{2}{r+1}}\Bigg\}e^{4(\rho+\varrho)T}.
	\end{align}

	Let us now take  $\psi_n\to\psi$ in  $\mathrm{C}([0,T];\D(\A^{\alpha}))$, as $n\to\infty$. Also, from the estimate \eqref{a720}, it is clear that the quantity
	\begin{align*}
	\|\v_{\psi_n}(t)\|_{\H}^2+\int_0^t\|\nabla\v_{\psi_n}(s)\|_{\H}^2\d s+\int_0^t\|\v_{\psi_n}(s)\|_{\wi\L^{r+1}}^{r+1}\d s,
	\end{align*}
	is bounded uniformly and independent of $n$, for all $t\in[0,T]$. Thus, from \eqref{a736} and \eqref{a737}, it is immediate that $\v_{\psi_n}\to\v_{\psi}$ in $\mathrm{C}([0,T];\H)\cap\mathrm{L}^2(0,T;\V)\cap\mathrm{L}^{r+1}(0,T;\wi\L^{r+1})$ and hence the continuity of the map $\Psi$ follows. 
	
	The case of $r=3$ and $\beta\mu>1$ can be proved in a similar way. 
\end{proof}

Let   $\mathcal{G}_0\left(\int_0^{\cdot}h(s)\d s\right)$ is the set of solutions of the equation:
\begin{equation}\label{a4.22}
\left\{
\begin{aligned}
\d \A^{\alpha}\x(t)+\mu\A^{1+\alpha}\x(t)\d t&=\sqrt{\e}\A^{\alpha}\Phi(t)h(t)\d t, \ t\in (0,T),\\
\x(0)&=\mathbf{0}.
\end{aligned}
\right.
\end{equation}

\begin{theorem}\label{thm4.8}
	Let $\Theta$ maps from $\mathrm{C}([0,T];\D(\A^{\alpha}))$ to $\mathrm{C}([0,T];\H)\cap\mathrm{L}^2(0,T;\V)\cap\mathrm{L}^{r+1}(0,T;\wi\L^{r+1})$ and is given by 
	\begin{align*}
	\Theta(\z)=\z+\Psi(\z),
	\end{align*}
	where the map $\Psi(\cdot)$ is defined in \eqref{a4.130} and \eqref{a4.13}.
	For any given $R>0$ and $\delta >0$, there exists a large positive constant $\varrho_0$ such that for all $\varrho_0$, if we define the set $\mathcal{A}_{\varrho}:=\Theta(\varrho\Theta^{-1}(\mathcal{B}_R^c)),$ 
	then the unique pathwise strong solution $\u(\cdot)$ of the system \eqref{4}  satisfies:
	\begin{align}\label{a4.24}
	\mathbb{P}\left\{\u(t)\in \mathcal{A}_{\varrho} \right\}\leq \exp\left\{-\varrho^2(\mathrm{J}(\mathcal{B}_R^c)-\delta)\right\},
	\end{align}
	where  \begin{align*}\mathcal{B}_R:=\bigg\{&\v\in \mathrm{C}([0,T];\H)\cap\mathrm{L}^2(0,T;\V)\cap\mathrm{L}^{r+1}(0,T;\wi\L^{r+1}):\nonumber\\&\quad \sup_{0\leq t\leq T}\|\v(t)\|_{\H}^2+\int_0^T\|\v(t)\|_{\V}^2\d t+\int_0^T\|\v(t)\|_{\wi\L^{r+1}}^{r+1}\d t<R \bigg\},\end{align*}
	\begin{align*}
	\mathrm{J}(\mathcal{B}_R^c)=\inf_{\x\in\Theta^{-1}(\mathcal{B}_R^c)}\mathrm{I}(\x),
	\end{align*}
	and 
	\begin{align}\label{a4.26}
	\mathrm{I}(\x)=\inf_{h\in\mathrm{L}^2(0,T;\H_0):\x\in\mathcal{G}_0\left(\int_0^{\cdot}h(s)\d s\right)}\left\{\frac{1}{2}\int_0^T\|h(t)\|_0^2\d t \right\}.
	\end{align}
\end{theorem}

\begin{proof}
	For each $h\in\mathrm{L}^2(0,T;\H_0)$, we use the notation $\mathcal{G}_0\left(\int_0^{\cdot}h(s)\d s\right)$ for the set of solutions of the equation \eqref{a4.22}. For each $\e>0$, let $\z^{\e}(\cdot)$ denotes the unique pathwise strong solution of the stochastic equation: 
	\begin{equation}\label{a4.23}
	\left\{
	\begin{aligned}
	\d \A^{\alpha}\z^{\e}(t)+\mu\A^{1+\alpha}\z^{\e}(t)\d t&=\sqrt{\e}\A^{\alpha}\Phi(t)h(t)\d t, \ t\in (0,T),\\
	\z^{\e}(0)&=\mathbf{0}.
	\end{aligned}
	\right.
	\end{equation}
	Then $\A^{\alpha}\z^{\e}(t)=\sqrt{\e}\int_0^t\mathrm{S}{(t-s)}\A^{\alpha}\Phi(s)\d\W(s)=\sqrt{\e}\A^{\alpha}\z(t)$, where $\mathrm{S}(\cdot)$ is the Stokes' semigroup and $\z(\cdot)\in\C([0,T];\D(\A^{\alpha}))$ is the unique pathwise strong solution of the system \eqref{a4.1}. Note that (see section 12.3, \cite{DaZ}, \cite{SSSP}) the large deviations rate function for the family $\A^{\alpha}\z^{\e}(\cdot)$ is given by 
	\begin{align*}
	\mathrm{I}(\x)=\inf_{h\in\mathrm{L}^2(0,T;\H_0):\x\in\mathcal{G}_0\left(\int_0^{\cdot}h(s)\d s\right)}\left\{\frac{1}{2}\int_0^T\|h(t)\|_0^2\d t \right\}.
	\end{align*}
	Let us now define the map $\Theta$ from $\mathrm{C}([0,T];\D(\A^{\alpha}))$ to $\mathrm{C}([0,T];\H)\cap\mathrm{L}^2(0,T;\V)\cap\mathrm{L}^{r+1}(0,T;\wi{\L}^{r+1})$ by 
	$
	\Theta(\z)=\z+\Psi(\z),
	$
	where the map $\Psi(\cdot)$ is defined in \eqref{a4.130} and \eqref{a4.13}. Using Lemma \ref{lem2.7}, we know that  the map $\Theta$ is continuous and \begin{align}\label{a4.03}
	\u^{\e}=\Theta(\z^{\e})=\Theta(\sqrt{\e}\z),\end{align} where $\u^{\e}(\cdot)$ satisfies: 
	\begin{equation}\label{a4.27}
	\left\{
	\begin{aligned}
	\d \u^{\e}(t)+[\mu\A\u^{\e}(t)+\B(\u^{\e}(t))+\beta\mathcal{C}(\u^{\e}(t))]\d t&=\sqrt{\e}\Phi(t)\d\W(t),\ t\in(0,T),\\
	\u^{\e}(0)&=\u_0.
	\end{aligned}
	\right.
	\end{equation}
	Then, using Contraction principle (see Theorem \ref{thm4.6}), we deduce that the family $\u^{\e}(\cdot)$ satisfies the large deviation principle with the rate function: 
	\begin{align*}
	\J(\mathcal{A})=\inf_{\x\in\Theta^{-1}(\mathcal{A})}\I(\x),
	\end{align*}
	for any Borel set $\mathcal{A}\in \mathrm{C}([0,T];\H)\cap\mathrm{L}^2(0,T;\V)\cap\mathrm{L}^{r+1}(0,T;\wi\L^{r+1})$, where $$\Theta^{-1}(\mathcal{A})=\left\{\x\in\mathrm{C}([0,T];\D(\A^{\alpha})):\Theta(\x)\in\mathcal{A}\right\}.$$ Thus, using the LDP (see Definition \ref{LDP} (i)), we have  
	\begin{align*}
	\limsup_{\e\to 0}\e\log \mathbb{P}\left\{\u^{\e}\in \mathcal{B}_R^c\right\}\leq -\J(\mathcal{B}_R^c),
	\end{align*}
	where $\mathcal{B}_R$ is an open ball in $ \mathrm{C}([0,T];\H)\cap\mathrm{L}^2(0,T;\V)\cap\mathrm{L}^{r+1}(0,T;\wi\L^{r+1})$ with center $\mathbf{0}$ and radius $R>0$. Thus, for any $\delta>0$, there exists an $\varepsilon_1>0$ such that for all $0<\e<\e_1$, we have 
	\begin{align*}
	\e\log \mathbb{P}\left\{\u^{\e}\in \mathcal{B}_R^c\right\}\leq -\J(\mathcal{B}_R^c)+\delta. 
	\end{align*}
	The above inequality easily gives
	\begin{align}\label{a4.30}
	\mathbb{P}\left\{\u^{\e}\in \mathcal{B}_R^c\right\}\leq \exp\left\{-\frac{1}{\e}(\J(\mathcal{B}_R^c)-\delta)\right\}.
	\end{align}
	From \eqref{a4.30}, it is clear that 
	\begin{align}\label{a4.31}
	\mathbb{P}\left\{\z\in\frac{1}{\sqrt{\e}}\Theta^{-1}(\mathcal{B}_R^c)\right\}\leq \exp\left\{-\frac{1}{\e}(\J(\mathcal{B}_R^c)-\delta)\right\},
	\end{align}
	using \eqref{a4.03}. Let us denote the set $\mathcal{A}$ to be $\Theta(\frac{1}{\sqrt{\e}}\Theta^{-1}(\mathcal{B}_R^c))$ and from \eqref{a4.31}, we infer that 
	\begin{align}\label{a4.32}
	\mathbb{P}\left\{\z\in \mathcal{A}\right\}\leq \exp\left\{-\frac{1}{\e}(\J(\mathcal{B}_R^c)-\delta)\right\},
	\end{align}
	since $\u=\Theta(\z)$, which completes the proof. 
\end{proof}
\begin{remark}
	If we take $\varrho_0=1$, then the set $\mathcal{A}_1$ becomes $\mathcal{B}_R^c$, and from \eqref{a4.24}, we deuce that 
	\begin{align}\label{a4.36}
	\mathbb{P}\left\{\u\in \mathcal{B}_R^c \right\}\leq \exp\left\{-\varrho^2(\mathrm{J}(\mathcal{B}_R^c)-\delta)\right\},
	\end{align}
	which gives the	rate of decay as $\mathrm{J}(\mathcal{B}_R^c)$. Moreover, if one can assure the existence of an $R >0$ such that $\mathcal{B}_R^c\subseteq \mathcal{A}_{\varrho_0}$, then also the Theorem \ref{thm4.8} leads to \eqref{a4.36}. 
\end{remark}
\begin{remark}
	From the Proposition \ref{prop3.1}, we know that the rate of decay is of the order of $R^2$. We can also follow the same procedure as in the Theorem \ref{thm4.8} to get a similar result. Let us define the set 
	\begin{align*}
	\F_R:=\left\{\x:\J(\x)\leq R^2 \right\},
	\end{align*}
	for $R> 0$ and define the set $\G_R$ as any open neighborhood of $\F_R$. Then, for any given any $\delta>0$, there exists an $\e_1 > 0$ such that for all $0<\e < \e_1$, from \eqref{a4.30}, we have
	\begin{align*}
	\mathbb{P}\left\{\u^{\e}\in\G_R^c\right\}\leq \exp\left\{-\frac{1}{\e}(\J(\G_R^c)-\delta)\right\}\leq \exp\left\{-\frac{1}{\e}(R^2-\delta)\right\},
	\end{align*}
	using the definition of the set $\G_R$. Hence, using \eqref{a4.03}, it is immediate that
	\begin{align*}
	\mathbb{P}\left\{\u\in\Theta\left(\frac{1}{\sqrt{\e}}\Theta^{-1}(\G_R^c)\right)\right\}\leq \exp\left\{-\frac{1}{\e}(R^2-\delta)\right\}.
	\end{align*} 
\end{remark}

\noindent \textbf{Conflict of interest.} The corresponding author states that there is no conflict of interest. 

 \medskip\noindent
{\bf Acknowledgments:} M. T. Mohan would  like to thank the Department of Science and Technology (DST), India for Innovation in Science Pursuit for Inspired Research (INSPIRE) Faculty Award (IFA17-MA110).


\begin{thebibliography}{99}
	
	\bibitem{SNA}	S.N. Antontsev and H.B. de Oliveira, The Navier–Stokes problem modified by an absorption term, \emph{Applicable Analysis}, {\bf 89}(12),  2010, 1805--1825. 
	
	
	
	
	
	
		\bibitem{NAEG} N. Aronszajn and E. Gagliardo, Interpolation spaces and interpolation methods, \emph{Ann. Mat. Pura. Appl.}, {\bf  68} (1965), 51--118.
		
		
	
	
	\bibitem{VB} V. Barbu, {\it Analysis and control of nonlinear infinite dimensional	systems}, Academic Press, Boston, 1993.
	
	
	
	
	
	
	
	
	
	
	
	
	
	
	\bibitem{ZBGD}  Z. Brze\'zniak and Gaurav Dhariwal, Stochastic tamed Navier-Stokes equations on $\mathbb{R}^3$: the existence and the uniqueness of solutions and the existence of an invariant measure,  https://arxiv.org/pdf/1904.13295.pdf. 
	
	
	
	
	\bibitem{BD1}  A. Budhiraja  and P. Dupuis, 
	A variational representation for positive functionals of infinite
	dimensional Brownian motion, {\em Probab. and Math. Stat.}, {\bf 20}
	(2000) 39--61.
	
	\bibitem{BD2}  A. Budhiraja, P. Dupuis,  V. Maroulas, Large deviations for infinite dimensional stochastic dynamical
	systems, \emph{Ann. Probab.} {\bf 36} (2008), 1390--1420.
	
	\bibitem{DLB}	D. L. Burkholder,	The best constant in the Davis inequality for the expectation of the martingale square function, \emph{Transactions of the	American Mathematical Society} {\bf 354} (1), 91--105.
	
	
	
	
	
	
	
	
	\bibitem{chow}  P.-L. Chow, {\it Stochastic partial differential equations}, Chapman $\&$ Hall/CRC, New York, 2007.
	
	\bibitem{Chow} P.-L. Chow and  J. Menaldi,  Exponential estimates in exit probability for some diffusion processes in Hilbert spaces, \emph{Stoch. and Stoch. Rep.}, {\bf 29} (1990), 377--393.
	
	\bibitem{Chow1} P.- L. Chow, Large deviation problem for some parabolic It8 equations, \emph{Communications on Pure and Applied Mathematics}, {\bf XLV}  (1992) , 97--120.
	
	
	
	
	
	
	\bibitem {ICAM} I. Chueshov and A. Millet,	Stochastic 2D hydrodynamical type systems: Well posedness and Large Deviations, \emph{Applied Mathematics and  Optimization}, {\bf 61} (2010), 379--420.
	
	
	
	
	
	
	
	\bibitem{PGC} 	P. G. Ciarlet, \emph{Linear and Nonlinear	Functional Analysis	with Applications}, SIAM Philadelphia, 2013.
	
	
		\bibitem{RFC} R. F. Curtain and A. J. Pritchard, \emph{	Functional Analysis in Modern Applied Mathematics}, 
	Mathematics in Science and Engineering, Vol. 132. Academic Press, London-New York, 1977.
	
	\bibitem{DaZ}
	\newblock G. Da Prato and J. Zabczyk,
	\newblock \emph{Stochastic Equations in Infinite Dimensions},
	\newblock Cambridge University Press, 1992.
	
	
	\bibitem {BD}	B. Davis, On the integrability of the martingale square function, \emph{Israel Journal of Mathematics} {\bf 8}(2) (1970), 187--190.
	
	
	
	
	
	
	\bibitem{DZ} A. Dembo,  and  O. Zeitouni, 
	{\em Large Deviations Techniques and Applications}, Springer-Verlag,
	New York, 2000.
	
	
	\bibitem{ZDRZ} Z.  Dong and R.  Zhang,	3D tamed Navier-Stokes equations driven by multiplicative L\'evy noise: Existence, uniqueness and large deviations, https://arxiv.org/pdf/1810.08868.pdf
	
	
	
	
	
	
	
	
	
	
	
	\bibitem{Evans} 
	\newblock  L. C. Evans, 
	\newblock  \emph{Partial Differential Equations}, 
	\newblock  Graduate studies in Mathematics, American Mathematical Society, 2nd Ed, 2010. 
	
	
	\bibitem{CLF} 	C. L. Fefferman, K. W. Hajduk and J. C. Robinson,	\emph{Simultaneous approximation in Lebesgue and Sobolev norms via eigenspaces}, https://arxiv.org/abs/1904.03337.
	
	\bibitem{FW} M. I.  Freidlin, and  A. D. Wentzell, 
	{\em Random Pertubations of Dynamical Systems}, Springer-Verlag, New
	York, 1984.
	
		\bibitem{DFHM} D. Fujiwara, H. Morimoto, An $L^r$-theorem of the Helmholtz decomposition of vector fields, \emph{J. Fac. Sci. Univ. Tokyo Sect. IA Math.,} {\bf 24} (1977),	685--700.

	\bibitem{HGHL} H. Gao and H. Liu,	Well-posedness and invariant measures for a class of stochastic 3D Navier-Stokes equations with damping driven by jump noise, \emph{Journal of Differential Equations}, {\bf 267} (2019), 5938--5975. 
	
	
	
	
	
	
	
	
	\bibitem{KWH}	K. W. Hajduk and J. C. Robinson, Energy equality for the 3D critical convective Brinkman-Forchheimer equations, \emph{Journal of Differential Equations}, {\bf 263} (2017), 7141--7161.
	
	
	
	
	\bibitem{Hsu} Po-Han Hsu and P. Sundar, Exponential inequalities for exit times for	stochastic Navier-Stokes equations and a class of evolutions, \emph{Communication on Stochastic Analysis}, {\bf 13} (3)  (2018), 343--358.
	
	
	
	
	
	
	
	
	\bibitem{AAI} A. A.  Ilyin, On the spectrum of the Stokes operator, \emph{Functional Analysis and Its Applications}, {\bf 43} (4) (2009), 254--263.
	
	
	
	
	
	
	
	
	\bibitem{KT2}  V. K. Kalantarov and S. Zelik, Smooth attractors for the Brinkman-Forchheimer equations with fast growing nonlinearities, \emph{Commun. Pure Appl. Anal.}, {\bf 11}	(2012) 2037--2054.
	
		\bibitem{KX} G.  Kallianpur, and  J. Xiong, 	{\em Stochastic Differential Equations in Infinite Dimensional		Spaces}, Institute of Math. Stat, 1996.
	
	
	
	
	
	\bibitem{LHGH}	H. Liu and H. Gao, 	Ergodicity and dynamics for the stochastic 3D Navier-Stokes equations with damping, \emph{Commun. Math. Sci.}, {\bf 16}(1) (2018), 97--122.
	
	\bibitem{LHGH1}	H. Liu and H. Gao, Stochastic 3D Navier–Stokes equations with nonlinear damping: martingale solution, strong solution and small time LDP, Chapter 2 in \emph{Interdisciplinary Mathematical SciencesStochastic PDEs and Modelling of Multiscale Complex System}, 9--36, 2019.
	
	\bibitem{LHGH2}	H. Liu,  L. Lin, C. Sun and Q. Xiao, The exponential behavior and stabilizability of the stochastic 3D Navier-Stokes equations with damping, \emph{Reviews in Mathematical Physics},
	{\bf 31} (7) (2019), 1950023. 
	
	
	\bibitem{WLMR} W. Liu and M. R\"ockner,	Local and global well-posedness of SPDE with generalized coercivity conditions, \emph{Journal of Differential Equations}, {\bf 254} (2013), 725--755. 
	
	\bibitem{WL}  W. Liu, Well-posedness of stochastic partial differential equations with Lyapunov condition, \emph{Journal of Differential Equations}, {\bf 255} (2013), 572--592. 
	
	
	
	
	
	
	
	
	
	
	
	
	
	
	
	\bibitem{MTM6} M. T. Mohan, Well posedness, large deviations and ergodicity of the stochastic 2D Oldroyd model of order one, \emph{Stochastic Processes and their Applications}, {\bf 130}(8) (2020), 4513-4562. 
	
	
	
	
	
	
	\bibitem{MTM7} M. T. Mohan, On the convective Brinkman-Forchheimer equations, \emph{Submitted}. 
	
	\bibitem{MTM8} M. T. Mohan, Stochastic convective Brinkman-Forchheimer equations, \emph{Submitted}, \url{https://arxiv.org/abs/2007.09376}. 
	
	\bibitem{MTM10} M. T. Mohan, Well-posedness and asymptotic behavior of the stochastic convective Brinkman-Forchheimer equations perturbed by pure jump noise, \emph{Submitted}, \url{https://arxiv.org/abs/2008.08577}. 
	
	
	
	
	
	
	
	\bibitem{N10} J. M. A. M. van Neerven,	$\gamma$-radonifying operators: A survey,	\emph{Proc. Centre Math. Appl. Austral. Nat. Univ.,}  \textbf{44} (2010), 1-61.

	\bibitem{DRM} D. Revuz,  M. Yor,  \emph{Continuous martingales and Brownian motion}, Third ed.,  Springer, Berlin, 1999.
	
	
	
	
	\bibitem{JCR4}	J.C. Robinson,  J.L. Rodrigo,  W. Sadowski, \emph{The three-dimensional Navier–Stokes equations, classical theory}, Cambridge Studies in Advanced Mathematics, Cambridge
	University Press, Cambridge, UK, 2016. 
	
	
	
	

	\bibitem{MRBS} 	M. R\"ockner, B. Schmuland and X. Zhang, Yamada-Watanabe theorem for stochastic evolution equations in infinite dimensions, \emph{Condensed Matter Physics}, {\bf 11}(2) (2008), 247--259.
	
	\bibitem{MRXZ1}	M. R\"ockner and X. Zhang, Stochastic tamed 3D Navier-Stokes equation: existence, uniqueness and ergodicity, \emph{Probability Theory and Related Fields}, {\bf 145} (2009) 211--267.
	
	\bibitem{MRTZ1}	M. R\"ockner, T. Zhang and X. Zhang,	Large deviations for stochastic tamed 3D Navier-Stokes equations, \emph{Applied Mathematics and Optimization}, {\bf 61} (2010), 267--285. 
	
	\bibitem{MRTZ}	 M. R\"ockner and T. Zhang, Stochastic 3D tamed Navier-Stokes equations: Existence, uniqueness and small time large deviations principles, \emph{Journal of Differential Equations}, {\bf 252} (2012), 716--744.
	
	\bibitem{RWW} M. R\"ockner, F.-Y. Wang and L. Wu,	Large deviations for stochastic generalized porous media equations, \emph{Stochastic Processes and their Applications},
	{\bf  116}(12) (2006), 1677--1689. 
	
	\bibitem{AVS} A. V. Skorokhod, Limit theorems for stochastic processes, \emph{Theory of Probability $\&$ Its Applications}, {\bf 1}(3), (1956), 261--290.
	
	
	
	
	\bibitem{SSSP} S. S. Sritharan and P. Sundar,	{Large deviations for the two-dimensional Navier-Stokes		equations with multiplicative noise}, \emph{Stochastic Processes and their		Applications}, {\bf 116} (2006), 1636--1659.
	
	
	
	
	
	
	
	\bibitem{Te1} R. Temam, 	\emph{Navier-Stokes Equations and Nonlinear Functional Analysis}, Second Edition, CBMS-NSF Regional Conference Series in Applied Mathematics, 1995.
	
	\bibitem{Va} S. R. S. Varadhan, 
	{\em Large deviations and Applications}, {\bf 46}, CBMS-NSF Series
	in Applied Mathematics, SIAM, Philadelphia, 1984.
	
	\bibitem{VF}   M. I. Visik, and A.V. Fursikov, \emph{Mathematical Problems of Statistical Hydromechanics}, Kluwer, Dordrecht, 1980.
	
	
	
	
	\bibitem{DYJD} D. Yang and J. Duan, Large deviations for the stochastic
	quasigeostrophic equation with	multiplicative noise, \emph{Journal of Mathematical Physics}, {\bf 51} (2010), 053301.
	
	\bibitem{BYo}	B. You,	The existence of a random attractor for the three dimensional damped Navier-Stokes equations with additive noise, \emph{Stochastic Analysis and Applications},  {\bf 35}(4) (2017), 691--700. 
	
	
	
	
	
	
	
	
	
	
	
\end{thebibliography}
\end{document}